\documentclass[11pt]{article}

\usepackage[margin=1in]{geometry}

\usepackage{graphicx}
\usepackage{amsfonts}
\usepackage{mathrsfs}
\usepackage[percent]{overpic}
\usepackage{tikz}
\usepackage{calc}
\usetikzlibrary{shapes.geometric, arrows, positioning, decorations.pathreplacing}
\usepackage{url}
\usepackage{subfigure}
\usepackage{appendix}
\usepackage[normalem]{ulem}
\usepackage{bm}
\usepackage{amssymb, amsmath}
\usepackage{parskip}

\usepackage{mathrsfs}
\usepackage{soul}

\usepackage{lipsum}
\usepackage{epstopdf}
\usepackage{algorithm}
\usepackage{algpseudocode}

\usepackage{amsmath,amssymb,amsfonts}
\usepackage{mathtools}
\usepackage[colorlinks=true,
            linkcolor=blue,
            citecolor=blue,
            urlcolor=blue]{hyperref}

\usepackage{amsthm}
\newtheorem{theorem}{Theorem}[section]

\theoremstyle{definition}

\theoremstyle{remark}
\newtheorem{remark}[theorem]{Remark}
\theoremstyle{hypothesis}

\theoremstyle{claim}

\theoremstyle{example}
\newtheorem{example}[theorem]{Example}

\usepackage{graphicx}
\usepackage{tikz} 
\usetikzlibrary{arrows.meta,shapes.geometric}

\usepackage{enumitem}

\usepackage{xcolor}

\usepackage[colorlinks=true,
            linkcolor=blue,
            citecolor=blue,
            urlcolor=blue]{hyperref}

\tikzset{ 
    tree_edge/.style={
        thick, -
    },
    edge from parent/.style={
        draw, thick, -
    } 
}

\tikzstyle{level 1}=[level distance=12.5mm,sibling distance=3.35cm] 
\tikzstyle{level 2}=[level distance=12.5mm,sibling distance=2cm]
\tikzstyle{level 3}=[level distance=12.5mm,sibling distance=1cm]

\ifpdf
  \DeclareGraphicsExtensions{.eps,.pdf,.png,.jpg}
\else
  \DeclareGraphicsExtensions{.eps}
\fi

\usepackage{cleveref}

\usepackage[numbers,sort&compress]{natbib}

\title{A Mass, Momentum, and Energy Conserving Semi-Lagrangian Adaptive-Rank (SLAR) Method for the Vlasov-Poisson System}

\author{Nanyi Zheng\thanks{Department of Mathematical Sciences, University of Delaware, Newark, DE, 19716, USA. nyzheng@udel.edu.}
\and William A. Sands\thanks{Theoretical Division, Los Alamos National Laboratory, Los Alamos, NM, 87545, USA. wsands@lanl.gov.}
\and Jing-Mei Qiu\thanks{Corresponding author. Department of Mathematical Sciences, University of Delaware, Newark, DE, 19716, USA.
jingqiu@udel.edu.}}

\begin{document}

\maketitle

\begin{abstract}
We propose a semi-Lagrangian adaptive-rank (SLAR) method that combines the large time-step capability of semi-Lagrangian schemes with the efficiency of adaptive-rank tensor representations while simultaneously enforcing local conservation laws for mass, momentum, and energy. The method builds on the high-dimensional SLAR framework introduced in our previous work \cite{zheng2025semihighD} and achieves high-order accuracy in both space and time. To address the loss of conservation in long-time simulations, we extend the implicit local macroscopic conservative (LoMaC) correction technique for the BGK equation \cite{sands2025adaptive} to the high-dimensional Vlasov--Poisson (VP) system. The implicit macroscopic system is discretized using backward differentiation formulas and solved with a Jacobian-free Newton-Krylov method. This approach enables a consistent coupling with semi-Lagrangian methods which are capable of taking large time steps. A novel component of the proposed method is a unified adaptive-weight projection technique that eliminates the ad hoc parameter tuning required by previous LoMaC approaches. These weights capture problem-dependent velocity space structures and are constructed from the low-rank velocity bases of the solution. The local semi-Lagrangian method used in this work reconstructs the solution at the feet of the characteristics using efficient tensor contractions. To the best of our knowledge, this is the first successful implementation of an implicit LoMaC method for the VP system up to the 2D--2V setting. Numerical experiments on several classical benchmark problems demonstrate the accuracy and efficiency of the proposed method, as well as its ability to preserve conservation laws in VP simulations.
\end{abstract}


\section{Introduction}

A wide range of numerical methods have been developed for the nonlinear Vlasov equation and, more generally, for multiscale kinetic models of plasma dynamics. Two classes of methods that are particularly relevant to this work are semi-Lagrangian (SL) schemes and low-rank or adaptive-rank tensor methods, which target complementary aspects of efficiency. SL schemes are particularly attractive because they alleviate the restrictive time step constraints imposed by Eulerian CFL conditions. However, their direct application to problems posed in high-dimensional phase space is limited by the curse of dimensionality, which motivates their combination with low-rank or adaptive-rank tensor methods. For broader reviews of SL schemes and low-rank methods, we refer to \cite{falcone2013semi, cho2024conservative} and \cite{einkemmer2025review}, respectively. Low-rank methods compress discrete representations of high-dimensional functions using structured tensor formats, such as tensor trains (TT), Tucker tensors, and
hierarchical Tucker (HT) tensors. In this work, we focus on the HT format as a continuation of the semi-Lagrangian adaptive-rank (SLAR) framework developed in our previous work \cite{zheng2025semihighD}. However, the ideas developed here can be extended to other low-rank tensor formats.

In the SL formulation considered here, the solution at the Eulerian degrees of freedom is updated by interpolating or projecting data at the feet of the characteristics, which are traced backward in time. Such methods have been developed in finite-difference (FD) \cite{sonnendrucker1999semi,qiu2010conservative},
finite-volume (FV) \cite{filbet2001conservative,crouseilles2010conservative,
zheng2022fourth}, and discontinuous Galerkin (DG) \cite{rossmanith2011positivity,cai2018high} settings for the Vlasov--Poisson
(VP) system. SL--FD schemes with dimensional splitting are attractive for their simplicity and efficiency, and high-order conservative formulations are available in that setting. Without dimensional splitting, however, conservative SL--FD schemes are more difficult to construct. SL--FV and SL--DG methods instead evolve cell averages and higher moments, offering a more natural route to conservation, although often at the expense of remapping solutions onto curved upstream cells with careful geometric reconstruction and clipping \cite{cai2018high, zheng2022fourth}. For nonlinear kinetic models such as the VP system, the characteristics depend on self-consistent fields generated by the solution itself, so accurate tracing typically requires predictor-corrector or exponential-integrator strategies \cite{qiu2017high,cai2021high}. Diffusion and source terms can also be incorporated naturally along characteristics \cite{li2023high,ding2020semi}.

Low-rank and adaptive-rank techniques have recently emerged as effective tools for reducing computational complexity and alleviating the curse of dimensionality in the VP system and other kinetic simulations \cite{einkemmer2025review}. Representative approaches for the VP system include SL low-rank schemes \cite{kormann2015semi}, dynamical low-rank methods (DLR) \cite{einkemmer2018low}, Eulerian step-and-truncate (SAT) formulations \cite{guo2022low}, quantized tensor-train methods \cite{ye2024quantized}, and adaptive-rank algorithms which rely on greedy sampling \cite{zheng2025semi,zheng2025semihighD}. While these methods differ in their choices for the tensor format and discretization strategy, they avoid forming the full high-dimensional tensor on the associated phase-space. In order to remain tractable, these methods require the use of tensor truncation when performing initialization or basic operations, which degrades the macroscopic conservation properties of the underlying kinetic equation.

Conservative projection techniques have been proposed to improve the physical fidelity of low-rank kinetic simulations of the VP system. For DLR methods, the first approach that conserves mass, momentum, and energy  was proposed in \cite{einkemmer2021mass}, where the standard DLR ansatz is augmented so that the velocity space contains the moment functions needed to enforce conservation in a weighted \(L^2\) sense. Related moment-preserving projection ideas were introduced in \cite{guo2024conservative} and further developed in \cite{guo2024local,guo2024localDG} for Eulerian SAT methods, leading to the locally macroscopic conservative (LoMaC) framework. The basic idea of LoMaC is to evolve the microscopic kinetic equation together with a corresponding macroscopic moment system, using the kinetic solution to provide closure for the moment equations. Since the first $d_{v}+2$ moments determine the macroscopic density, momentum density, and kinetic-energy density, a conservative solver can be applied to the moment system. The resulting macroscopic quantities are then used to construct a projection that corrects the non-conservative adaptive-rank kinetic solution. This coupling is similar in spirit to high-order/low-order (HOLO) methods \cite{chacon2017multiscale}, but the role of the moment system is different. In HOLO methods, the low-order moment system is used within an iterative nonlinear coupling strategy to accelerate or stabilize the solution of the high-order kinetic problem. This requires repeated updates to both systems until they are self-consistent. In the LoMaC framework considered here, the kinetic equation is advanced only once to obtain a provisional solution, and the moment system is then used to define a conservative correction. 

While this strategy has been successfully incorporated into Eulerian SAT schemes as an explicit post-processing step, its extension to SLAR methods presents two main challenges. First, since SL methods are not restricted by the usual CFL condition, consistent coupling with the macroscopic moment system requires an implicit treatment. Second, the conservative projection depends on the construction of a suitable velocity space ansatz used to encode the corrected moments. In recent work \cite{sands2025adaptive}, we proposed an implicit LoMaC technique for the BGK equation using a local Maxwellian ansatz. This choice is natural for collisional problems which contain a relaxation mechanism that drives the solution toward this equilibrium. The local Maxwellian also encodes the macroscopic moments of the distribution and prescribes the appropriate decay in velocity space. However, the collisionless VP system does not contain an analogous relaxation mechanism, so the distribution may not be well described by this locally Maxwellian ansatz. Moreover, existing explicit LoMaC corrections for Eulerian low-rank solvers \cite{guo2024local,guo2024localDG} rely on problem-dependent, empirically tuned velocity space weights.

To address these challenges, we propose an SLAR method with an implicit LoMaC correction that uses an adaptive weight function. Our approach builds on the hierarchical Tucker adaptive cross approximation (HTACA) technique recently introduced in \cite{zheng2025semihighD}. HTACA constructs a low-rank approximation in the HT format by sampling selected entries of the tensor, typically along fibers, which eliminates the need to form the full high-dimensional tensor. In the present work, these sampled entries are evaluated using a new local SL--FD solver that exploits the tensor-product structure of the problem and requires \(\mathcal{O}(d r^3)\) operations per grid point, where \(r\) denotes the maximum rank of the solution at the previous time step. In comparison, the previous SL--FD scheme in \cite{zheng2025semihighD} has a per-entry evaluation cost of \(\mathcal{O}(d^3 r^3)\). Using von Neumann stability analysis, we show that the proposed SL--FD scheme is unconditionally stable when applied to the \(d\)-dimensional constant-coefficient linear advection equation.

The central contribution of this work is a new implicit LoMaC correction that enforces local conservation of mass, momentum, and total energy. A key innovation of the proposed projection is its adaptive weight function, which can be computed efficiently from the velocity bases available in the low-rank representation, eliminating the ad hoc parameter tuning required by earlier approaches \cite{guo2024local,guo2024localDG}. Such adaptive projection reflects problem-dependent velocity space structures. The macroscopic moment system, whose solution provides the moments used to construct the projection, is discretized using backward differentiation formulas (BDF), and the resulting nonlinear systems are solved by a Jacobian-free Newton--Krylov (JFNK) method, following our previous work \cite{sands2025adaptive}. In contrast to explicit LoMaC algorithms previously developed for the VP system \cite{guo2024local,guo2024localDG}, the proposed implicit correction is compatible with the large time steps permitted by SL discretizations. The proposed correction substantially improves the poor long-time conservation behavior observed in the SLAR method of \cite{zheng2025semihighD}.

We briefly summarize the notation used throughout the paper.
Scalar functions are denoted by lowercase letters, e.g., $u(x)$,
while their discrete values are written as $u_i$ or $u_{i_1,\dots,i_d}$.
Vector-valued functions are denoted by bold letters, e.g.,
$\bm u(x)$ or $\bm U(x)$. Discrete vectors are denoted by bold lowercase letters, e.g.,
$\mathbf u \in \mathbb{R}^N$, whereas discrete matrices and higher-order
tensors are denoted by bold uppercase letters, for example
$\mathbf A \in \mathbb{R}^{N\times M}$ and
$\mathbf X \in \mathbb{R}^{N_1\times \cdots \times N_d}$.
Calligraphic uppercase
letters, such as $\mathcal{F}$, are used to denote
discrete tensors in compressed form, such as the HT format.
The scalar entries of a bold tensor are denoted using the multi-index notation, e.g., $X_{i_1,\dots,i_d}$. 


The remainder of the paper is organized as follows. In \Cref{sec:model}, we introduce the kinetic VP system and its associated macroscopic system. In \Cref{sec:scheme}, we present the high-dimensional SLAR scheme with the proposed implicit LoMaC correction. Numerical results are reported in \Cref{sec:numerical_tests}, and the conclusion is given in \Cref{sec:conclusion}.

\section{Model Equations}\label{sec:model}

In this section, we summarize the kinetic model and its associated macroscopic conservation laws. A delicate aspect of the VP system is the treatment of the energy equation. In particular, the moment equation obtained from the kinetic energy alone does not yield a locally conservative balance law for the total energy, since it must be coupled with the field energy associated with the self-consistent electric field. We therefore briefly outline the derivation of the macroscopic system and emphasize the form of the total energy density and flux. This structure guides the construction of the discretization developed in \Cref{sec:scheme}.

\subsection{Kinetic Model}

We consider a collisionless electrostatic plasma model consisting of ions and electrons. 
To simplify the model, we assume that the heavier ions remain stationary and provide a spatially uniform neutralizing background charge density. 
Under this assumption, the electron dynamics are governed by the VP system:
\begin{align}
    &\partial_t f_{\mathrm{e}} 
    + \bm{v}\cdot\nabla_{\bm{x}} f_{\mathrm{e}} 
    + \frac{q_{\mathrm{e}}}{m_{\mathrm{e}}} \bm{E}\cdot\nabla_{\bm{v}} f_{\mathrm{e}} = 0,
    \quad (\bm{x},\bm{v},t)\in \Omega_{\bm{x}}\times\Omega_{\bm{v}}\times(0,T], \label{eq:Vlasov}\\
    &-\Delta_{\bm{x}} \phi = \rho/\epsilon_{0}, 
    \quad (\bm{x},t)\in \Omega_{\bm{x}}\times(0,T], \label{eq:Poisson}\\
    &\bm{E} = -\nabla_{\bm{x}} \phi,
    \quad (\bm{x},t)\in \Omega_{\bm{x}}\times(0,T]. \label{eq:Efromphi}
\end{align}
Here, $f_{\mathrm{e}} = f_{\mathrm{e}}(\bm{x},\bm{v},t)$ denotes the electron distribution function in phase space, which represents the probability of finding an electron at position $\bm{x} \in \Omega_{\bm{x}} \subset \mathbb{R}^{d_x}$ with velocity $\bm{v} \in \Omega_{\bm{v}} = \times_{\mu = 1}^{d_v}\left(\Omega_{v^{(\mu)}}\right)\subset \mathbb{R}^{d_v}$ at time $t$. 
The constants $q_{\mathrm{e}}$ and $m_{\mathrm{e}}$ denote the electron charge and mass, respectively, and $\epsilon_0$ is the vacuum permittivity. The self-consistent electrostatic potential $\phi = \phi(\bm{x},t)$ determines the electric field through $\bm{E} = -\nabla_{\bm{x}}\phi$, and $\rho(\bm{x},t)$ denotes the total charge density. In this work, periodic boundary conditions are prescribed for the system \eqref{eq:Vlasov}--\eqref{eq:Efromphi} on the phase space \(\Omega_{\bm{x}}\times\Omega_{\bm{v}}\). Moreover, we rescale the VP system \eqref{eq:Vlasov}--\eqref{eq:Poisson} so that 
\[
m_{\mathrm e}=1, \quad q_{\mathrm i}=1, \quad q_{\mathrm e}=-1, \quad \epsilon_{0} = 1.
\]

Next, we introduce notation for the velocity moments of the electron distribution function. For convenience, we denote integration over the velocity domain by
\[
\left\langle \cdot \right\rangle_{\bm v}
:=
\int_{\Omega_{\bm v}} (\cdot),d\bm v .
\]
We define the raw moments
\begin{equation*}
\begin{aligned}
    \bm M(f_{\mathrm e})
    &:=
    \Big(
    M_0(f_{\mathrm e}),
    M_{1,1}(f_{\mathrm e}),
    \ldots,
    M_{1,d_{v}}(f_{\mathrm e}),
    M_2(f_{\mathrm e})
    \Big), \\
    &:=
    \Big(
    M_0(f_{\mathrm e}),
    \bm M_{1}(f_{\mathrm e}),
    M_2(f_{\mathrm e})
    \Big),
\end{aligned}
\end{equation*}
where
\begin{equation}
\label{eq:raw moments}
    M_0(f_{\mathrm e})
    :=
    \left\langle f_{\mathrm e}\right\rangle_{\bm v},
    \qquad
    M_{1,\mu}(f_{\mathrm e})
    :=
    \langle v^{(\mu)} f_{\mathrm e}\rangle_{\bm v},
    \qquad
    M_2(f_{\mathrm e})
    :=
    \langle |\bm v|^2 f_{\mathrm e}\rangle_{\bm v}.
\end{equation}
Note that the electron number density, momentum density, and kinetic energy density can be defined in terms of the moments \eqref{eq:raw moments} by
\begin{equation}
\label{eq:physical_moments}
n_{\mathrm e}
=
M_0(f_{\mathrm e}),
\qquad
m_{\mathrm e}n_{\mathrm e}\bm u_{\mathrm e}
=
m_{\mathrm e}\bm M_1(f_{\mathrm e}),
\qquad
e_{k,\mathrm e}
=
\frac12 m_{\mathrm e}M_2(f_{\mathrm e}).
\end{equation}
The charge and current densities are electromagnetic quantities which can be derived from the moments \eqref{eq:physical_moments}. In particular, the total charge density is
\begin{equation}
\label{eq:charge_density}
\rho(\bm x,t)
=
\rho_{\mathrm i}+\rho_{\mathrm e}(\bm x,t)
=
q_{\mathrm i}n_{\mathrm i}
+
q_{\mathrm e}n_{\mathrm e}(\bm x,t)
=
q_{\mathrm i}n_{\mathrm i}
+
q_{\mathrm e}M_{0}(f_{\mathrm e})(\bm x,t),
\end{equation}
where \(n_{\mathrm i}\) is constant under the stationary-ion assumption. Similarly, the total current density is
\begin{equation}
\label{eq:current_density}
\bm J(\bm x,t)
=
\bm J_{\mathrm i}+\bm J_{\mathrm e}(\bm x,t)
=
\bm J_{\mathrm e}(\bm x,t)
=
q_{\mathrm e}n_{\mathrm e}\bm u_{\mathrm e}(\bm x,t)
=
q_{\mathrm e}\bm M_1(f_{\mathrm e}).
\end{equation}
The electron distribution function \(f_{\mathrm e}\) determines the deviation from charge neutrality through \(M_0(f_{\mathrm e})\), which induces the electrostatic potential \(\phi\). The resulting electric field accelerates the electrons through the force \(q_{\mathrm e}\bm E\).

\subsection{Macroscopic System}

The macroscopic system can be derived in the usual way by taking moments of the
kinetic equation \eqref{eq:Vlasov} with respect to the test functions
\(\Phi(\bm v)\in\{1,m_{\mathrm e} \bm v,\frac12 m_{\mathrm e}|\bm v|^2\}\). This produces a
system of balance laws corresponding to particle number density, momentum density, and kinetic
energy density, respectively. We note that the kinetic energy density is not conserved
by itself, as it can be exchanged with the field. Under periodic boundary conditions, or, more generally, under boundary
conditions for which the corresponding surface integrals vanish, it follows
that the spatial integral of the total energy density is conserved. We define this total energy density by
\begin{equation*}
    e := e_{k,\mathrm{e}} + e_{p},
\end{equation*}
where \(e_{k,\mathrm{e}}\) is the electron kinetic energy density and \(e_p\) is
the electric field energy density, given by \(e_{p} = \frac{\epsilon_{0}}{2} |\bm{E}|^{2}\).

As the kinetic energy density is not conserved by itself, we instead use a
conservation law for the total energy density. In particular, it can be shown
that the time derivative of the potential energy density satisfies
\begin{equation}\label{eq:electric_filed_energy_equation}
    \partial_t e_p
    =
    \epsilon_0\,\nabla_{\bm{x}}\!\cdot\!\big(\phi\,\nabla_{\bm{x}}(\partial_t\phi)\big)
    -
    \nabla_{\bm{x}}\!\cdot(\bm{J}\phi)
    -
    \bm{J}\cdot\bm{E},
\end{equation}
where $\bm J$ is defined in \eqref{eq:current_density}. When this balance is combined with the kinetic energy equation, we obtain the
following conservation law for the total energy:
\begin{equation}\label{eq:total_energy_density_equation}
    \partial_t e
    +
    \nabla_{\bm{x}}\cdot
    \left(
        \frac{1}{2}
        \left\langle
            m_{\mathrm e}|\bm{v}|^2\bm{v} f_{\mathrm e}
        \right\rangle_{\bm{v}}
        -
        \epsilon_0\phi\nabla_{\bm{x}}(\partial_t\phi)
        +
        \bm{J}\phi
    \right)
    =
    0.
\end{equation}
Here, the quantity \(\partial_t\phi\) solves the Poisson equation
\begin{equation*}
    -\epsilon_0\Delta_{\bm{x}}(\partial_t\phi)
    =
    -\nabla_{\bm{x}}\cdot\bm{J}.
\end{equation*}
The macroscopic system can be written in the conservative form
\begin{equation}
    \label{eq:macro_sys}
    \partial_t \bm U
    +
    \nabla_{\bm x}\cdot \bm F(\bm U)
    =
    \bm S,
\end{equation}
where
\renewcommand{\arraystretch}{1.25}
\begin{equation}
    \label{eq:macro_sys_comp}
    \bm U :=
    \begin{pmatrix}
        n_{\mathrm e} \\
        m_{\mathrm e} n_{\mathrm e}\bm u_{\mathrm e} \\
        e
    \end{pmatrix},
    \quad
    \bm F(\bm U) :=
    \begin{pmatrix}
        \left\langle \bm v f_{\mathrm e} \right\rangle_{\bm v} \\
        \left\langle m_{\mathrm e}(\bm v \otimes \bm v) f_{\mathrm e} \right\rangle_{\bm v} \\
        \dfrac{1}{2}
        \left\langle
            m_{\mathrm e}|\bm v|^{2}\bm v f_{\mathrm e}
        \right\rangle_{\bm v}
        + \bm F_{e_p}
    \end{pmatrix},
    \quad
    \bm S :=
    \begin{pmatrix}
        0 \\
        \rho_{\mathrm e}\bm E \\
        0
    \end{pmatrix}.
\end{equation}
Here the electric field energy flux is
\begin{equation}\label{eq:electric_field_flux}
    \bm F_{e_p}
    :=
    -\epsilon_{0}\phi \nabla_{\bm x}(\partial_t\phi)
    +
    \bm J \phi.
\end{equation}

We refer the reader to the Supplementary Material for details regarding the derivation of \eqref{eq:electric_filed_energy_equation}--\eqref{eq:electric_field_flux}. The moment system \eqref{eq:macro_sys}--\eqref{eq:macro_sys_comp} identifies the conserved quantities of the VP system, but it is not closed. Rather than evolving this system independently, we use the low-rank kinetic scheme described in \Cref{sec:LoMaC} to supply a dynamical closure. The resulting macroscopic update provides the physically consistent conserved quantities, which are then imposed on the low-rank approximation.

In later sections, we will make use of a connection between the conserved variables $\bm U$ and the raw moments $\bm M$. Specifically, the raw moments can be recovered from the conserved quantities using the relations
\begin{equation}
\label{eq:M from U}
    M_{0} = U_{0}, \qquad M_{1,\mu} = \frac{1}{m_{e}} U_{1,\mu}, \qquad M_{2} = \frac{2}{m_{e}}\left(U_{2} - e_{p}\right).
\end{equation}
Note that the last relation assumes that the electric field energy density \(e_{p}\) is computed from the solution of Poisson's equation, with charge density \eqref{eq:charge_density} determined by \(U_{0}\).

\section{Numerical Method}\label{sec:scheme}

In this section, we first introduce a SLAR scheme with high-dimensional tensor interpolations for the linear advection equation and for the VP system in 
\Cref{sec:SL_FD_scheme}. Then we present the adaptive-weight LoMaC projection and the implicit solver for the coupling of the kinetic SLAR scheme with the
macroscopic solver in \Cref{sec:LoMaC}. 

\subsection{SLAR Discretization}
\label{sec:SL_FD_scheme}

Consider the linear advection equation
\begin{equation}\label{eq:advection_ddim}
    \partial_t f(\bm{x},t) 
    + \bm{a}(\bm{x},t)\cdot\nabla_{\bm{x}} f(\bm{x},t) 
    = 0,
\end{equation}
where $\bm{x} = (x^{(1)},\dots,x^{(d)}) \in \mathbb{R}^d$ and 
$\bm{a}(\bm{x},t)$ denotes a prescribed velocity field. 
It is well known that the solution of \eqref{eq:advection_ddim} remains invariant along characteristic trajectories, which forms the foundation of the SL approach.

We discretize the computational domain by a tensor product of uniform meshes in each coordinate direction. 
For $\mu = 1,\dots,d$, the grids are defined by
\begin{equation*}
    a^{(\mu)} 
    = x^{(\mu)}_{\frac{1}{2}} 
    < x^{(\mu)}_{\frac{3}{2}} 
    < \cdots 
    < x^{(\mu)}_{N_{\mu}+\frac{1}{2}} 
    = b^{(\mu)}.
\end{equation*}
The corresponding cells, cell centers, and mesh spacing are given by
\begin{equation*}
    I^{(\mu)}_{i_{\mu}} 
    = \bigl[x^{(\mu)}_{i_{\mu}-\frac{1}{2}},\, 
            x^{(\mu)}_{i_{\mu}+\frac{1}{2}}\bigr], 
    \qquad 
    x^{(\mu)}_{i_{\mu}} 
    = \frac{1}{2}
      \bigl(
      x^{(\mu)}_{i_{\mu}-\frac{1}{2}}
      + x^{(\mu)}_{i_{\mu}+\frac{1}{2}}
      \bigr),\qquad \Delta x^{(\mu)} 
    = x^{(\mu)}_{i_{\mu}+\frac{1}{2}}
      - x^{(\mu)}_{i_{\mu}-\frac{1}{2}}.
\end{equation*}
Using multi-index notation, we write a generic cell as
\(
I_{i_1,\dots,i_d}
=
I^{(1)}_{i_1}
\times
I^{(2)}_{i_2}
\times
\cdots
\times
I^{(d)}_{i_d}\),
and its corresponding cell center as
\(\mathbf{x}_{i_1,\dots,i_d}
=
\bigl(
x^{(1)}_{i_1},\dots,
x^{(d)}_{i_d}
\bigr).
\)

A backward SL--FD update can be obtained by tracing characteristics backward in time using a high-order Runge-Kutta (RK) integrator. Then, the value of $f$ at a grid point 
$\mathbf{x}_{i_1,\dots,i_d}$ at time $t^{n+1}$ is updated according to
\begin{equation}
    f(\mathbf{x}_{i_1,\dots,i_d}, t^{n+1}) =
    f(\mathbf{x}^\star, t^{n}),
    \label{eq:SL_basic_formulation}
\end{equation}
where $\mathbf{x}^\star = (x^{(1),\star},\dots,x^{(d),\star})$ 
is the foot of the local characteristic passing through 
$\mathbf{x}_{i_1,\dots,i_d}$ at time $t^{n+1}$. To evaluate $f(\mathbf{x}^\star, t^n)$ in \eqref{eq:SL_basic_formulation},
we use a tensor-product quadratic interpolation from a local $3^d$-point stencil. The interpolant lies in $\mathbb{Q}_2 = \bigotimes_{\mu=1}^d \mathbb{P}_2$, achieving third-order local accuracy for smooth solutions.

Assume that $\mathbf{x}^\star$ lies in the cell $I_{j_1,\dots,j_d}$. 
For each direction $\mu = 1,\dots,d$, select the three-point stencil index set
$\mathcal{J}_\mu = \{j_\mu - 1,\, j_\mu,\, j_\mu + 1\}$,
and let $\mathcal{J} = \mathcal{J}_1 \times \cdots \times \mathcal{J}_d$.
Define the normalized coordinate
\begin{equation*}
    \xi^{(\mu)} := \frac{x^{(\mu),\star} - x^{(\mu)}_{j_\mu}}{\Delta x^{(\mu)}},
\end{equation*}
and let $\mathbf{w}(\xi^{(\mu)}) \in \mathbb{R}^3$ be the 1D quadratic Lagrange interpolation weight vector 
\begin{equation}
    \label{eq:1d_quadratic_interp_weight}
    \mathbf{w}(\xi^{(\mu)})
    =
    \Bigl(
        \tfrac{1}{2}\xi^{(\mu)}(\xi^{(\mu)}-1),\;
        1 - (\xi^{(\mu)})^2,\;
        \tfrac{1}{2}\xi^{(\mu)}(\xi^{(\mu)}+1)
    \Bigr).
\end{equation}
The rank-one weight tensor $\mathbf{W}(\bm\xi) := \bigotimes_{\mu=1}^d \mathbf{w}(\xi^{(\mu)})$, with $\bm\xi = (\xi^{(1)},\dots,\xi^{(d)})$, encodes the full $d$-dimensional interpolation weights. Let $\mathbf{F}^n \in \mathbb{R}^{N_1 \times \cdots \times N_d}$ denote the global nodal tensor
with entries $F^n_{i_1,\dots,i_d} = f(\mathbf{x}_{i_1,\dots,i_d}, t^n)$,
and let $\mathbf{F}^n_{\mathcal{J}}$ denote its $3^d$ local block extracted at the stencil $\mathcal{J}$. The new SL--FD update is given by
\begin{equation}\label{eq:tp_interp}
    f(\mathbf{x}^\star,t^n)\approx \sum_{i_1=1}^{3}\sum_{i_2=1}^{3}\cdots\sum_{i_d=1}^{3}
    (F^n_{\mathcal{J}})_{i_1i_2\cdots i_d}\,
    W_{i_1i_2\cdots i_d} 
    := \left\langle
    \mathbf{F}^n_{\mathcal J},
    \mathbf{W}
    \right\rangle,
\end{equation}
where $\langle\cdot,\cdot\rangle$ is the Frobenius inner product. A direct evaluation of \eqref{eq:tp_interp} requires $\mathcal{O}(3^d)$ operations. When $\mathbf{F}^n$ is given in a low-rank format (e.g., TT, Tucker, HT, etc.),
the direct evaluation of \eqref{eq:tp_interp} can be replaced by
a sequence of structured contractions.
We briefly describe the procedure in the HT setting, but note that other formats are analogous.
Let $\mathcal{F}^n$ be an HT approximation of $\mathbf{F}^n$.
At each leaf node $\mu$, let $\mathbf{U}^{(\mu)} \in \mathbb{R}^{N_\mu \times r_\mu}$ be the leaf basis matrix of $\mathcal{F}^n$, and let $\mathbf{U}^{(\mu)}_{\mathcal{J}_\mu} \in \mathbb{R}^{3 \times r_\mu}$ be the submatrix formed by the three rows indexed by $\mathcal{J}_\mu$. We evaluate the interpolated leaf basis at $x^{(\mu),\star}$ via quadratic interpolation,
\begin{equation*}
    \mathbf{c}^{(\mu)}
    :=
    \mathbf{w}(\xi^{(\mu)})\,\mathbf{U}^{(\mu)}_{\mathcal{J}_\mu}
    \in \mathbb{R}^{1\times r_\mu},
    \qquad \mu=1,\dots,d.
\end{equation*}
The inner product in \eqref{eq:tp_interp} is then recovered by contracting
$\{\mathbf{c}^{(\mu)}\}_{\mu=1}^d$ with the non-leaf transfer tensors of $\mathcal{F}^n$ recursively up to the root of the HT dimension tree. With uniform HT rank $r$, this costs $\mathcal{O}(d\,r^3)$,
a significant reduction compared to $\mathcal{O}(d^3 r^3)$ for the reconstruction in \cite{zheng2025semihighD}. Although we use third-order interpolation in this work, this idea extends in a natural way to higher-order interpolation.

Next, we show that the proposed SL--FD scheme is unconditionally stable when applied to the $d$-dimensional constant-coefficient linear advection equation.
\begin{theorem}[Unconditional \(\ell^2\)-stability] \label{thm:uncond_stab} 
Consider the constant-coefficient linear advection equation 
\[ \partial_t f+\bm a\cdot\nabla_{\bm x}f=0 
\] 
on a periodic uniform Cartesian grid. Then the SL--FD scheme with tensor-product quadratic interpolation weights \eqref{eq:1d_quadratic_interp_weight} is unconditionally stable in the discrete \(\ell^2\)-norm. In particular, for any \(\Delta t>0\), we have
\[ \|\mathbf F^{n+1}\|_{\ell^2} \le \|\mathbf F^n\|_{\ell^2}, \qquad n\ge 0. 
\] 
\end{theorem}

\begin{proof}
The result follows from a standard von Neumann analysis. The one-dimensional
quadratic interpolation update has amplification factor bounded by one for any
CFL number, and the tensor-product interpolation gives a multidimensional
amplification factor that factorizes into the product of the one-dimensional
factors. The details of the proof are provided in \cref{app:stability_proof}.
\end{proof}

The SL--FD solver developed above enables localized evaluation of solution entries at the next time level through high-dimensional characteristic tracing and tensor-product reconstruction. This SL--FD solver provides the local entry evaluations required by the HTACA
compression strategy of \cite{zheng2025semihighD}, resulting in a non-conservative SLAR method for high-dimensional kinetic equations. For the linear advection equation \eqref{eq:advection_ddim}, given an HT representation $\mathcal{F}^n$ at time $t^n$, the solution is advanced by
\[
\mathcal{F}^{n+1} 
=
\mathrm{SLAR}\!\left(\mathcal{F}^n; \bm{a}, \Delta t, \varepsilon_{\text{Base}}\right),
\]
where the local SL--FD solver provides the pointwise values required by the SL formulation, and HTACA constructs an HT approximation of the updated tensor from these localized entries. The overall accuracy and rank adaptivity are controlled by the baseline tolerance $\varepsilon_{\text{Base}}$.

For the nonlinear Vlasov--Poisson system, we employ the third-order RK exponential integrator (RKEI) scheme \cite{celledoni2003commutator}, following \cite{cai2021high,zheng2025semihighD}, which decomposes the characteristic evolution into three linear advection subproblems with frozen velocity fields at each stage. Each subproblem is an advection equation, which can be advanced using SLAR. This leads to the abstract update
\begin{equation}
    \mathcal{F}^{n+1}
    =
    \mathrm{SLAR}_{\mathrm{VP}}
    \bigl(
    \mathcal{F}^n;\,
    \Delta t,\,
    \varepsilon_{\mathrm{Base}}
    \bigr).
    \label{eq: slarvp}
\end{equation}
Here, \(\mathrm{SLAR}_{\mathrm{VP}}\) denotes one full SLAR step for the VP system with RKEI time integration. The self-consistent electric field is computed in HT format using the Poisson
solver from \cite{zheng2025semihighD}, which uses the FFT. We refer to that work for implementation details of the field solver and rank-adaptive compression.

\subsection{Implicit LoMaC Correction}\label{sec:LoMaC}

The distribution function obtained from the SLAR update in
\eqref{eq: slarvp} does not, in general, conserve mass, momentum, or energy. This loss of conservation is caused by several factors, including the nonlinear SL discretization as well as the tensor truncation used to compress the adaptive-rank representation.

To enforce local conservation of the macroscopic densities, we introduce an implicit LoMaC framework equipped with an adaptive velocity space weight. The framework has two main components. First, we construct a dynamic weight function that captures the dominant velocity space structure of the distribution and use it to define the weighted Hilbert space in which the LoMaC projection is performed. We motivate this choice in the continuous setting, where the construction can be interpreted in terms of leading velocity modes, and then describe the modifications needed for the fully discrete algorithm. Second, as in our previous work \cite{sands2025adaptive}, we advance the macroscopic moment equations implicitly. This allows the macroscopic update to remain compatible with the large time steps used by the SLAR kinetic solver, including time steps far beyond those permitted by an Eulerian CFL restriction.

\subsubsection{Adaptive-Weight Projection}
\label{sec:AWLoMaC_functional}

In \cite{sands2025adaptive}, we developed an implicit LoMaC framework for the BGK equation in which the correction to the distribution function is constructed using a local Maxwellian ansatz. This ansatz encodes the relevant macroscopic moments and provides the correct decay as \(|v|\to\infty\). For collisional models, this is a natural choice because the collision operator drives the distribution toward a local Maxwellian. For collisionless Vlasov dynamics, however, no analogous relaxation mechanism is present, and the distribution may remain far from local equilibrium. Consequently, a Maxwellian ansatz is not generally well adapted to the non-Maxwellian velocity space structures that arise in Vlasov simulations. To address this issue, we define the weight function adaptively using the leading singular function in the velocity direction. This choice captures the dominant velocity space profile of the distribution. In the continuous setting, the nonnegativity of the leading singular function follows from the positivity of compact integral operators, e.g., the
Krein--Rutman theorem. In the discrete setting, this reduces to the Perron--Frobenius theorem for nonnegative matrices.

Consider a reference distribution \(f_{\mathrm e}(\bm{x},\bm{v},t) > 0\) at a fixed time \(t\). 
For each velocity component \(v^{(\mu)}\), define
\[
\bm{y}^{(\mu)} := (\bm{x},\bm{v}_{\neq \mu})
\in
\Omega_{\bm{x}} \times \Omega_{\bm{v}_{\neq \mu}},
\qquad
\Omega_{\bm{v}_{\neq \mu}}
=
\underset{\nu\neq \mu}{\times}\Omega_{v^{(\nu)}},
\]
where \(\bm{v}_{\neq \mu}\) denotes all velocity components except \(v^{(\mu)}\).
Without loss of generality, we drop the superscript \((\mu)\) and the time variable \(t\) whenever no confusion arises.

We define the Hilbert--Schmidt operator
\[
\mathcal T_{\mu} :
L^2(\Omega_{\bm{y}})
\to
L^2(\Omega_{v^{(\mu)}}),\quad g
\mapsto (\mathcal{T}_\mu g)(v)=
\int_{\Omega_{\bm{y}}}
f_{\mathrm{e}}(\bm{y},v)\,
g(\bm{y})\, d\bm{y}
\]
and the associated self-adjoint operator
\[
\mathcal T_\mu \mathcal T_\mu^* :
L^2(\Omega_{v^{(\mu)}}) \to L^2(\Omega_{v^{(\mu)}}),\quad u \mapsto
(\mathcal T_\mu \mathcal T_\mu^* u)(v)
=
\int_{\Omega_{v^{(\mu)}}}
K_{\mu}(v,v')\,u(v')\,dv',
\]
where the kernel \(K_\mu\) is defined as
\(
K_\mu(v,v')
:=
\int_{\Omega_{\bm{y}}}
f_{\mathrm{e}}(\bm{y}, v)\,
f_{\mathrm{e}}(\bm{y},v')\,d\bm{y}.
\)

The leading eigenfunction of \(\mathcal{T}_\mu\mathcal{T}_{\mu}^*\) admits the variational characterization
\[
u_1^{(\mu)}
=
\arg\max_{\|u\|_{L^2(\Omega_{v^{(\mu)}})}=1}
\left\langle \mathcal T_\mu \mathcal T_\mu^* u,\, u \right\rangle
=
\arg\max_{\|u\|_{L^2(\Omega_{v^{(\mu)}})}=1}
\|\mathcal T_\mu^* u\|_{L^2(\Omega_{\bm{y}})}^2.
\]
Equivalently,
\[
\|\mathcal T_\mu^* u\|_{L^2(\Omega_{\bm{y}})}^2
=
\int_{\Omega_{\bm{y}}}
\left|
\int_{\Omega_{v^{(\mu)}}}
f_{\mathrm{e}}(\bm{y},v)\,u(v)\,dv
\right|^2
d\bm{y}.
\]
This characterization shows that \(u_1^{(\mu)}\) is the first principal mode of the family of 
\(v^{(\mu)}\)-slices of \(f_{\mathrm e}\). Indeed, for each fixed \(\bm{y}\), the inner integral \(\int_{\Omega_{v^{(\mu)}}}
f_{\mathrm{e}}(\bm{y},v)\,u(v)\,dv\) represents the projection coefficient of the slice 
\(f_{\mathrm{e}}(\bm{y},\cdot)\) along the unit direction \(u\),
and the maximization selects the unit function \(u\) that maximizes the total squared projection
over all \(\bm{y}\). Therefore, \(u_1^{(\mu)}\) provides the optimal one-dimensional profile that captures,
in an \(L^2\)-averaged sense, the dominant shape and exponential decay of \(f_{\mathrm{e}}\) in the \(v^{(\mu)}\)-direction. 

We define a weight function $\omega(\bm v)$ as
\[
\omega(\bm v)
:=
\prod_{\mu=1}^{d_v}
u_1^{(\mu)}\bigl(v^{(\mu)}\bigr).
\]
Since \(f_{\mathrm e}>0\), the kernel \(K_\mu(v,v')\) is strictly positive, and hence it follows from the Krein--Rutman theorem that each leading eigenfunction \(u_1^{(\mu)}\) can be chosen to be strictly positive. Consequently, \(\omega(\bm v)>0\) on \(\Omega_{\bm v}\).
We therefore define the weighted inner product
\begin{equation*}
\langle g,h \rangle_{\omega^{-1}}
=
\int_{\Omega_{\bm v}}
g(\bm v)\,h(\bm v)\,\omega(\bm v)^{-1}\,d\bm v,
\end{equation*}
and the associated weighted Hilbert space \(L^2_{\omega^{-1}}(\Omega_{\bm v})\). Since \(u_1^{(\mu)}\) captures the dominant profile of \(f_{\mathrm e}\) in each \(v^{(\mu)}\),
the weight \(\omega(\bm v)\) reflects the principal velocity structure of \(f_{\mathrm e}\).
Moreover, under the exponential decay of \(f_{\mathrm e}\), the leading eigenfunctions
\(u_1^{(\mu)}\) inherit exponential decay in \(v^{(\mu)}\), and hence the weight
\(\omega(\bm v)\) also decays exponentially. Consequently,
\[
\mathcal{N}(\omega)
=
\mathrm{span}
\Big\{
\omega(\bm{v}),
\;
v^{(1)}\,\omega(\bm{v}),\dots,
v^{(d_v)}\,\omega(\bm{v}),
\;
|\bm{v}|^2\,\omega(\bm{v})
\Big\}
\subset
L^2_{\omega^{-1}}(\Omega_{\bm{v}}).
\]
Let \(\mathcal{P}_{\mathcal{N}(\omega)}\) denote the orthogonal projection onto
\(\mathcal{N}(\omega)\) with respect to \(\langle \cdot,\cdot \rangle_{\omega^{-1}}\).
We define the adaptive-weight LoMaC projection as the phase-space operator
\[
\mathcal P:
L^2(\Omega_{\bm{x}})\widehat{\otimes}L^2_{\omega^{-1}}(\Omega_{\bm{v}})
\;\to\;
L^2(\Omega_{\bm{x}})\widehat{\otimes}\mathcal{N}(\omega),
\]
given by
\(
\mathcal P
=
\mathcal I_{L^2(\Omega_{\bm{x}})}
\otimes
\mathcal{P}_{\mathcal{N}(\omega)}.
\)
Here \(\widehat{\otimes}\) denotes the Hilbert tensor product (i.e., the completion of the algebraic tensor product), and \(\mathcal I_{L^2(\Omega_{\bm{x}})}\) is the identity operator on \(L^2(\Omega_{\bm{x}})\). Equivalently, \(\mathcal P\) applies the weighted projection pointwise in \(\bm x\). The following theorem establishes the main properties of the weighted projection.

\begin{theorem}[Adaptive-weight LoMaC projection]
\label{thm:adaptive_LoMaC_functional}

If
\(
f\in L^2(\Omega_{\bm{x}})\widehat{\otimes}L^2_{\omega^{-1}}(\Omega_{\bm{v}})
\),
then the following statements hold:

\begin{enumerate}

\item (Macroscopic moment preservation)
The projection \(\mathcal P f\) preserves the velocity moments associated with
\(
\psi(\bm v)\in\{1,\,v^{(1)},\ldots,v^{(d_v)},\,|\bm v|^2\},
\)
that is,
\[
\int_{\Omega_{\bm v}}
(\mathcal P f)(\bm x,\bm v)\,\psi(\bm v)\,d\bm v
=
\int_{\Omega_{\bm v}}
f(\bm x,\bm v)\,\psi(\bm v)\,d\bm v.
\]

\item (Low-dimensional velocity structure)
\[
\mathcal P f
\in
L^2(\Omega_{\bm{x}})
\,\widehat{\otimes}\,
\mathcal{N}(\omega)
\subset
L^2(\Omega_{\bm{x}})\widehat{\otimes}L^2_{\omega^{-1}}(\Omega_{\bm{v}}),
\qquad
\dim \mathcal{N}(\omega) = d_v + 2.
\]

\item (Velocity decay structure)
If \(\Omega_{\bm{v}}=\mathbb{R}^{d_v}\) and \(\omega(\bm{v})\) decays exponentially as \(|\bm{v}|\to\infty\),
then \(\mathcal P f\) inherits the same exponential decay in \(\bm{v}\),
up to at most quadratic polynomial factors.

\end{enumerate}
\end{theorem}

\begin{proof}
We divide the proof into three steps.

\medskip
\noindent
\textit{Step 1: Moment representers and preservation.}
Let \(\psi(\bm{v}) \in \{1, v^{(1)}, \ldots, v^{(d_v)}, |\bm{v}|^2\}\) and define
\(\psi^{\omega}(\bm{v}) := \psi(\bm{v})\,\omega(\bm{v}) \in \mathcal{N}(\omega)\).
Since
\[
\langle g,h\rangle_{\omega^{-1}}
=
\int_{\Omega_{\bm{v}}} g(\bm{v})h(\bm{v})\,\omega(\bm{v})^{-1}\,d\bm{v},
\]
we have
\[
\int_{\Omega_{\bm{v}}}
(\mathcal I-\mathcal P)f(\bm{x},\bm{v})\,
\psi(\bm{v})\, d\bm{v}
=
\left\langle
(\mathcal I-\mathcal P)f(\bm{x},\cdot),
\psi^{\omega}
\right\rangle_{\omega^{-1}}.
\]
Here \(\mathcal{I}\) denotes the identity operator on the phase space.
Since \(\psi^{\omega} \in \mathcal{N}(\omega)\) and
\(
\mathcal P
=
\mathcal I_{L^2(\Omega_{\bm{x}})}\otimes
\mathcal{P}_{\mathcal{N}(\omega)},
\)
the definition of orthogonal projection in
\(L^2_{\omega^{-1}}(\Omega_{\bm{v}})\) implies
\[
\left\langle
(\mathcal I-\mathcal P)f(\bm{x},\cdot),
\psi^{\omega}
\right\rangle_{\omega^{-1}}
=0.
\]
Hence,
\[
\int_{\Omega_{\bm{v}}}
(\mathcal P f)(\bm{x},\bm{v})\,
\psi(\bm{v})\, d\bm{v}
=
\int_{\Omega_{\bm{v}}}
f(\bm{x},\bm{v})\,
\psi(\bm{v})\, d\bm{v}.
\]

\medskip
\noindent
\textit{Step 2: Range characterization.}
This follows directly from the definition of \(\mathcal{P}\).

\medskip
\noindent
\textit{Step 3: Velocity decay structure.}
Since
\(
\mathcal P f \in L^2(\Omega_{\bm{x}})\widehat{\otimes}\mathcal{N}(\omega),
\)
it admits the representation
\begin{equation*}
(\mathcal P f)(\bm{x},\bm{v})
=
\sum_{k=1}^{d_v+2} c_k(\bm{x})\,\psi_k(\bm{v})\,\omega(\bm{v}),
\qquad
\psi_k(\bm{v}) \in \{1, v_1, \ldots, v_{d_v}, |\bm{v}|^2\}.
\end{equation*}
In particular, if \(\Omega_{\bm{v}}=\mathbb{R}^{d_v}\) and \(\omega(\bm{v})\) decays exponentially fast as \(|\bm{v}|\to\infty\),
then the decay of \(\mathcal P f\) as \(|\bm{v}|\to\infty\)
is governed by \(\omega(\bm{v})\), up to the polynomial factors
\(1, v_1, \ldots, v_{d_v}, |\bm{v}|^2\).
\end{proof}

\begin{remark}[Adaptive velocity weight] 
The weight \(\omega(\bm v)\) is chosen to reflect the dominant velocity space structure of the distribution function at the current time step. In contrast to the original LoMaC projection~\cite{guo2024local}, where the weight function is prescribed a priori, the present construction determines the weight directly from the numerical solution. In particular, the one-dimensional factors defining \(\omega(\bm v)\) are obtained from leading velocity-mode information extracted from the current phase-space distribution. Since such a distribution is available at every time step, the projection space \(\mathcal N(\omega)\) adapts dynamically to the evolving velocity profile. For this reason, we refer to \(\omega\) as an \emph{adaptive velocity weight}. 
\end{remark}

\begin{remark}[Positivity of the weight function]
To ensure well-posedness of the weighted inner product
\(\langle\cdot,\cdot\rangle_{\omega^{-1}}\), the weight
\(\omega(\bm v)\) must be strictly positive. In the continuous construction,
this positivity is guaranteed under suitable positivity assumptions on the
reference distribution \(f_{\mathrm e}\). In the discrete setting, strict positivity of the numerical distribution is difficult to enforce. Therefore, in practice, we take the absolute value of each discrete velocity factor \(u_1^{(\mu)}\) and replace any entries below a small floor parameter \(\delta>0\) with \(\delta\). This guarantees a strictly positive discrete weight function; see the Supplementary Material for details.
\end{remark}

\subsubsection{Discrete Projection}\label{sec:AWLoMaC_discrete}

We now describe the adaptive-weight LoMaC projection in the fully discrete setting.
Let
\[
\mathcal F^n
\in
\mathbb R^{N_{x^{(1)}}\times\cdots\times N_{x^{(d_x)}}
\times
N_{v^{(1)}}\times\cdots\times N_{v^{(d_v)}}}
\]
denote the discrete phase-space tensor at time level $t^n$. At each time step, an intermediate distribution
$\mathcal F^{n+1,\star}$ is first obtained from the non-conservative SLAR solver.
From $\mathcal F^{n+1,\star}$, we construct an adaptive \emph{discrete} velocity weight
\[
\mathbf{\Omega}^\star
\in
\mathbb R^{N_{v^{(1)}}\times\cdots\times N_{v^{(d_v)}}},
\]
which is kept fixed during the LoMaC correction at time level $n+1$.
In particular, $\mathbf{\Omega}^\star$ is chosen in a separable (rank-one) tensor form; the explicit construction is deferred to Section~2 in the Supplementary Material.

Let $\mathcal{N}(\mathbf{\Omega}^\star)$ denote the discrete velocity subspace
\[
\mathcal{N}(\mathbf{\Omega}^\star)
=
\mathrm{span}
\big\{
\mathbf{\Omega}^\star,
\;
\mathbf{V}_1 * \mathbf{\Omega}^\star,
\;
\cdots,
\;
\mathbf{V}_{d_v} * \mathbf{\Omega}^\star,
\;
 \mathbf{V}^{* 2} * \mathbf{\Omega}^\star
\big\}
\subset
\mathbb R^{N_{v^{(1)}}\times\cdots\times N_{v^{(d_v)}}},
\]
whose dimension is $d_v+2$.
Here $*$ denotes the Hadamard (entrywise) product. For each $\mu = 1,\ldots,d_v$, $\mathbf{V}_\mu\in\mathbb{R}^{N_{v^{(1)}}\times\cdots\times N_{v^{(d_v)}}}$ denotes the discrete samples of $v^{(\mu)}$, i.e. the $\mu$-fiber of $\mathbf{V}_\mu$ equals the grid vector $\mathbf{v}^{(\mu)}\in\mathbb{R}^{N_{v^{(\mu)}}}$. Furthermore, $\mathbf{V}^{* 2} = \sum_{\mu=1}^{d_v}\mathbf{V}_\mu*\mathbf{V}_\mu$ represents the discrete samples of $|\bm{v}|^2$.
We define the discrete projection
\[
\mathcal P^\star
=
\mathcal I_{\mathbb R^{N_{x^{(1)}}\times\cdots\times N_{x^{(d_x)}}}}
\otimes
\mathcal P_{\mathcal{N}(\mathbf{\Omega}^\star)},
\]
where
\(\mathcal P_{\mathcal{N}(\mathbf{\Omega}^\star)}\)
denotes the weighted orthogonal projection onto
\(\mathcal{N}(\mathbf{\Omega}^\star)\)
with respect to the discrete counterpart of the weighted inner product
\(\langle \cdot,\cdot\rangle_{\omega^{-1}}\).
An explicit construction of this projection is provided in
Section~2 of the Supplementary Material.
This discrete operator mirrors the projection introduced in \cref{thm:adaptive_LoMaC_functional} and serves as its
finite-dimensional analogue.
In particular, it preserves the discrete macroscopic moments,
enforces the same $(d_v+2)$-dimensional velocity subspace structure,
and retains the adaptive velocity-decay template encoded by the
discrete weight $\mathbf{\Omega}^\star$.

By construction, $\mathcal P^\star$ acts only on the velocity degrees of freedom.
Hence, for each fixed spatial multi-index $i$,
\[
(\mathcal P^\star \mathcal F)_{i,\cdot}
\in
\mathcal{N}(\mathbf{\Omega}^\star).
\]
Moreover, the discrete projection preserves the raw moments, i.e.,
\begin{equation}
\label{eq:moments_pres_disc}
\mathbf M(\mathcal P^\star \mathcal F)
=
\mathbf M(\mathcal F),
\end{equation}
where $\mathbf M(\cdot)$ denotes the discrete analogue of \eqref{eq:raw moments}.
The derivation of \eqref{eq:moments_pres_disc} is provided in
Section~2 of the Supplementary Material. This is the discrete analogue to property 1 of \cref{thm:adaptive_LoMaC_functional}.

Since $\mathcal P^\star \mathcal F$ lies in the $(d_v+2)$-dimensional
velocity subspace $\mathcal N(\mathbf{\Omega}^\star)$ for each spatial index, it is uniquely determined from the conserved variables $\mathbf U(\mathcal F)$. In particular, the raw moments are related to the conserved variables through relation \eqref{eq:M from U}. Consequently, there exists a discrete reconstruction operator
\[
\mathcal R :
\mathbb R^{(d_v+2)\prod_{\mu=1}^{d_x} N_{x^{(\mu)}}}
\;\longrightarrow\;
\mathbb{R}^{N_{x^{(1)}}\times\cdots\times N_{x^{(d_x)}}
\times
N_{v^{(1)}}\times\cdots\times N_{v^{(d_v)}}},
\]
such that
\[
\mathcal P^\star \mathcal F
=
\mathcal R\big(\mathbf U(\mathcal F)\big).
\]
We provide an explicit construction of $\mathcal R$ in
Section~2 of the Supplementary Material using weighted Gram-Schmidt orthogonalization.

Let $\mathbf U^{n+1}$ denote the macroscopic solution obtained from a
conservative solver of the system \eqref{eq:macro_sys}. 
The discrete adaptive-weight LoMaC correction is defined as
\begin{equation}\label{eq:LoMaC_update_discrete}
\mathcal F^{n+1}
=
\mathcal F^{n+1,\star}
-
\mathcal R\big(\mathbf U(\mathcal F^{n+1,\star})\big)
+
\mathcal R\big(\mathbf U^{n+1}\big).
\end{equation}
The first two terms remove the contributions of the raw moments associated with the non-conservative kinetic solution obtained with the SLAR method. The last term provides the correction which is consistent with the local conservation laws \eqref{eq:macro_sys}. By the construction of $\mathcal R$, it follows that
\[
\mathbf U(\mathcal F^{n+1})
=
\mathbf U^{n+1}.
\]
That is, the conserved moments from the corrected kinetic solution agree exactly with those produced by the conservative macroscopic solver. This is a direct consequence of equation \eqref{eq:moments_pres_disc} and the identity \eqref{eq:M from U}.

\subsubsection{Implicit Moment Update}
\label{sssection: implicit}

The implicit first-order backward Euler (BE) discretization of the macroscopic system \eqref{eq:macro_sys} gives
\begin{equation}\label{eq:BE_macro}
    \mathbf{U}^{n+1}
    = \mathbf{U}^{n}
      - \Delta t\,\mathcal{D}_{\bm{x}} \cdot \big(\mathbf{F}(\mathbf{U}^{n+1})\big)
      + \Delta t\,\mathbf{S}^{n+1},
\end{equation}
where $\mathcal{D}_{\bm{x}}\cdot$ denotes a conservative discrete approximation of
$\nabla_{\bm{x}}\cdot(\cdot)$.
The source term is given by
\[
\mathbf{S}^{n+1} =
\begin{pmatrix}
\bm{0},
\bm{\rho}_{\mathrm{e}}^{\,n+1} * \mathbf{E}_1^{\,n+1},
\cdots
\bm{\rho}_{\mathrm{e}}^{\,n+1} * \mathbf{E}_{d_x}^{\,n+1},
\bm{0}
\end{pmatrix}^{\top},
\]
where $*$ denotes element-wise multiplication and
$\mathbf{E}^{n+1}=-\nabla_{\bm{x}}\phi^{n+1}$ is obtained from the Poisson solve at time level $n+1$.
In \eqref{eq:BE_macro}, the macroscopic flux
$\mathbf{F}(\mathbf{U}^{n+1})$
contains both the kinetic moment contributions and the field-energy flux
$\mathbf{F}_{e_p}$.
The kinetic part of the flux is obtained as velocity moments of the corrected kinetic state
$\mathcal{F}^{n+1}$ through a LoMaC update \eqref{eq:LoMaC_update_discrete}. $\mathcal{F}^{n+1}$ is a function which explicitly depends on the moments from the SLAR update of the kinetic solution (the first two terms in \eqref{eq:LoMaC_update_discrete}), and encodes the dependence on $\mathbf{U}^{n+1}$ in an implicit way (the last term in \eqref{eq:LoMaC_update_discrete}). 
For each spatial direction, we adopt the conservative flux-difference formulation as in~\cite{sands2025adaptive}, and generalize it to the multidimensional discretization by applying the same conservative derivative operator
componentwise in each spatial direction. The field-energy flux $\mathbf{F}_{e_p}$ depends on the electric potential
$\bm{\phi}^{n+1}$ and its time derivative.
To ensure consistency with the discrete energy structure,
all spatial derivatives appearing in
$\mathbf{F}_{e_p}$,
the electric field
$\mathbf{E}^{n+1}$,
and the auxiliary Poisson problem for $\partial_t\phi^{n+1}$
are evaluated using spectral differentiation via the FFT.
This choice guarantees algebraic compatibility between the discrete gradient,
divergence, and Laplacian operators.
After applying space and time discretization to the system \eqref{eq:BE_macro}, we obtain the following nonlinear
algebraic system for $\mathbf{U}^{n+1}$:
\begin{equation}\label{eq:BE_nonlinear_system}
    \mathbf{G}(\mathbf{U}^{n+1}) = \bm{0},
    \qquad
    \mathbf{G}(\mathbf{U}^{n+1})
    :=
    \mathbf{U}^{n+1}-\mathbf{U}^{n}
    +\Delta t\,\mathcal{D}_{\bm{x}}\cdot\big(\mathbf{F}(\mathbf{U}^{n+1})\big)
    -\Delta t\,\mathbf{S}^{n+1}.
\end{equation}
We solve the system \eqref{eq:BE_nonlinear_system} using a JFNK method \cite{knoll2004jacobian,knoll2005jacobian}, as in our previous work \cite{sands2025adaptive}. An advantage of this approach is that it eliminates the need to store the Jacobian associated with the nonlinear residual. Instead, the Jacobian-vector products used in the inner Krylov solver, e.g., GMRES, can be approximated using differences of the nonlinear residual. Additionally, for modest time steps, the non-conservative SLAR solution provides a reasonable initial guess to the nonlinear system so that it is not necessary to consider preconditioning. This observation is consistent with our previous work \cite{sands2025adaptive}. For readers interested in full implementation details, including the spectral spatial discretization and solver parameters, an open-source reference implementation of the present method is available in the accompanying code repository \href{https://github.com/Nanyheng/LoMaC_SLAR}{\texttt{https://github.com/Nanyheng/LoMaC\_SLAR}}.

The BE discretization introduced above can be extended to higher-order
implicit time integrators.
We consider two classes of methods to address the coupling between the kinetic and macroscopic systems.
The first option is to employ diagonally-implicit RK (DIRK) schemes.
In this case, each intermediate stage requires the solution of a nonlinear macroscopic system,
which is handled by the same JFNK strategy described above
(See~\cite{sands2025adaptive} for further details.). For DIRK schemes, the intermediate stage distributions are generated by applying the SLAR update \eqref{eq: slarvp} at the corresponding stage times. Each stage therefore consists of a SLAR evaluation and a JFNK solve.
While this approach yields high-order accuracy,
the computational cost scales with the number of stages.
The second option is to use BDF methods.
Unlike DIRK schemes, BDF methods do not introduce intermediate stage times,
so that each time step requires only one evaluation of
the SLAR update \eqref{eq: slarvp} and one JFNK solve.
Higher-order BDF schemes naturally lead to a single nonlinear system per step,
which is again solved by JFNK.
However, BDF methods require appropriate start-up procedures. In our implementation, we employ a DIRK integrator to generate the initial
time steps required for start-up. We then switch to a BDF method for the remaining time integration. A fixed-step BDF-2 scheme is often the default choice due to its A-stability. In our implementation, however, we use the classical variable-step BDF-3 scheme \cite{gear1971numerical} to achieve higher temporal accuracy and adaptive time stepping. Additional care is required because, even in the fixed-step case, the stability region of BDF-3 does not fully contain the imaginary axis. For related results on step-ratio conditions for the perturbation stability of variable-step BDF methods, see \cite{li2022stability}.

\begin{figure}[t]
\begin{center}
\begin{tikzpicture}[scale=0.8, transform shape]

    \draw (2.5,0) rectangle (6.5,1); 
    \node at (2+2+0.5,0.5) {$\mathrm{SLAR}_{\mathrm{VP}}\bigl(\mathcal{F}^n;\,\Delta t,\,\varepsilon_{\mathrm{Base}}\bigr)$};

    \draw (3.4,2) rectangle (4.6,3); 
    \node at (2+2,2.5) {$\mathcal{F}^{n+1,\star}$}; 

    \draw (7.25,0) rectangle (14.5,1); 
    \node at (10.9,0.5) {$\mathbf{U}^{n+1} = \mathbf{U}^{n}
      - \Delta t\,\mathcal{D}_{\bm{x}} \cdot\Big(\mathbf{F} \big( \mathcal F^{n+1} \big) \Big)
      + \Delta t\,\mathbf{S}^{n+1}$}; 

    \draw (7.45,2) rectangle (14.65,3); 
    \node at (11.1,2.5) {$\mathcal F^{n+1}= \mathcal F^{n+1,\star}-\mathcal R\big(\mathbf U(\mathcal F^{n+1,\star})\big) + \mathcal R\big(\mathbf U^{n+1}\big)$}; 

    \node at (11,-1.5) {\textbf{Converged?}}; 
    \node at (14.8,-1.8) {\scriptsize No};
    \node at (7.0,-1.8) {\scriptsize Yes};
    
    \node[anchor=west] at (4,1.5) {\scriptsize Adaptive-rank non-conservative update};
    \node[anchor=west] at (11,-0.65) {\scriptsize Implicit conservative moment update};
    \node[anchor=west] at (11,1.5) {\scriptsize Update macroscopic flux};

    \node[anchor=west] at (14.75,2.8) {\scriptsize Modify closure};

    \node[anchor=west] at (4,-0.75) {\scriptsize Prepare for the next time step};
    
    \node at (6.2, 2.8) {\scriptsize Initialize closure};

    \draw[->][thick, dashed] (4,1) -- (4,2);

    \draw[->][thick, dashed] (4.6,2.5) -- (7.45,2.5);

    \draw[->][thick] (11,2) -- (11,1);

    \draw[->][thick] (11,0) -- (11,-1.25);

    \draw[thick] (12.5,-1.5) -- (16+1,-1.5);
    \draw[thick] (17,-1.5) -- (16+1,2.5);
    \draw[->][thick] (17,2.5) -- (14.65,2.5);

    \draw[thick, dashed] (9.5,-1.5) -- (4,-1.5);
    \draw[->][thick, dashed] (4,-1.5) -- (4,0);
\end{tikzpicture}
\end{center}
\caption{Flowchart for the VP system with the SLAR method and the implicit LoMaC correction.}
\label{fig:Fluid correction diagram}
\end{figure}
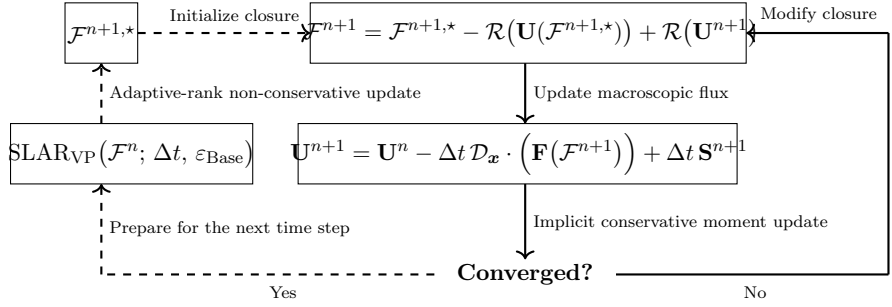

The flow chart of the proposed implicit LoMaC algorithm is shown in \Cref{fig:Fluid correction diagram}. We denote by \(\mathcal F^{n+1,\star}\) the non-conservative distribution obtained from the SLAR update \eqref{eq: slarvp}. The corrected distribution \(\mathcal F^{n+1}\) is then obtained through the LoMaC correction of the form \eqref{eq:LoMaC_update_discrete}. In this update, the components depending on \(\mathcal F^{n+1,\star}\) are known explicitly, while the dependence on the macroscopic state \(\mathbf U^{n+1}\) enters implicitly through the macroscopic moment system \eqref{eq:macro_sys}. This correction is used to close the macroscopic system, which is solved by a JFNK method. Upon convergence, the resulting macroscopic state \(\mathbf U^{n+1}\) is used to complete the correction and obtain \(\mathcal F^{n+1}\). The kinetic and macroscopic solutions are then self-consistent and serve as the input for the next time step.

\section{Numerical Results}\label{sec:numerical_tests}

We present benchmark results for the 1D--1V VP system in \Cref{sec:1D1V_VP_tests} and the 2D--2V VP system in \Cref{sec:2D2V_VP_tests}. We use the S-stable DIRK-3 method from Theorem~5 of \cite{alexander1977diagonally} for start-up, followed by the variable-step BDF-3 method in \cite{gear1971numerical}.
In the macroscopic correction, the JFNK solver uses tolerances of $10^{-10}$ for Newton iteration and $10^{-4}$ for the inner Krylov solver. 

\subsection{1D--1V Examples}\label{sec:1D1V_VP_tests}

In this subsection, we report three 1D--1V benchmark tests: Landau damping, the two-stream instability, and the bump-on-tail instability. Unless otherwise noted, we use the HTACA tolerance \(\varepsilon_{\mathrm{Base}}=10^{-4}\) from \cite{zheng2025semihighD}. These tests assess the accuracy, robustness, adaptive-rank behavior and conservation of the proposed scheme.

\begin{example}(1D--1V Landau damping)
Consider the initial condition
\begin{equation*}
f(x,v,0)=\frac1{\sqrt{2\pi}}\left(1+\alpha\cos(kx)\right)\exp\left(-\frac{v^2}{2}\right),
\end{equation*}
where $\Omega_{x}=[0,4\pi]$, $\Omega_{v}=[-2\pi,2\pi]$, \(k=0.5\). We set \(\alpha=0.01\) for weak Landau damping and \(\alpha=0.5\) for strong Landau damping. Unless otherwise specified, we use a \(256\times256\) mesh and a CFL number of 5 for the tests below. For weak Landau damping, we set \(\varepsilon_{\text{Base}}=10^{-5}\) to match the resolution requirement.

\Cref{fig:LD_summary} shows numerical results for both the weak and strong Landau damping tests. The first two rows show the evolution of the electric energy, the phase-space distribution at \(t=40\), and the corresponding adaptive weight function selected by the LoMaC projection for the weak and strong damping configurations, respectively. In the weak damping case, the method captures the rapid decay in the electric field energy which is transferred to kinetic energy. The distribution function remains close to a perturbed Maxwellian and develops only mild filamentation. The corresponding adaptive weight inherits the dominant velocity space structure which is a smooth Maxwellian-like profile. In the strong damping case, the electric field energy first undergoes a rapid decay, as in the weak case, which is later followed by growth as particles transfer their kinetic energy back into the electric field. The phase-space distribution exhibits pronounced filamentation and nonlinear deformation due to trapping, and the adaptive weight captures both the filamentary features and the velocity decay of the distribution. The third row of \Cref{fig:LD_summary} shows the conservation properties for the strong damping case. The errors in total mass, momentum, and total energy remain at the level of the nonlinear solver tolerance throughout the simulation, demonstrating that the LoMaC correction effectively enforces macroscopic conservation. Without the LoMaC correction, the corresponding errors are \(\mathcal{O}(10^{-4})\) for mass and energy, and \(\mathcal{O}(10^{-3})\) for momentum.

\begin{figure}[!htbp]
    \centering
    
    \hspace{-0.4cm}
    \subfigure{
    \includegraphics[width=0.32\textwidth]{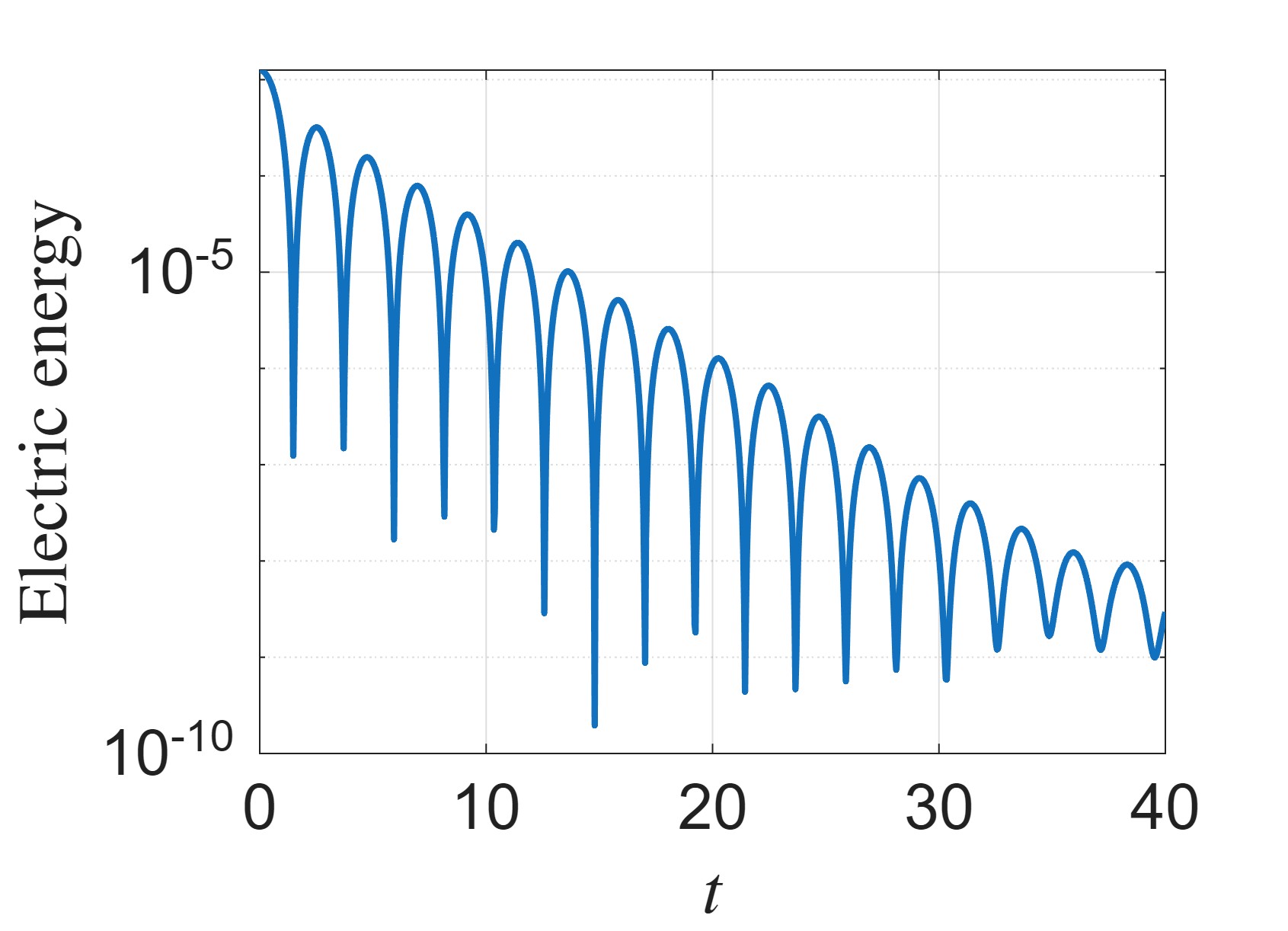}
    }
    \hspace{-0.55cm}
    \subfigure{
    \includegraphics[width=0.32\textwidth]{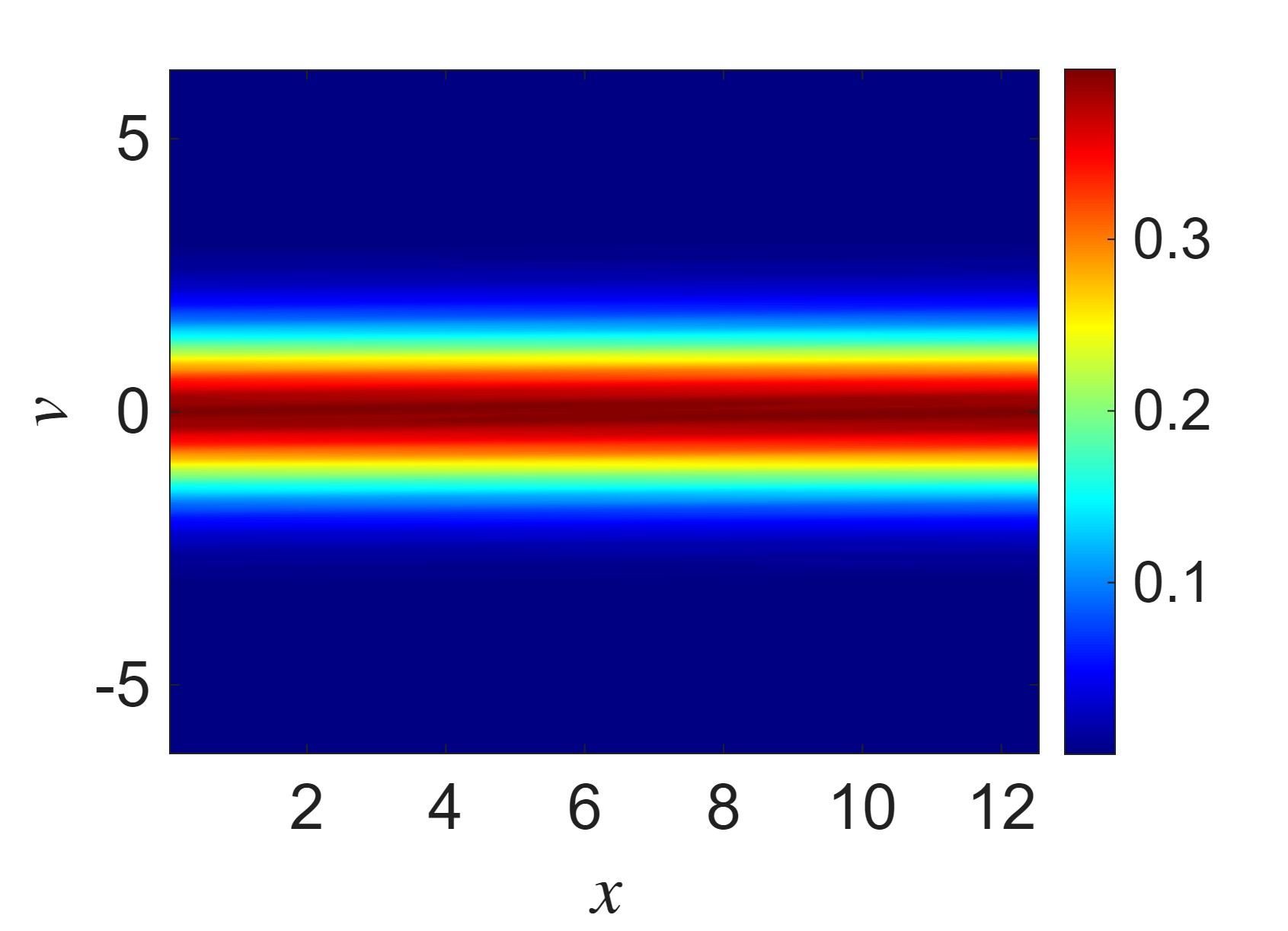}
    }
    \hspace{-0.55cm}
    \subfigure{
    \includegraphics[width=0.28\textwidth]{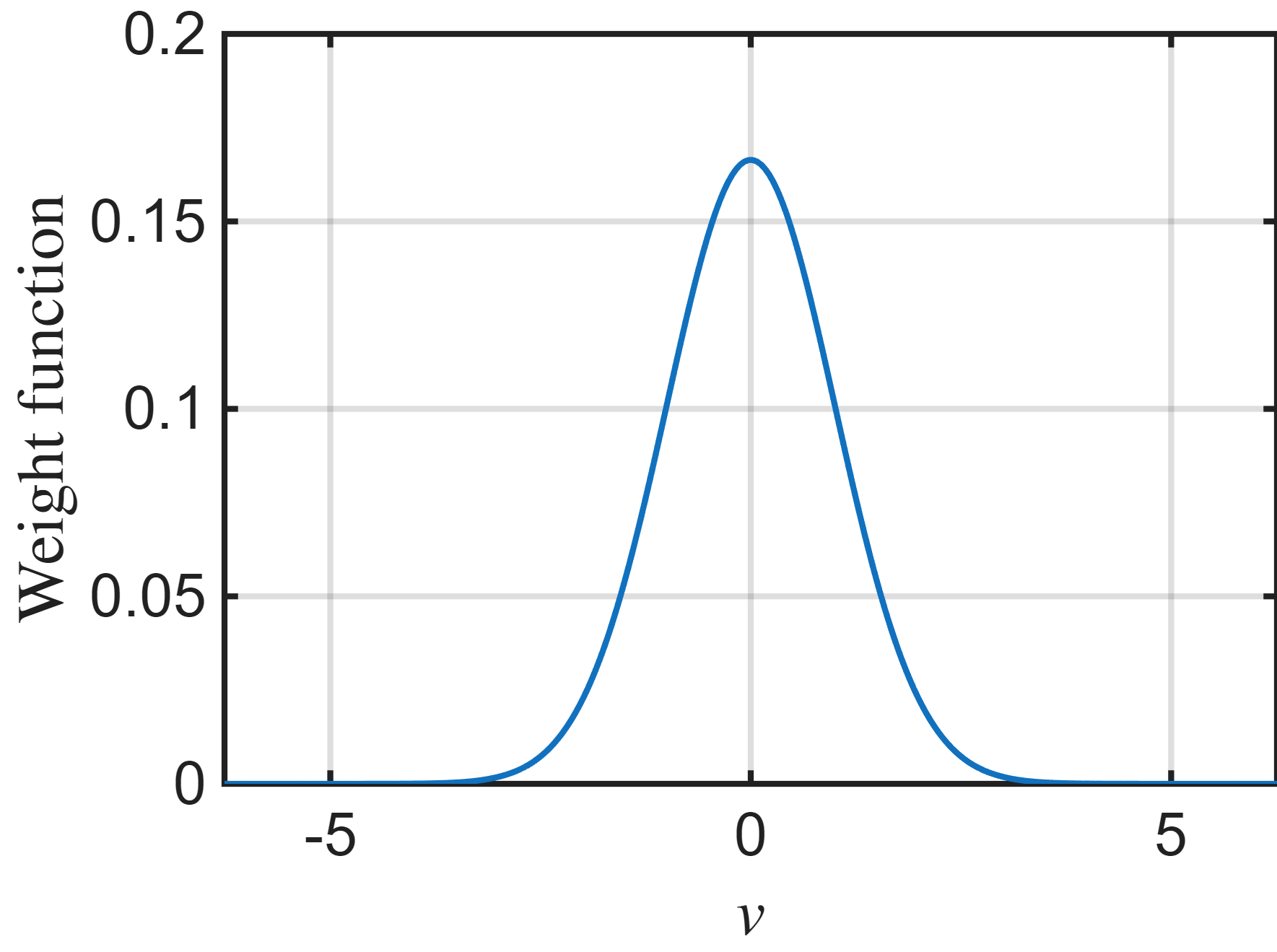}
    }

    \vspace{-0.25cm}

    \hspace{-0.4cm}
    \subfigure{
    \includegraphics[width=0.32\textwidth]{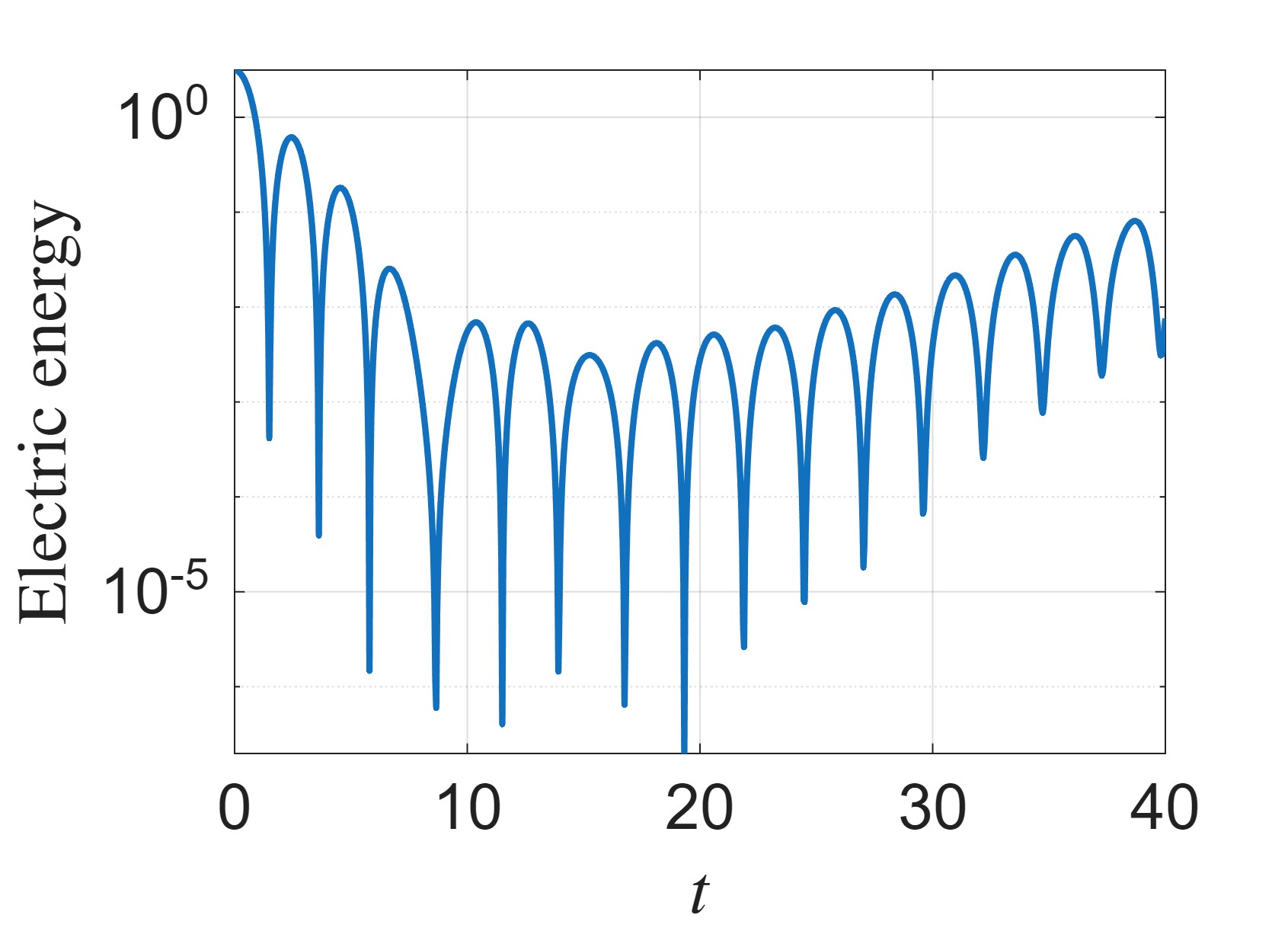}
    }
    \hspace{-0.55cm}
    \subfigure{
    \includegraphics[width=0.32\textwidth]{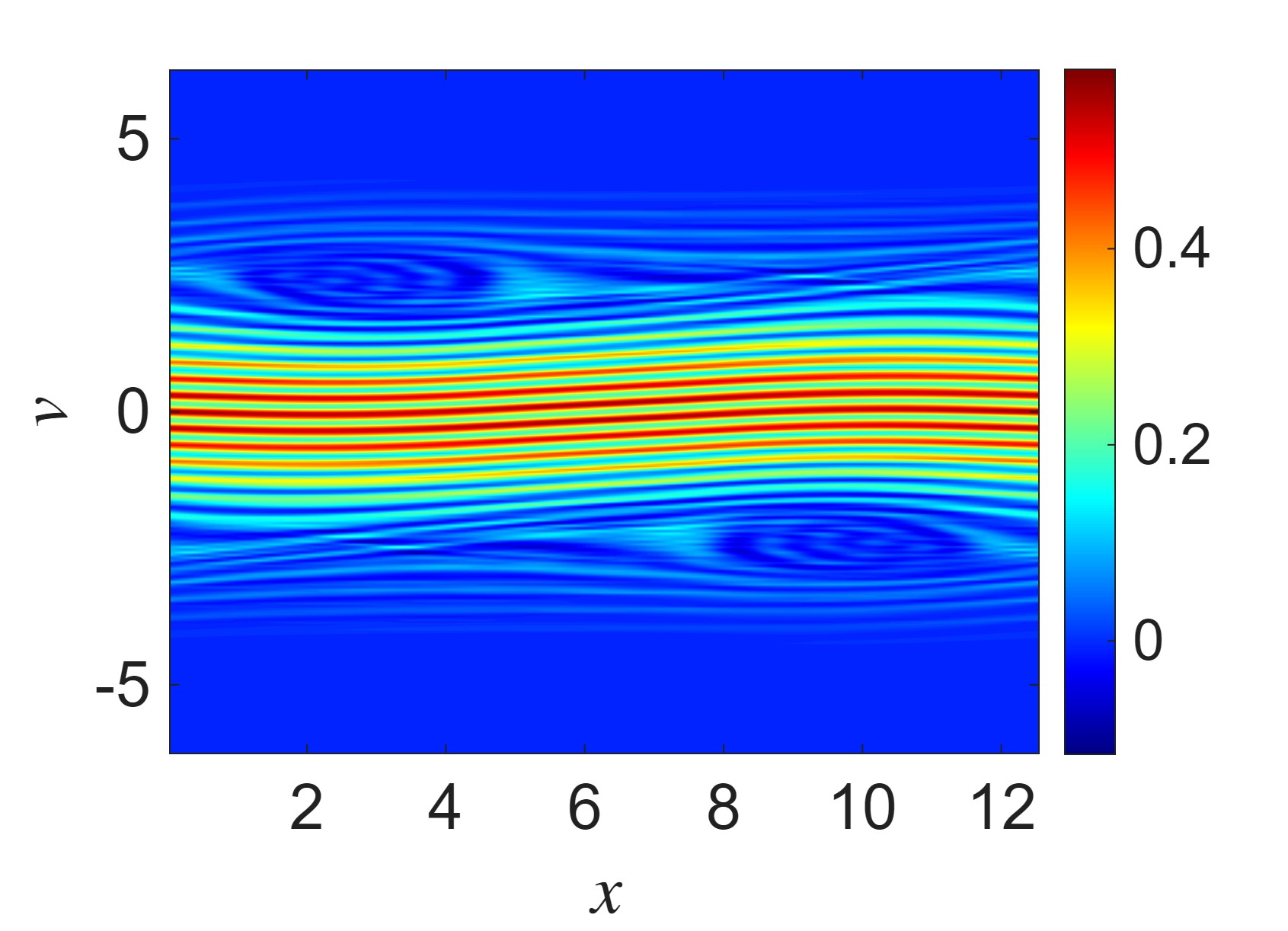}
    }
    \hspace{-0.55cm}
    \subfigure{
    \includegraphics[width=0.28\textwidth]{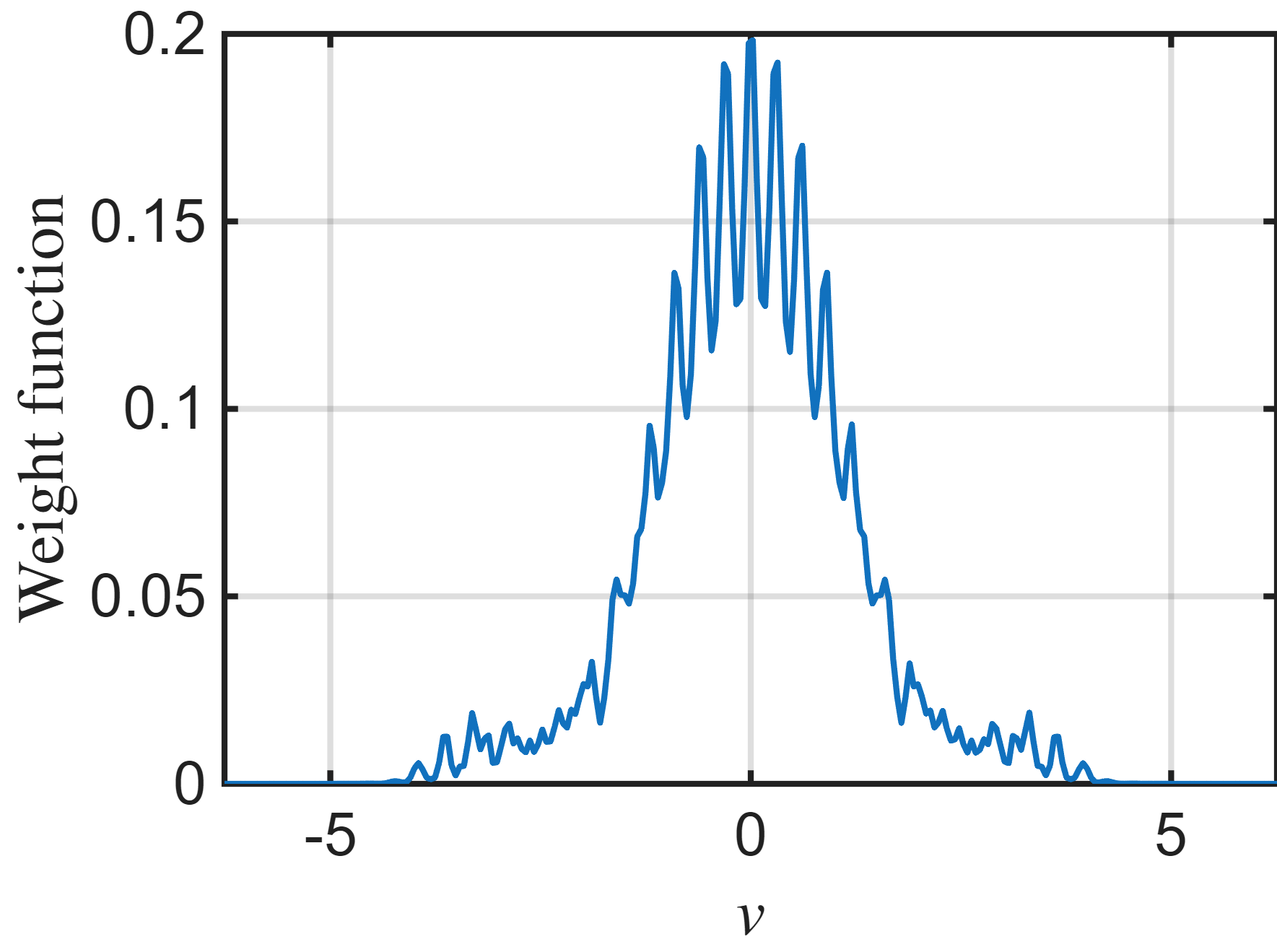}
    }

    \vspace{-0.25cm}

    \hspace{-0.2cm}
    \subfigure{
    \includegraphics[width=0.32\textwidth]{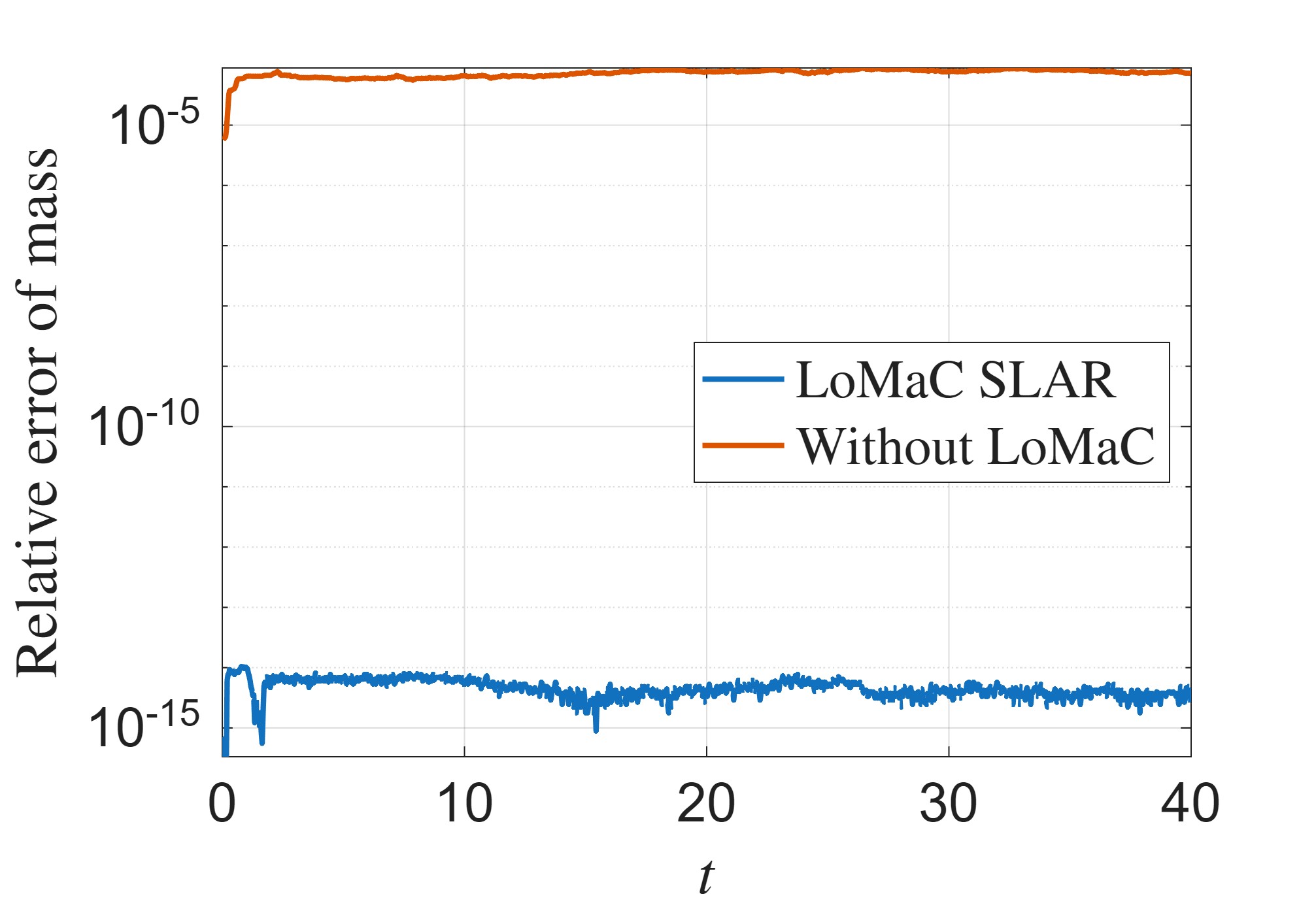}
    }
    \hspace{-0.55cm}
    \subfigure{
    \includegraphics[width=0.32\textwidth]{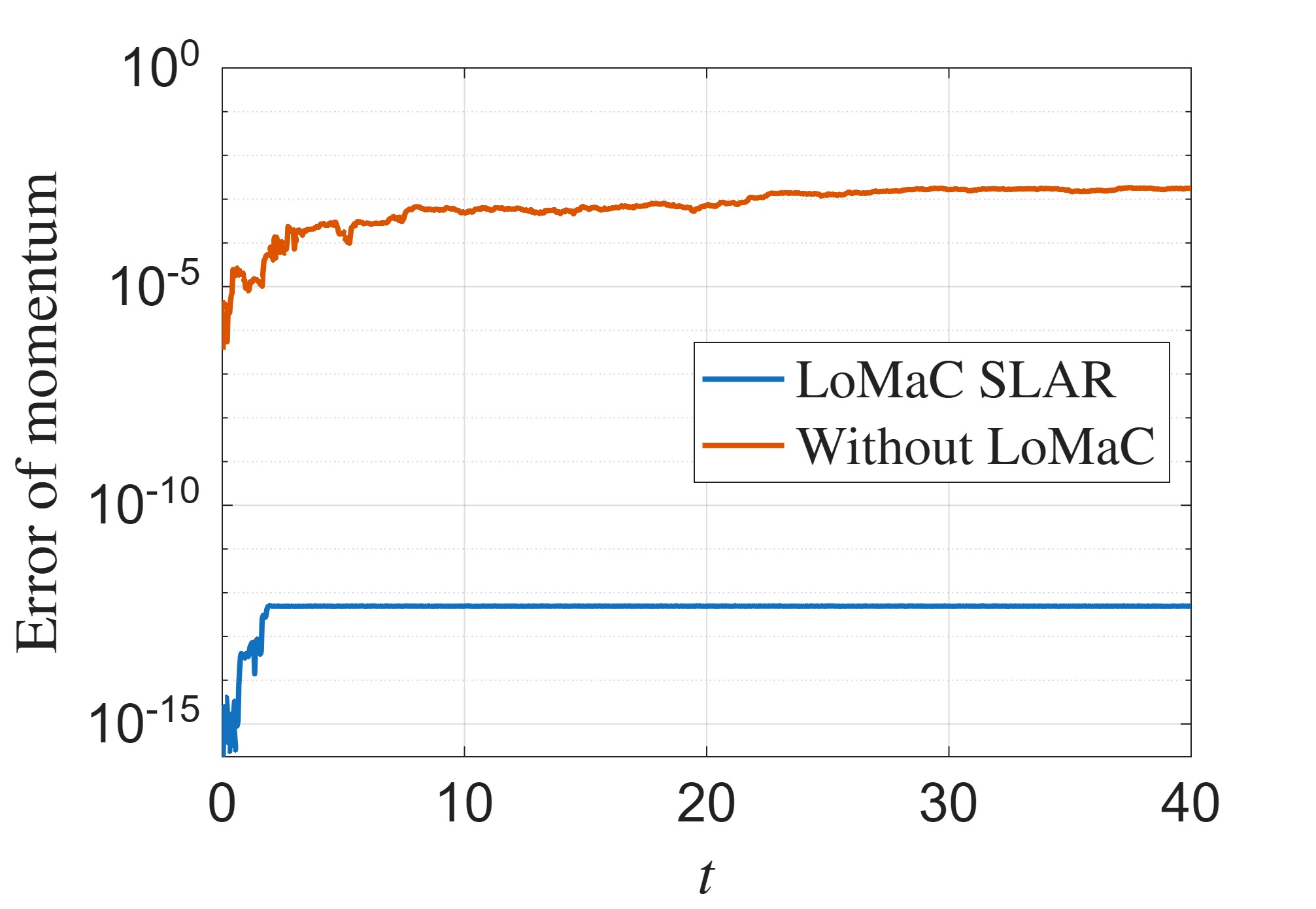}
    }
    \hspace{-0.55cm}
    \subfigure{
    \includegraphics[width=0.32\textwidth]{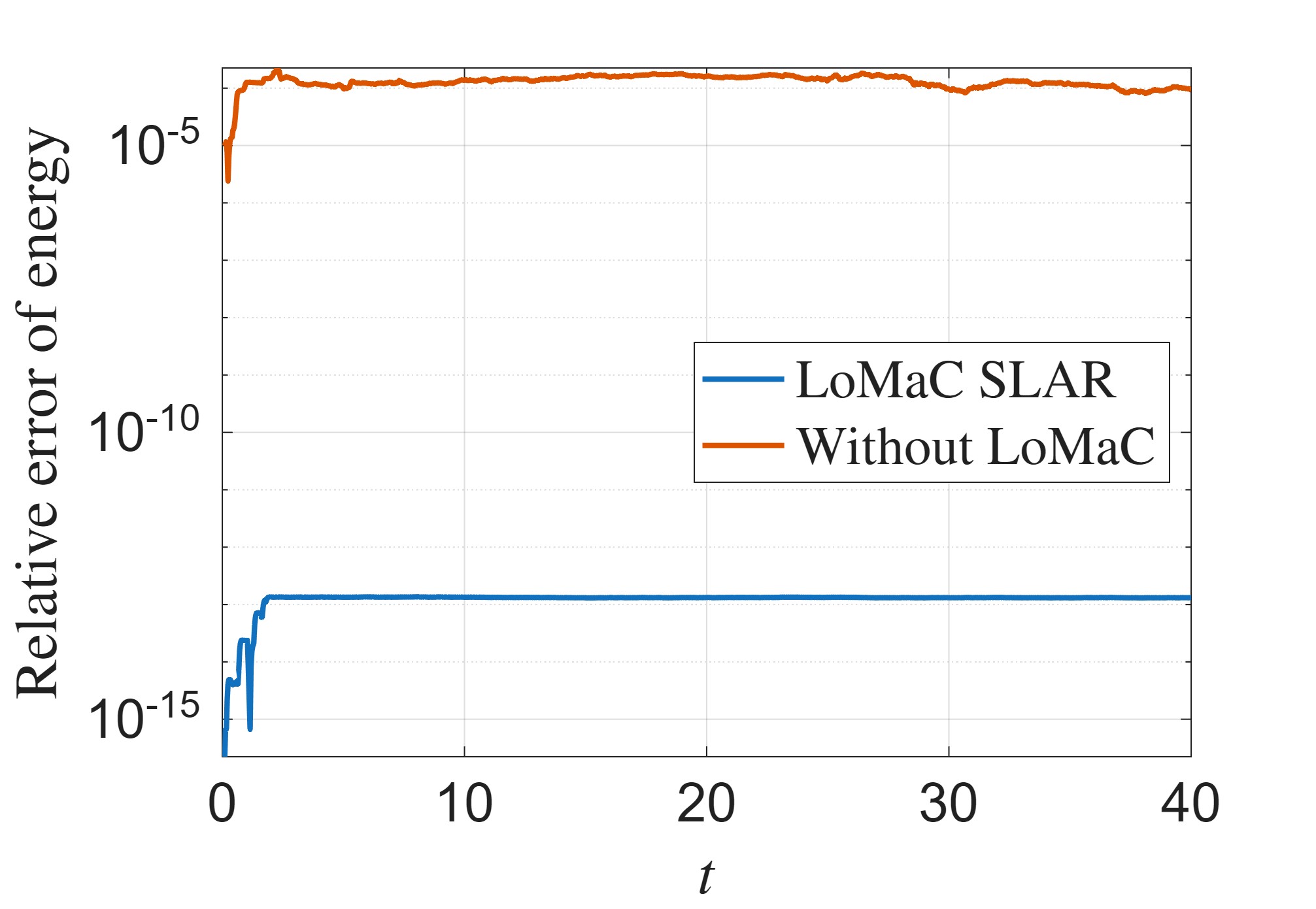}
    }
    \caption{(1D--1V Landau damping). The first row shows the electric energy, a contour plot of the distribution function at \(t=40\), and the adaptive weight function for the weak Landau damping configuration. The second row shows the corresponding quantities for the strong Landau damping configuration. The third row shows the relative or absolute errors in mass, momentum, and total energy for the SLAR method with and without the correction for the strong Landau damping.}
    \label{fig:LD_summary}
\end{figure}

\Cref{fig:CFL_vs_something_SLD} presents a temporal refinement study of the proposed solver using two different phase-space meshes. Reference solutions are computed with the fourth-order conservative SL--FV method \cite{zheng2022fourth} on a refined \(512\times512\) mesh with a CFL of 1. The log-log error curves confirm third-order temporal accuracy. As the CFL number increases, the average numerical rank grows moderately, while the Newton and Krylov iteration counts increase noticeably for sufficiently large CFL numbers. Since the implicit macroscopic correction is initialized using a non-conservative update, larger time steps produce a larger initial nonlinear residual in the JFNK solve. This requires more iterations to reach the prescribed tolerance. No preconditioning is used, so this sensitivity reflects the conditioning of the implicit macroscopic closure at large time steps. These results motivate future preconditioning strategies for the LoMaC solver. Despite the increased iteration counts, the scheme remains stable over the tested CFL range.

\begin{figure}[!htbp]
    \centering
    
    \hspace{-0.85cm}
    \subfigure{
    \includegraphics[width=0.42\textwidth]{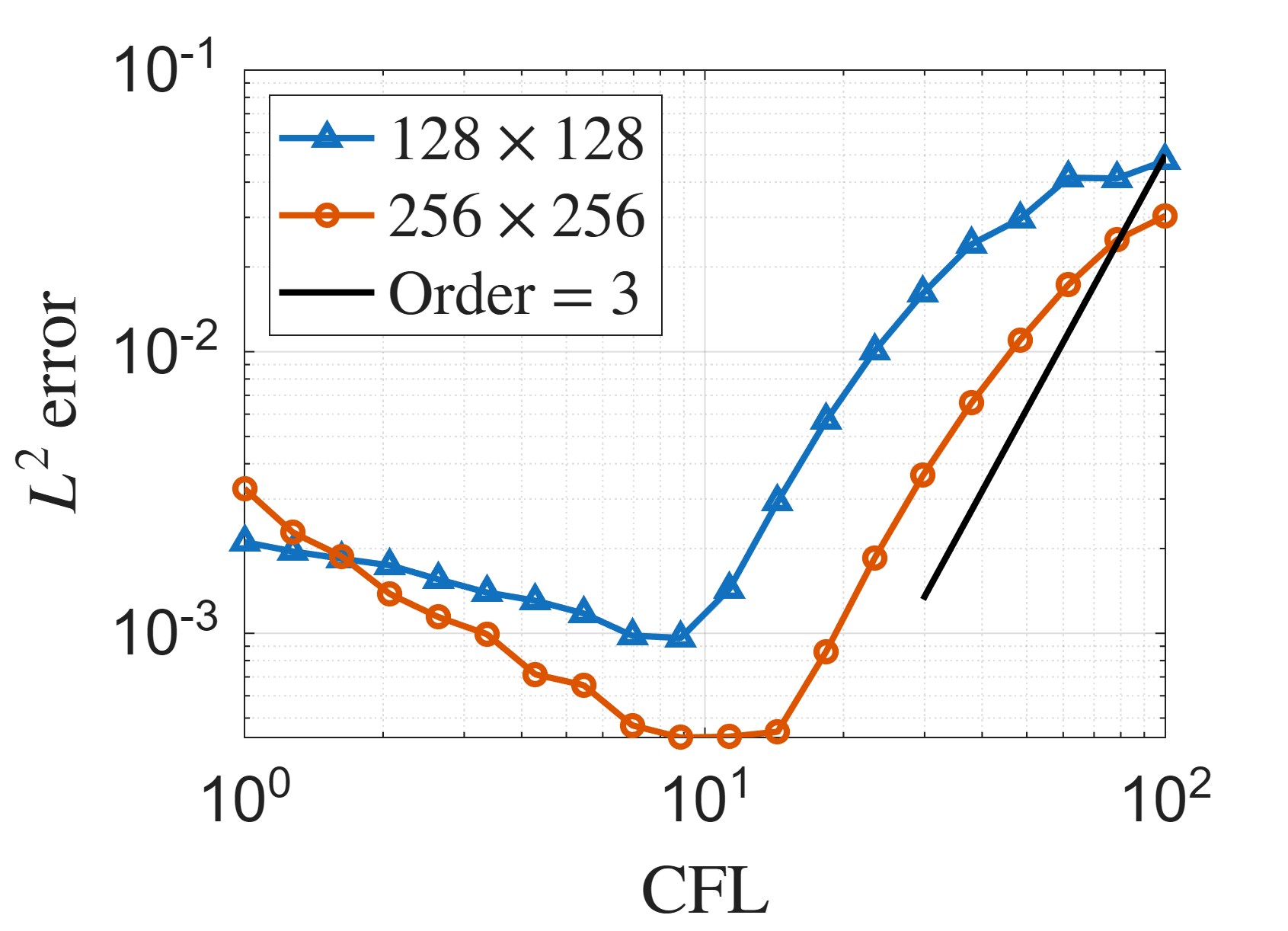}
    }
    \hspace{-0.55cm}
    \subfigure{
    \includegraphics[width=0.42\textwidth]{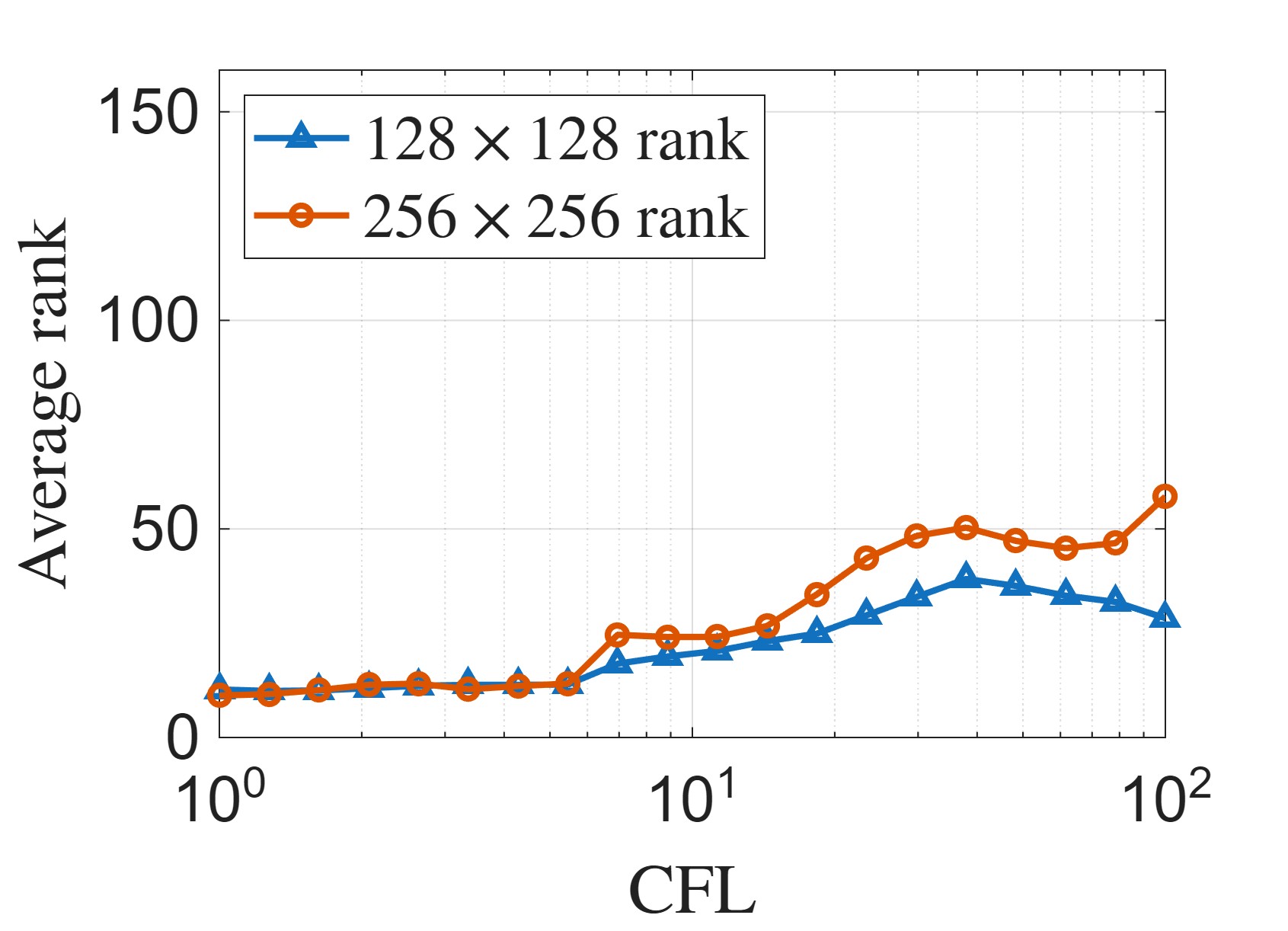}
    }
    
    \vspace{-0.25cm}

    \hspace{-0.6cm}
    \subfigure{
    \includegraphics[width=0.4\textwidth]{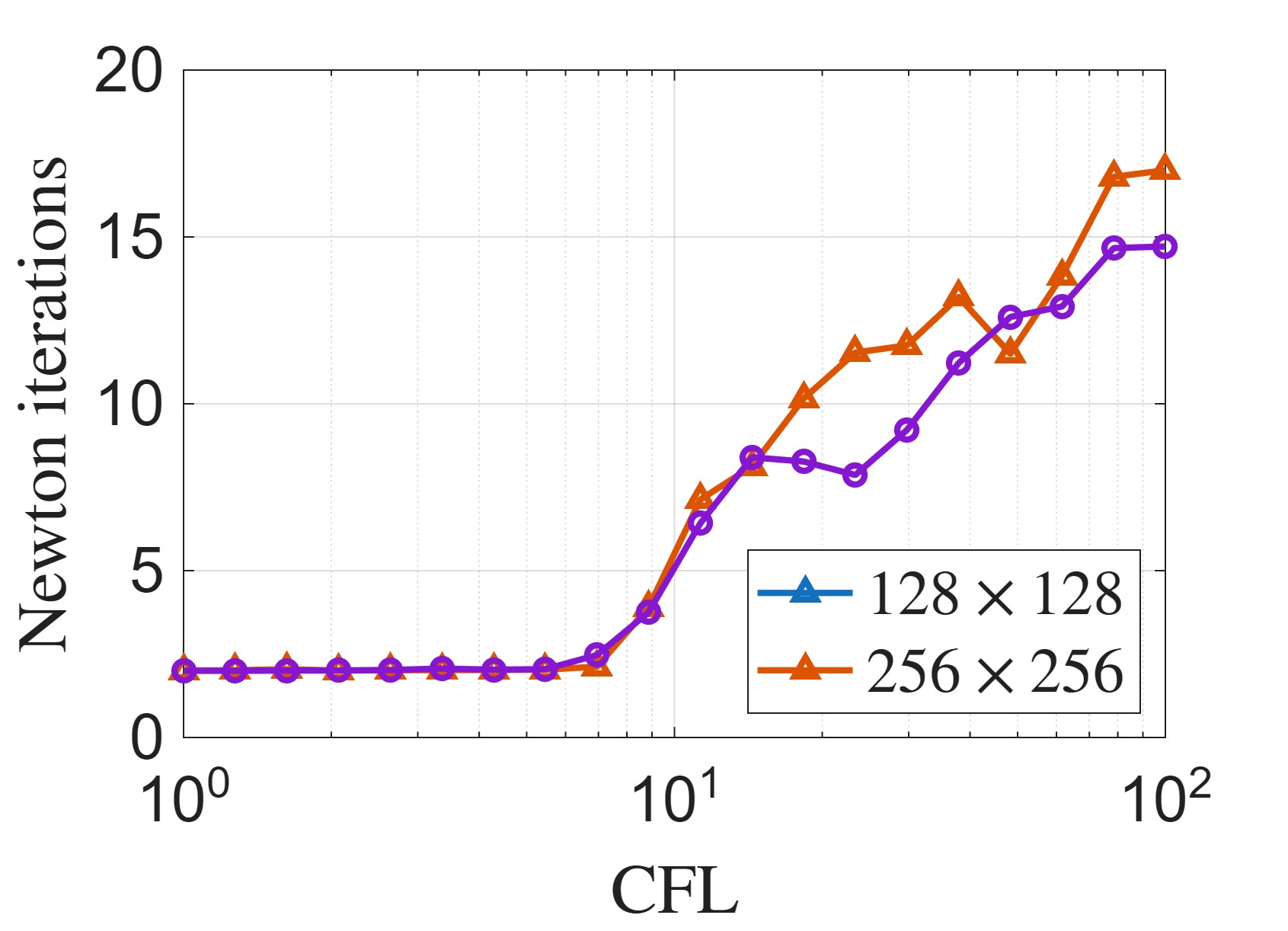}
    }
    \hspace{-0.55cm}
    \subfigure{
    \includegraphics[width=0.4\textwidth]{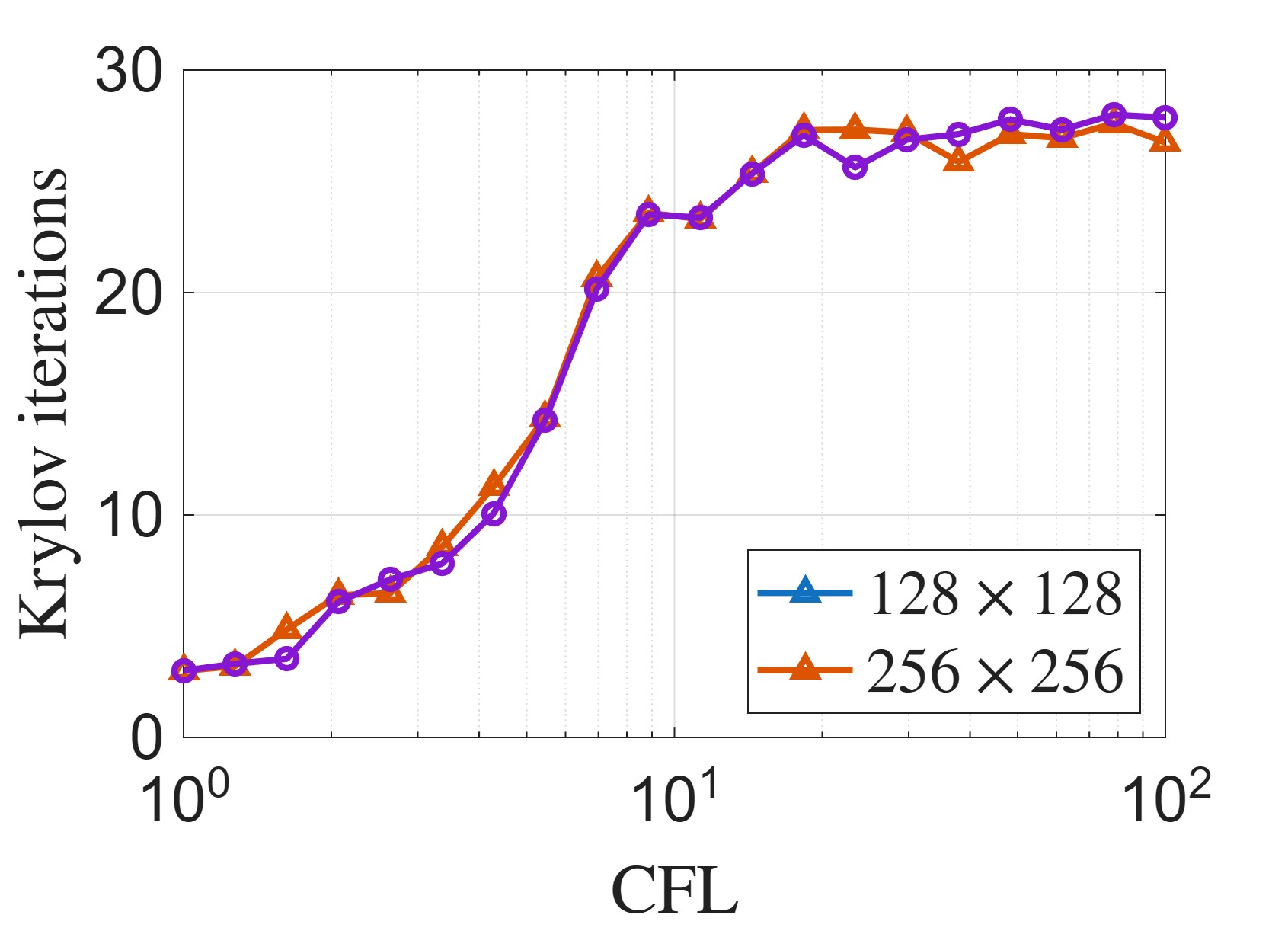}
    }
    \caption{(1D--1V Strong Landau damping). Temporal refinement study using different CFL numbers. Each experiment uses a fixed mesh and evolves the solution up to time $t=5$. Top left: Log-log plot of CFL numbers versus $L^2$ errors. Top right: semi-log plot of CFL numbers versus average ranks. Bottom left: semi-log plot of CFL numbers versus average Newton iterations per time step. Bottom right: semi-log plot of CFL numbers versus average Krylov iterations in each Newton step.}
    \label{fig:CFL_vs_something_SLD}
\end{figure}

\end{example}

\begin{example}(1D--1V Two-stream instability). Consider the initial condition
\begin{equation*}
\begin{split}
f(x,v,t=0) = \frac1{\sqrt{2\pi}}\left(1+\alpha\cos(kx)\right)\left[\exp\left(-\frac{(v-v_0)^2}{2}\right)+\exp\left(-\frac{(v+v_0)^2}{2}\right)\right],
\end{split}
\end{equation*}
where $\Omega_{x}=[0,10\pi]$, $\Omega_{v}=[-8,8]$, \(k = 0.2\), \(\alpha = 0.001\). In the numerical test below, a mesh of $256\times256$ and a CFL of $5$ are used. 

\Cref{fig:TSI_results} shows representative results for the two-stream instability. The first row displays the electric energy, the rank history, and a contour plot of the numerical solution at \(t=40\). The electric energy exhibits the expected linear growth followed by nonlinear saturation during the formation of the vortex. The rank increases sharply during the linear growth phase, which reflects the rapid development of fine-scale phase-space structures, as evidenced by the characteristic particle-trapping regions observed in the contour plot. These results show that the low-rank representation captures the essential nonlinear dynamics while maintaining a compressed approximation of the solution. The second row shows a mesh plot of the distribution function and the corresponding adaptive weight selected by the LoMaC projection at \(t=40\). The distribution exhibits a pronounced vortex structure in phase space. The adaptive weight captures the resulting bimodal velocity profile of the two streams and preserves the overall velocity decay, which reflects the dominant velocity space structure of the nonlinear solution. The conservation results are similar to those for the Landau damping test and are therefore omitted for conciseness.

\begin{figure}[!htbp]
\centering

    \hspace{-0.4cm}
    \subfigure{
    \includegraphics[width=0.33\textwidth]{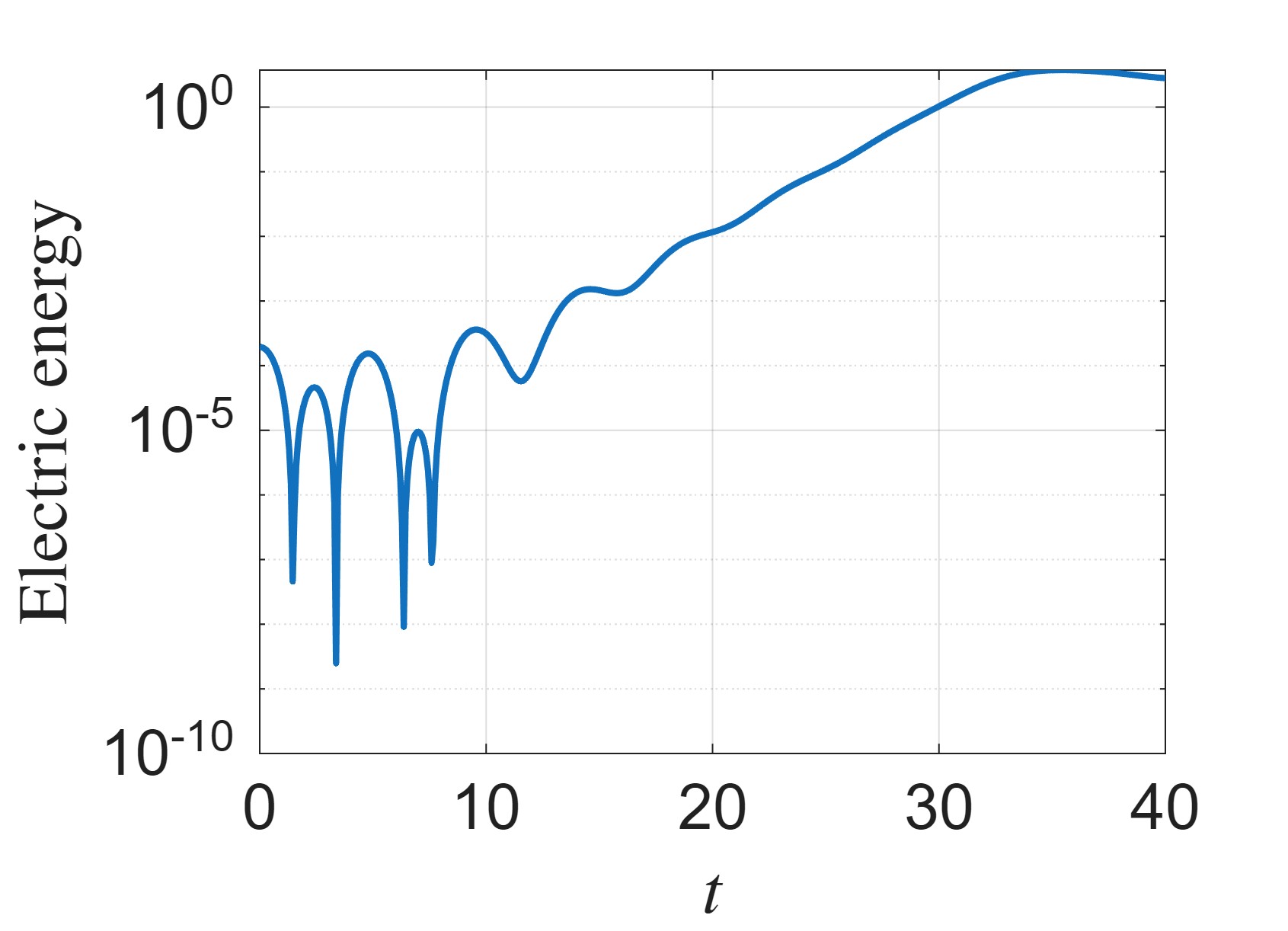}
    }
    \hspace{-0.55cm}
    \subfigure{
    \includegraphics[width=0.33\textwidth]{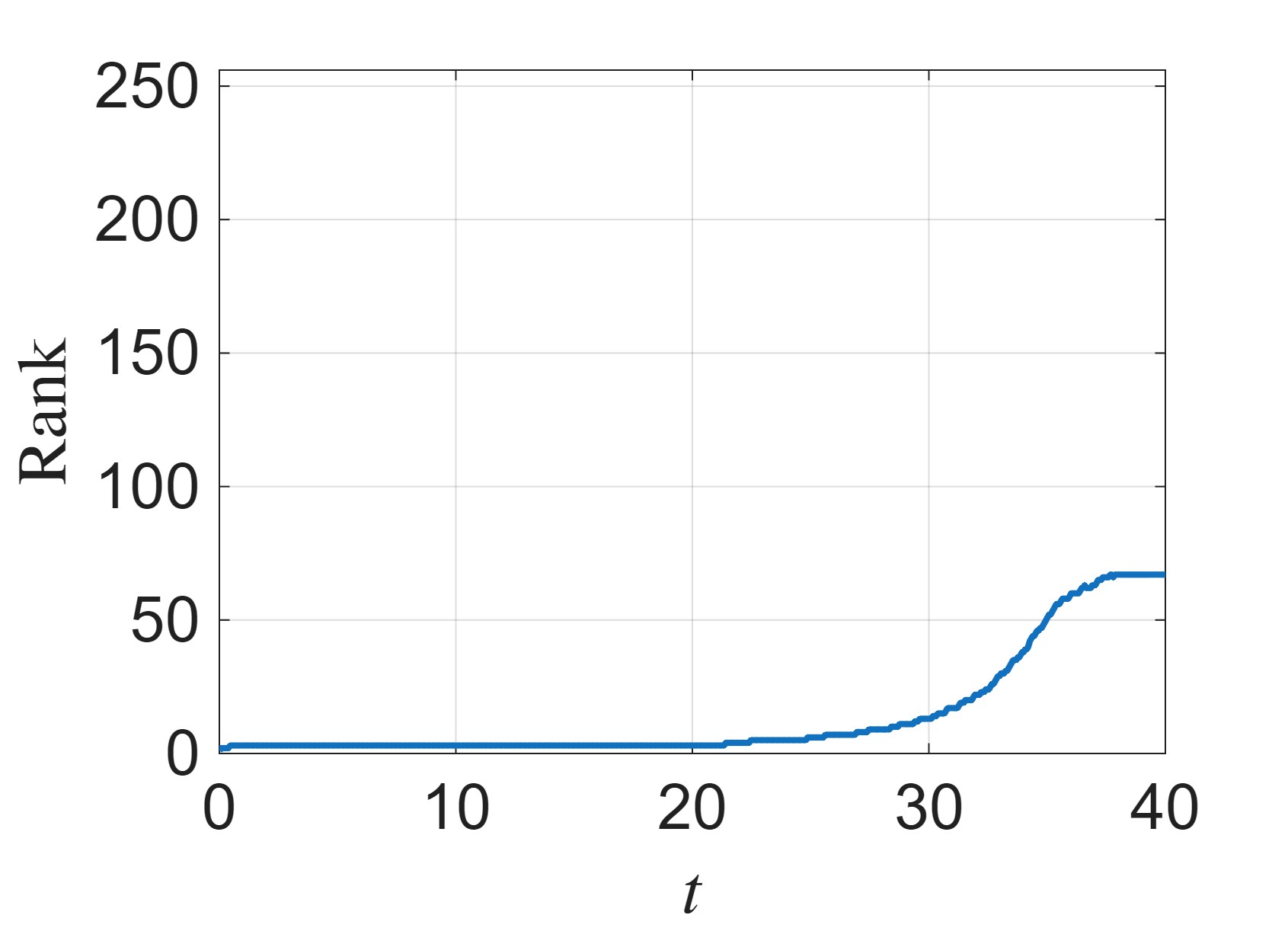}
    }
    \hspace{-0.55cm}
    \subfigure{
    \includegraphics[width=0.33\textwidth]{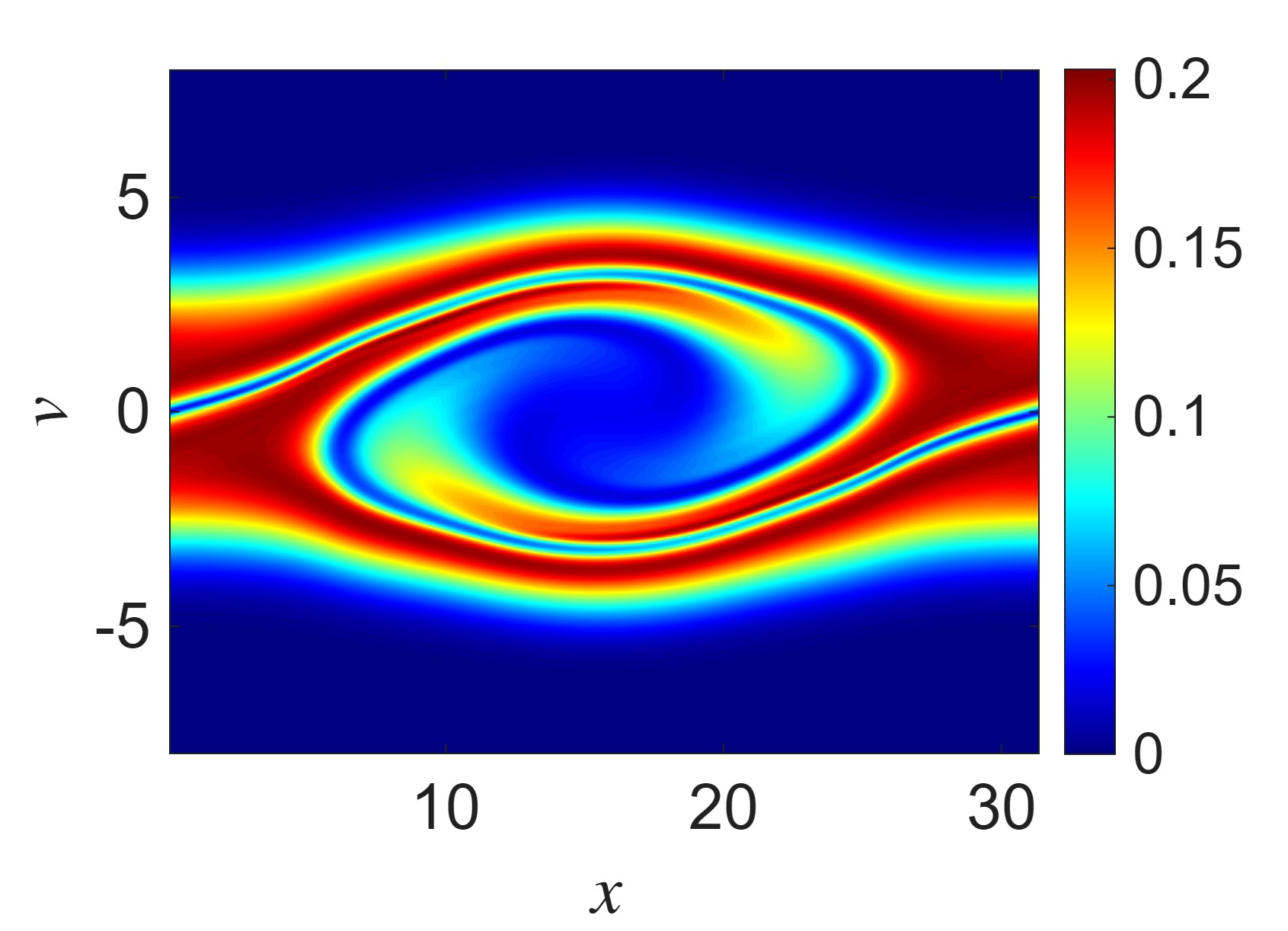}
    }

    \hspace{-0.4cm}
    \subfigure{
    \includegraphics[width=0.4\textwidth]{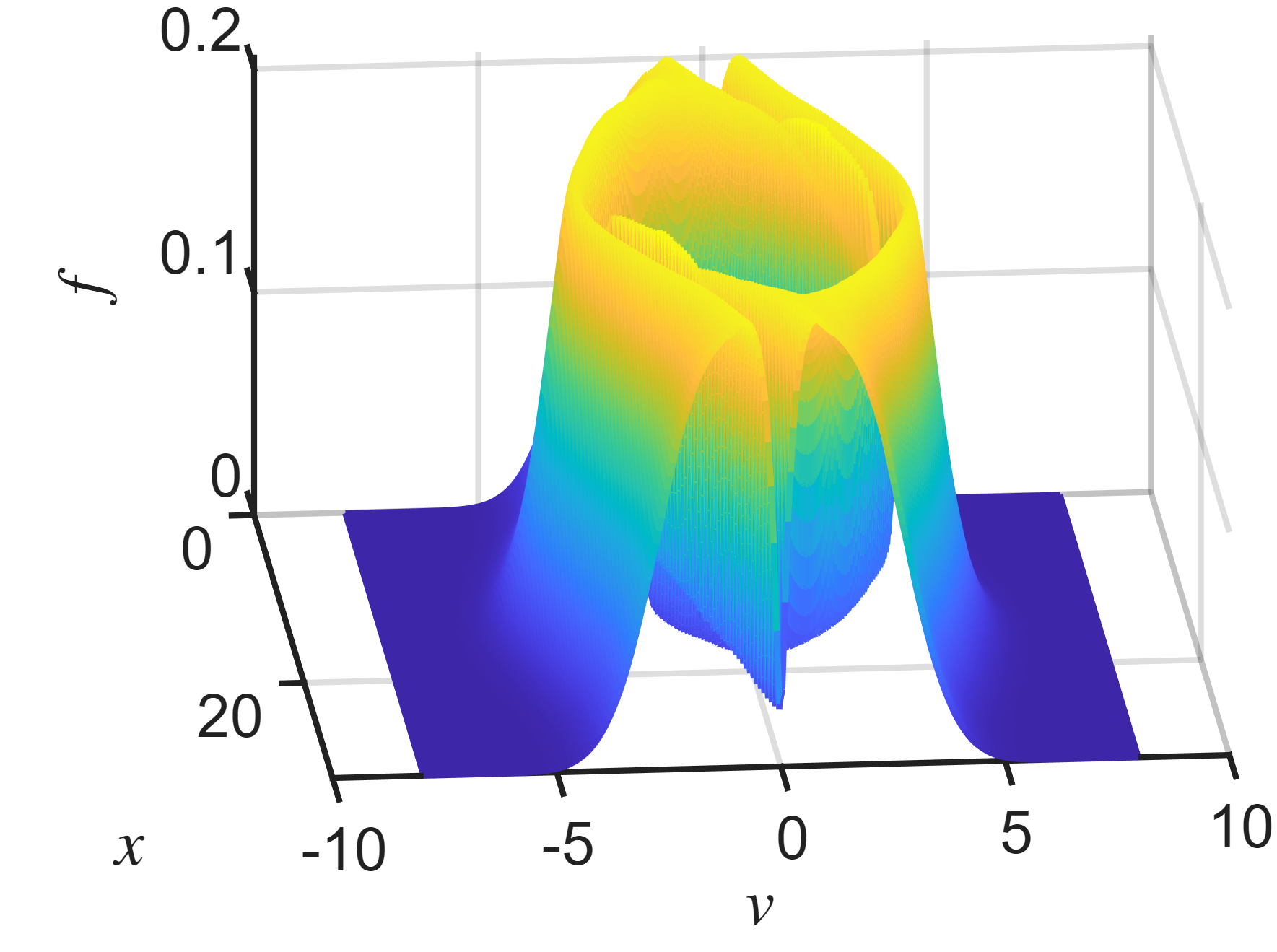}
    }
    \subfigure{
    \includegraphics[width=0.4\textwidth]{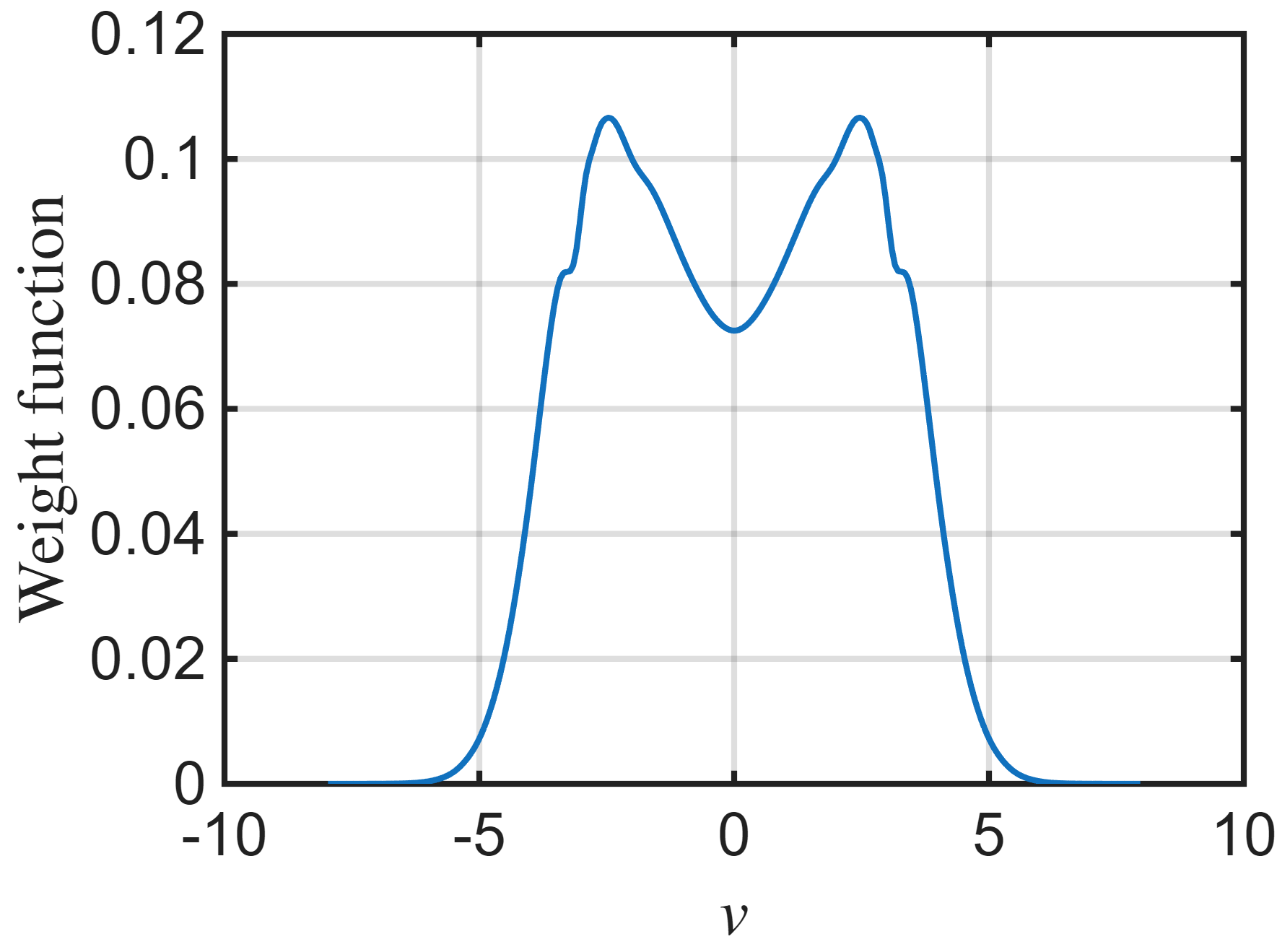}
    }
    \caption{(1D--1V Two-stream instability). The first row displays the time evolution of the electric energy, the rank history, and a contour plot of the numerical solution at $t=40$, respectively. In the second row, we present a mesh plot of the distribution function and the corresponding adaptive weight function selected by the LoMaC projection at $t=40$.}
\label{fig:TSI_results}
\end{figure}

\end{example}

\begin{example}(1D--1V Bump-on-tail instability). Consider the initial condition
\begin{equation*}
\begin{split}
f(x,v,t=0) = (1+\alpha\cos(kx))\left(n_p\exp\left(-\frac{v^2}{2}\right) + n_b\exp\left(-\frac{(v-u)^2}{2v_t}\right)\right),
\end{split}
\end{equation*}
with $\Omega_{x}=[0,20\pi/3]$, $\Omega_{v}=[-13,13]$, where \(\alpha = 0.04\), \(k = 0.3\), \(n_p = 9/(10\sqrt{2\pi})\), \(n_b = 2/(10\sqrt{2\pi})\), \(u=4.5\), and \(v_t=0.5\). A mesh of $256\times256$ and a CFL of $5$ are used for the simulation. 

\Cref{fig:BOT_results} shows representative results for the bump-on-tail test. The first row displays the electric energy, the rank history, and a contour plot of the numerical solution at \(t=40\). The electric energy exhibits the expected linear growth followed by nonlinear saturation. The rank increases gradually as phase-space filamentation develops. The contour plot shows the asymmetric deformation induced by the tail perturbation. The method, again, captures high-velocity tails and nonlinear trapping structures. The second row shows a mesh plot of the distribution function and the corresponding adaptive weight selected by the LoMaC projection at \(t=40\). The distribution develops a localized nonlinear structure in phase space. The adaptive weight reflects this asymmetric velocity space profile, capturing both the dominant Maxwellian-like bulk and the secondary structure observed in the high-velocity tail while preserving the overall velocity decay. The conservation results are similar to those for the Landau damping test and are again omitted for conciseness.

\begin{figure}[!htbp]
\centering

    \hspace{-0.4cm}
    \subfigure{
    \includegraphics[width=0.33\textwidth]{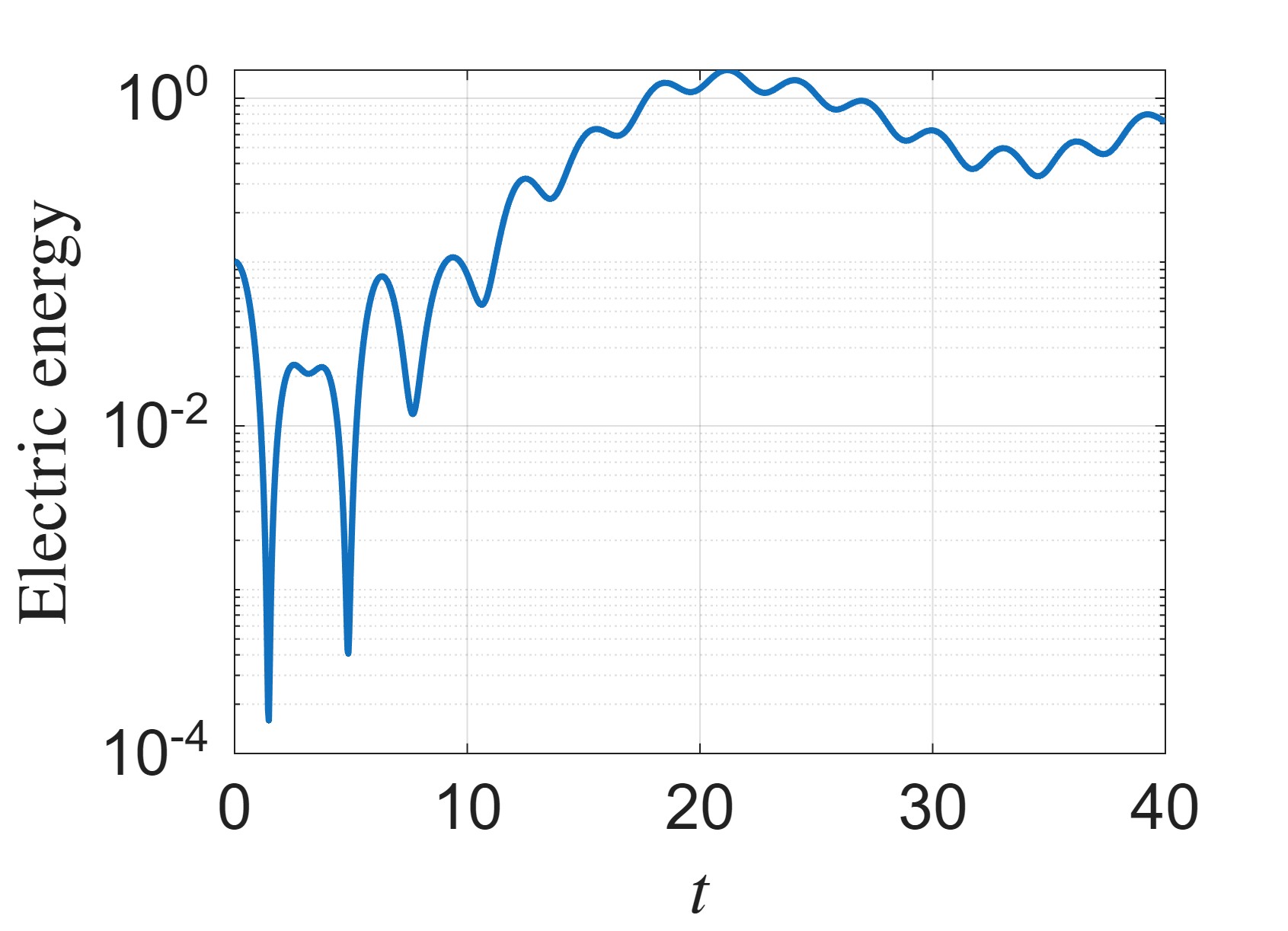}
    }
    \hspace{-0.55cm}
    \subfigure{
    \includegraphics[width=0.33\textwidth]{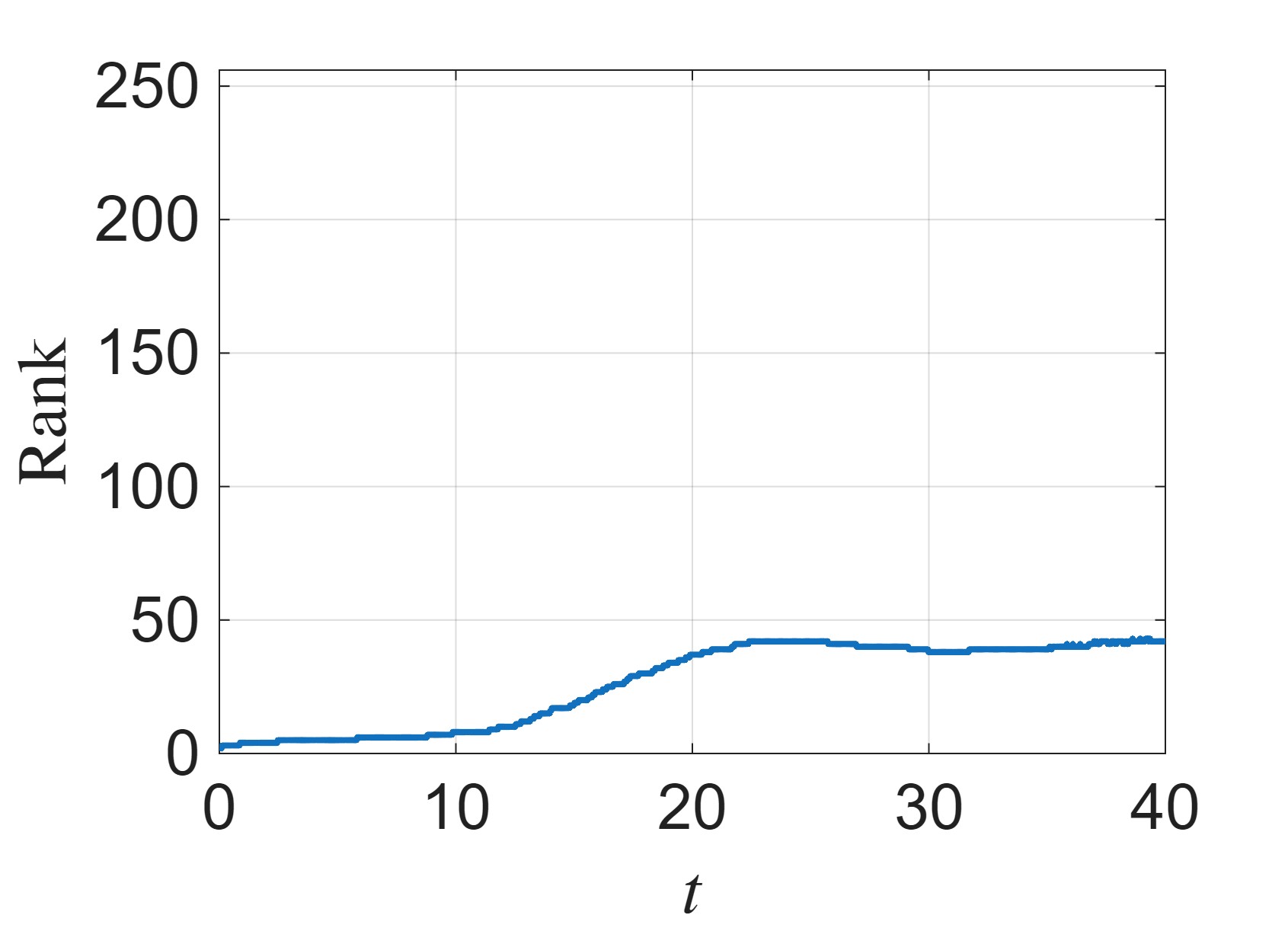}
    }
    \hspace{-0.55cm}
    \subfigure{
    \includegraphics[width=0.33\textwidth]{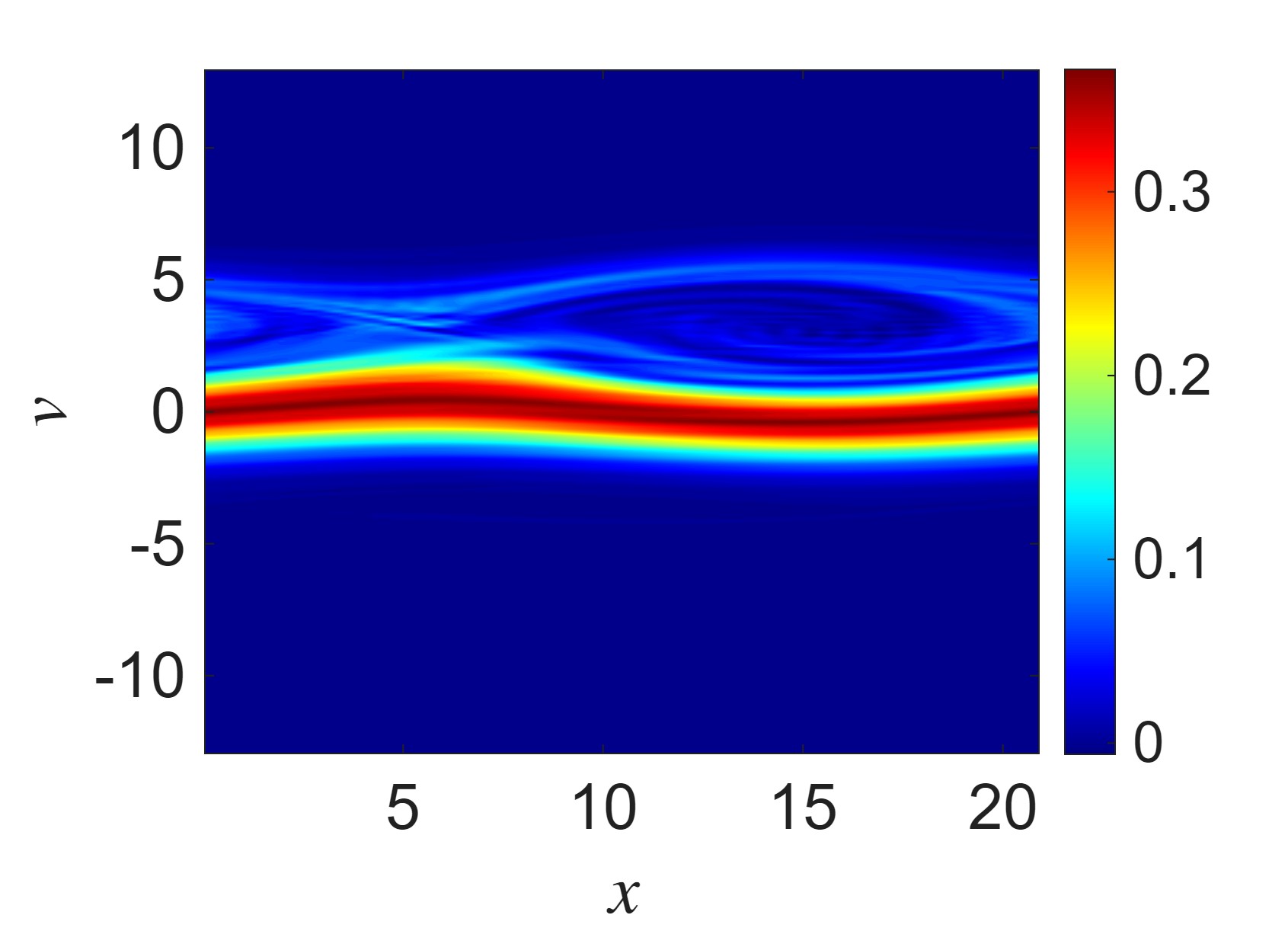}
    }

    \hspace{-0.4cm}
    \subfigure{
    \includegraphics[width=0.4\textwidth]{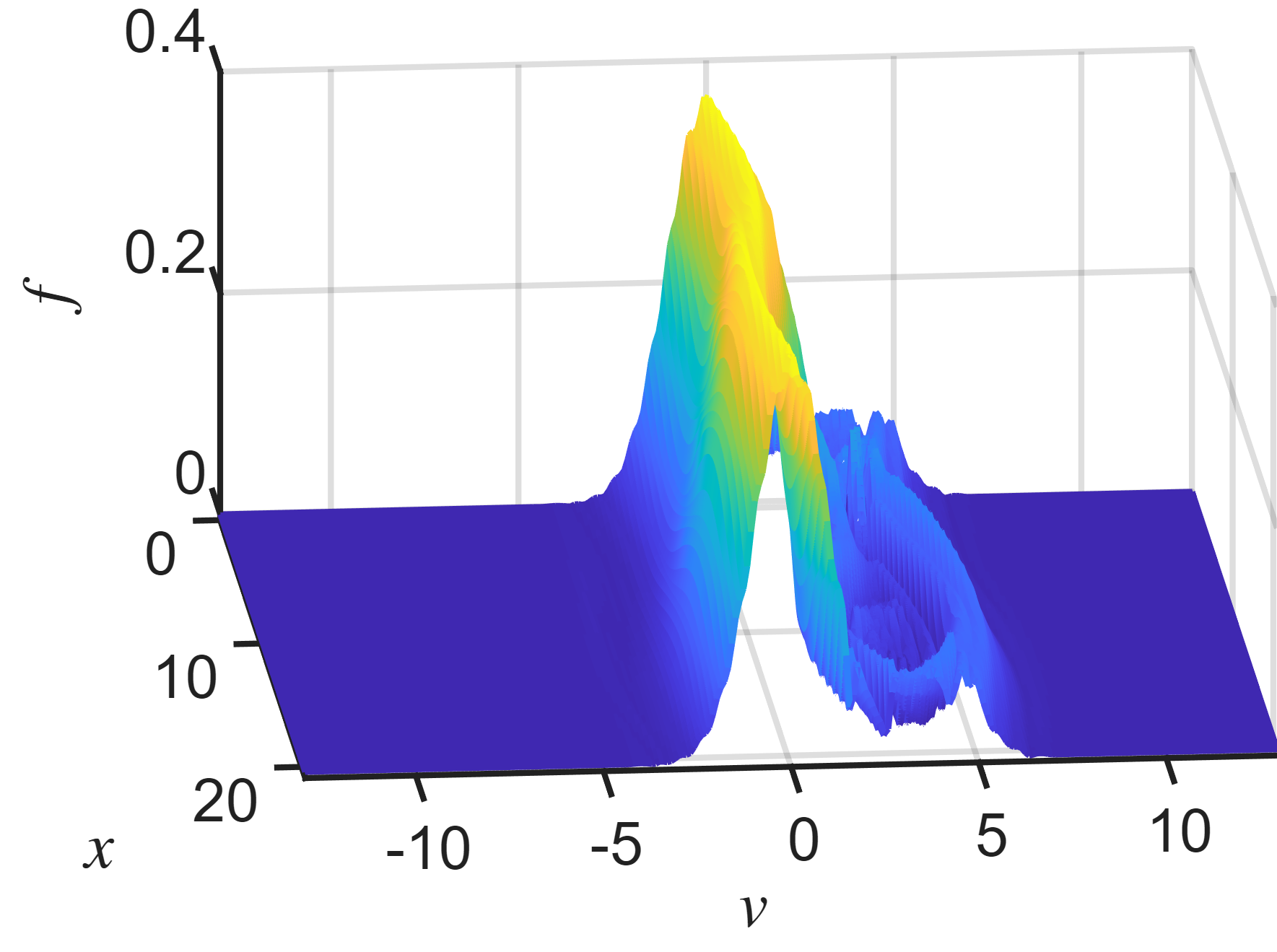}
    }
    \subfigure{
    \includegraphics[width=0.4\textwidth]{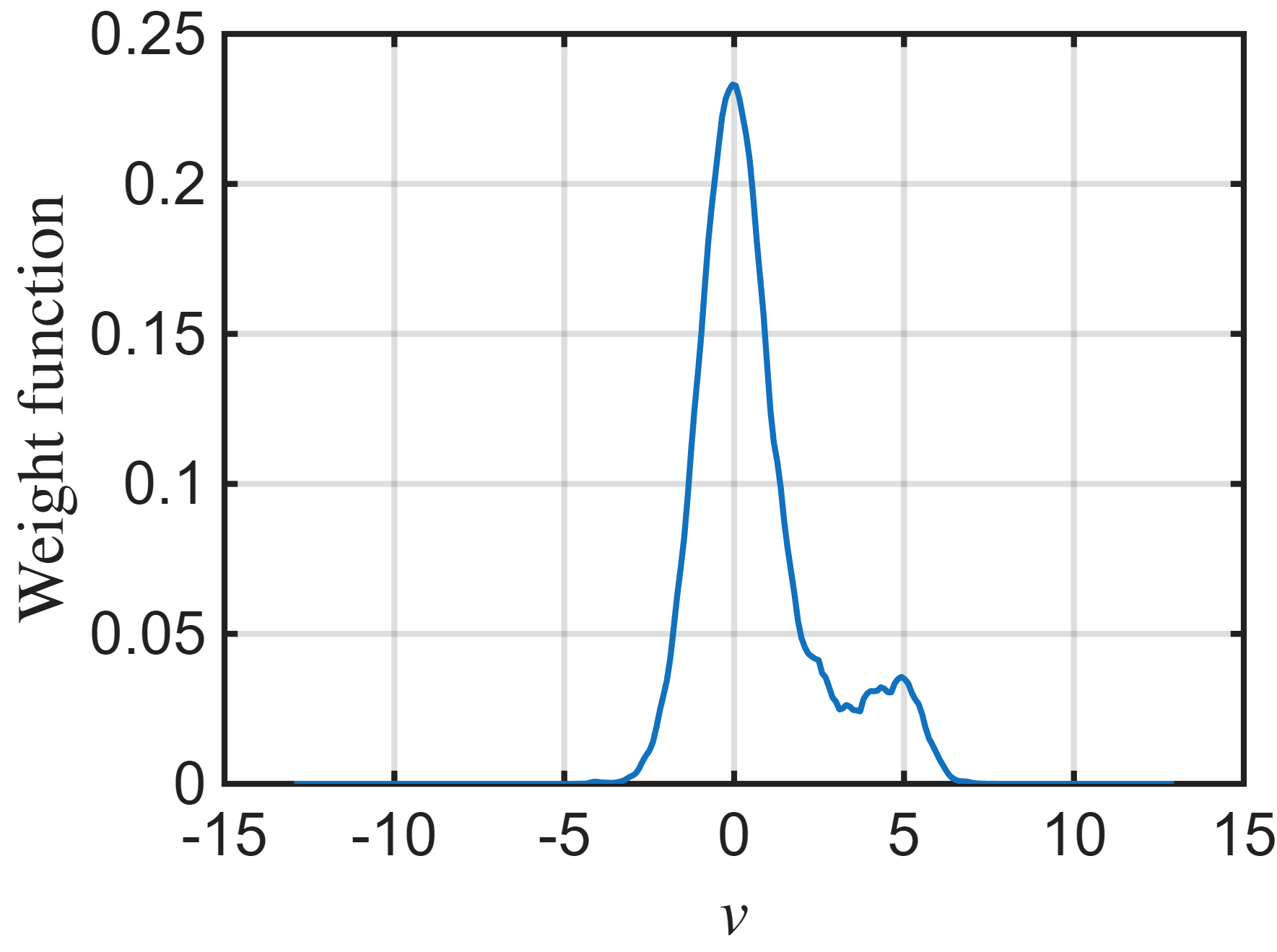}
    }
    \caption{(1D--1V Bump-on-tail instability). The first row displays the time evolution of the electric energy, the rank history, and a contour plot of the numerical solution at $t=40$, respectively. In the second row, we present a mesh plot of the distribution function and the corresponding adaptive weight function selected by the LoMaC projection at $t=40$.}
\label{fig:BOT_results}
\end{figure}

\end{example}

\subsection{2D--2V Examples}\label{sec:2D2V_VP_tests}
In this subsection, we present two 2D--2V benchmark tests: two-stream instability
and bump-on-tail instability. The computations use a mesh of $128^4$ with a CFL
number of 5.
Through these tests, we investigate the robustness, adaptive-rank behavior, and
conservation properties of the proposed scheme.
\begin{example}(2D--2V two-stream instability, \cite{kormann2015semi}). Consider the initial condition
\begin{equation*}
\begin{aligned}
    f(\bm{x},\bm{v},0) &=
    \frac{1}{(2\sqrt{2\pi})^2}
    \left(1 + \alpha \sum_{\mu=1}^{2} \cos\!\big(k x^{(\mu)}\big)\right) \\
    &\quad \times \prod_{\mu=1}^{2}
    \Bigl[
        \exp\!\left(-\tfrac{(v^{(\mu)}-v_0)^2}{2}\right)
        + \exp\!\left(-\tfrac{(v^{(\mu)}+v_0)^2}{2}\right)
    \Bigr],
\end{aligned}
\end{equation*}
on $\Omega_{\bm x}=[0,10\pi]^2$ and $\Omega_{\bm v}=[-8,8]^2$,
where $k=0.2$, $\alpha=0.001$, and $v_0=2.4$. For this test, the HTACA tolerance is set to $\varepsilon_{\mathrm Base} = 5\times10^{-4}$.

\Cref{fig:2D2V_TSI_results} shows representative results for the 2D--2V two-stream instability. The first row displays the time evolution of the electric energy, the HT rank history, and a representative \((x^{(1)},v^{(1)})\)-slice of the distribution at \(t=40\). As expected, the electric energy undergoes rapid linear growth, followed by nonlinear saturation. The rank history shows that both the non-leaf ranks \(r_{12}\) and \(r_{34}\) and the leaf ranks \(r_1,\dots,r_4\) increase as nonlinear filamentation develops, with the leaf ranks growing more noticeably. The rightmost panel in the first row shows a two-dimensional slice of the four-dimensional distribution at \(x^{(2)} = x^{(2)}_1 = \Delta x^{(2)}/2\) and \(v^{(2)} = v^{(2)}_{70} = 11/16\). The characteristic vortex-like trapping structures and strong filamentation are observed, demonstrating that the method captures the nonlinear interaction. The second row shows the conservation properties. The relative or absolute errors in total mass, momentum, and total energy remain at the level of the nonlinear solver tolerance throughout the simulation. Here, we only show the error in the first component of momentum; the second component behaves similarly. Without LoMaC, the corresponding errors are \(\mathcal{O}(10^{-4})\) for mass and energy and \(\mathcal{O}(10^{-2})\) for momentum. The last row shows a representative \((v^{(1)},v^{(2)})\)-slice of the distribution and the corresponding adaptive weight selected by the LoMaC projection at \(t=40\). The slice exhibits multiple localized peaks in velocity space, reflecting the development of nontrivial velocity space structures in the nonlinear regime. The adaptive weight captures the dominant multi-peak profile and the decay in the velocity domain.

\begin{figure}[!htbp]
\centering

    \hspace{-0.4cm}
    \subfigure{
    \includegraphics[width=0.32\textwidth]{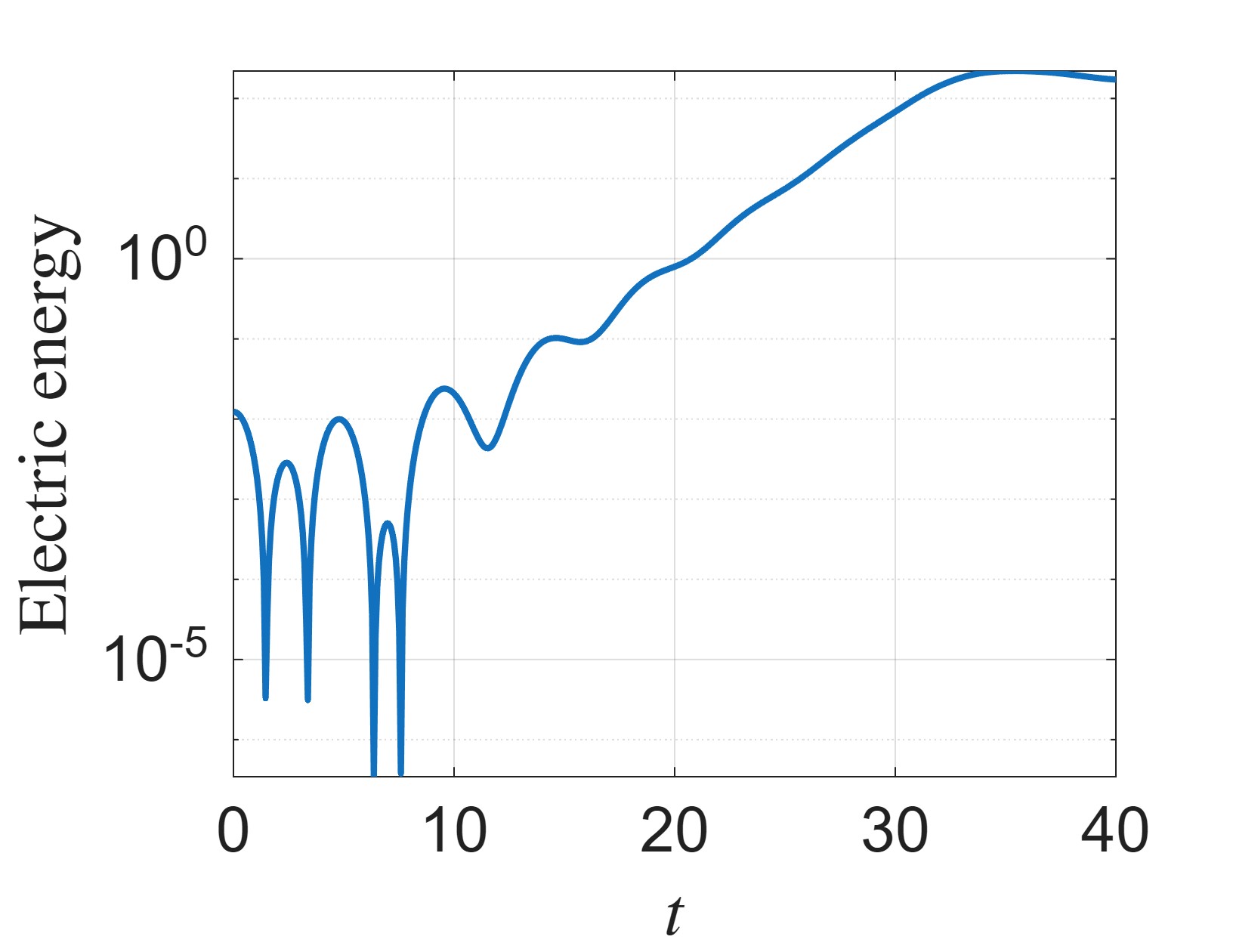}
    }
    \hspace{-0.55cm}
    \subfigure{
    \includegraphics[width=0.32\textwidth]{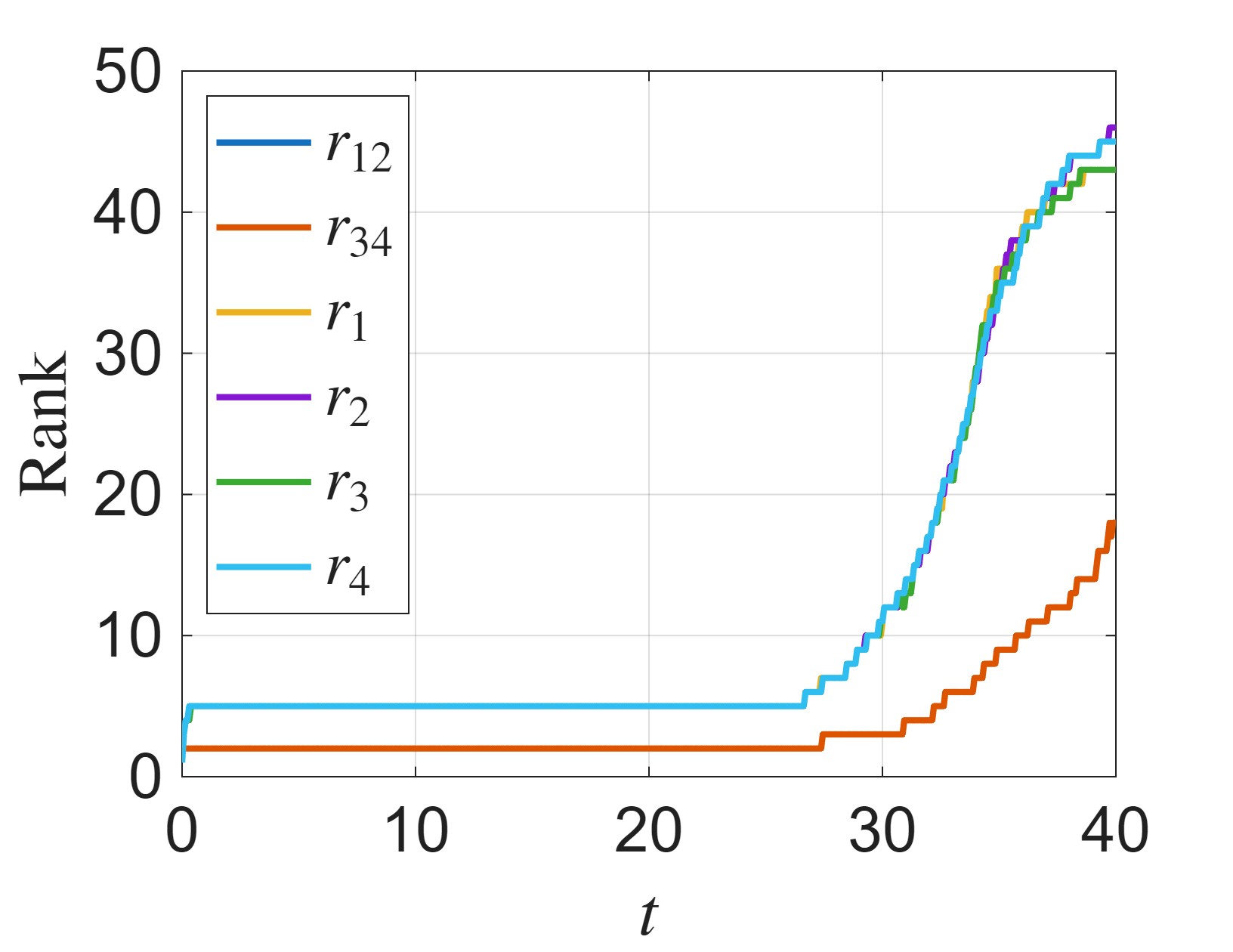}
    }
    \hspace{-0.55cm}
    \subfigure{
    \includegraphics[width=0.32\textwidth]{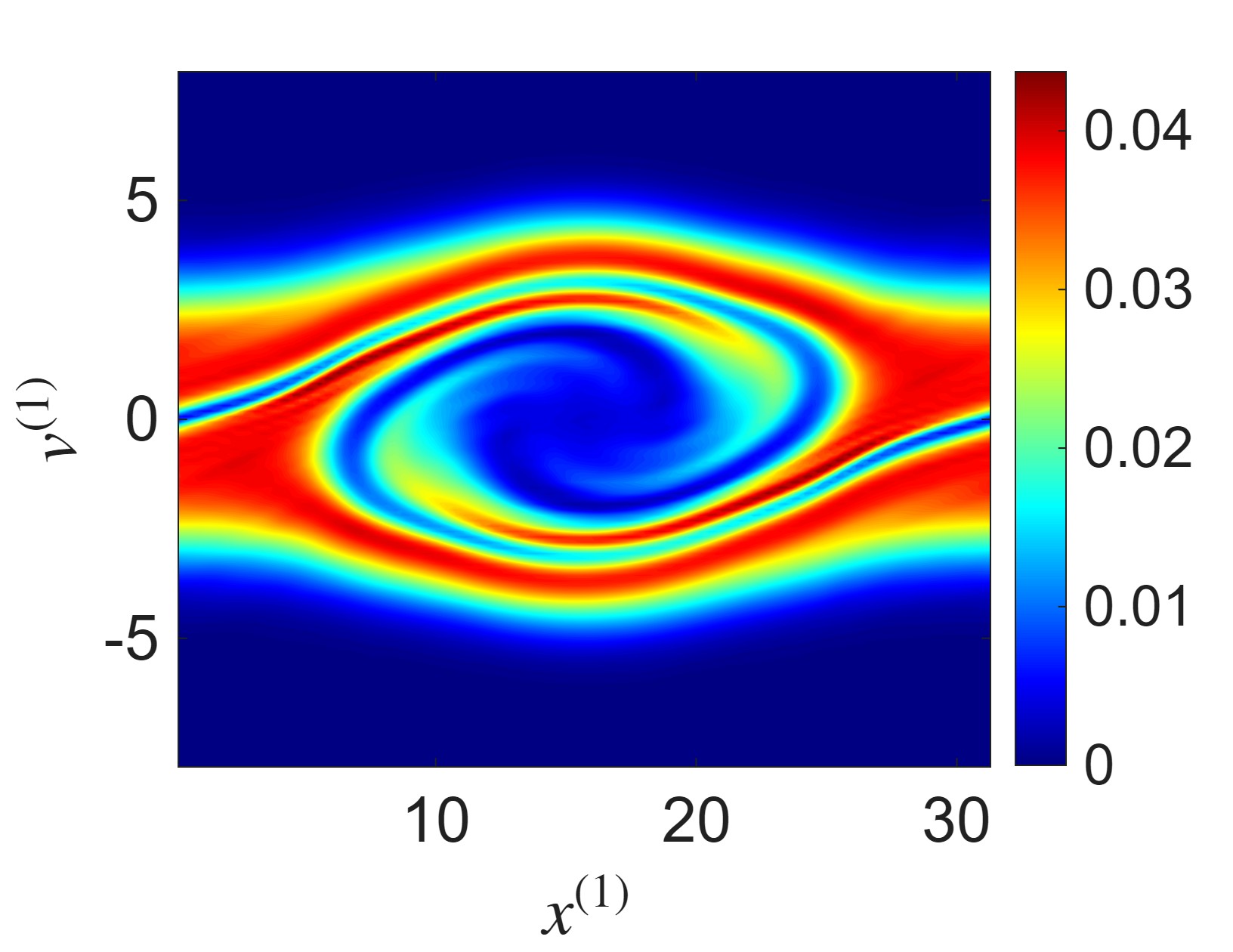}
    }

    \vspace{-0.25cm}

    \hspace{-0.4cm}
    \subfigure{
    \includegraphics[width=0.32\textwidth]{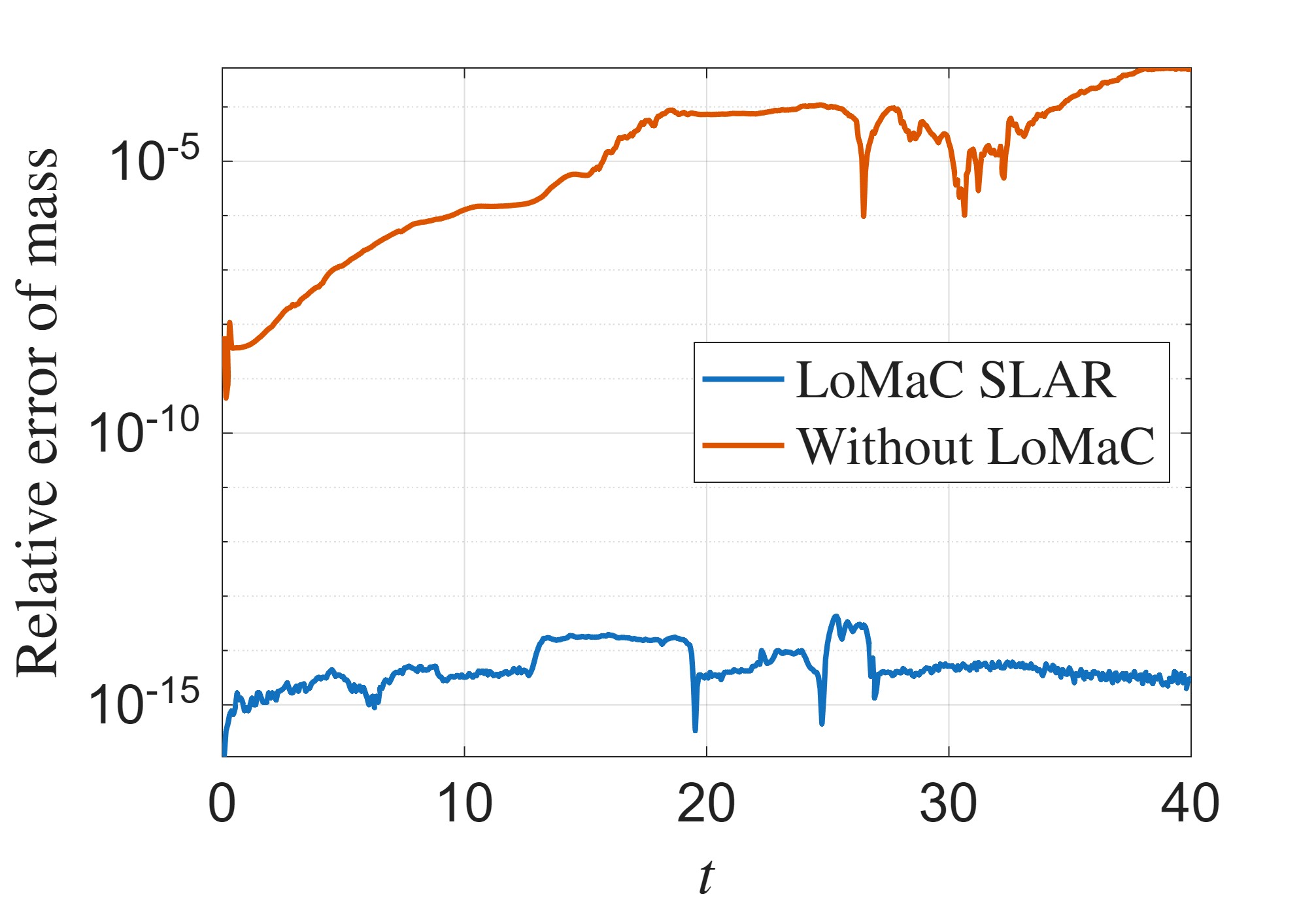}
    }
    \hspace{-0.55cm}
    \subfigure{
    \includegraphics[width=0.32\textwidth]{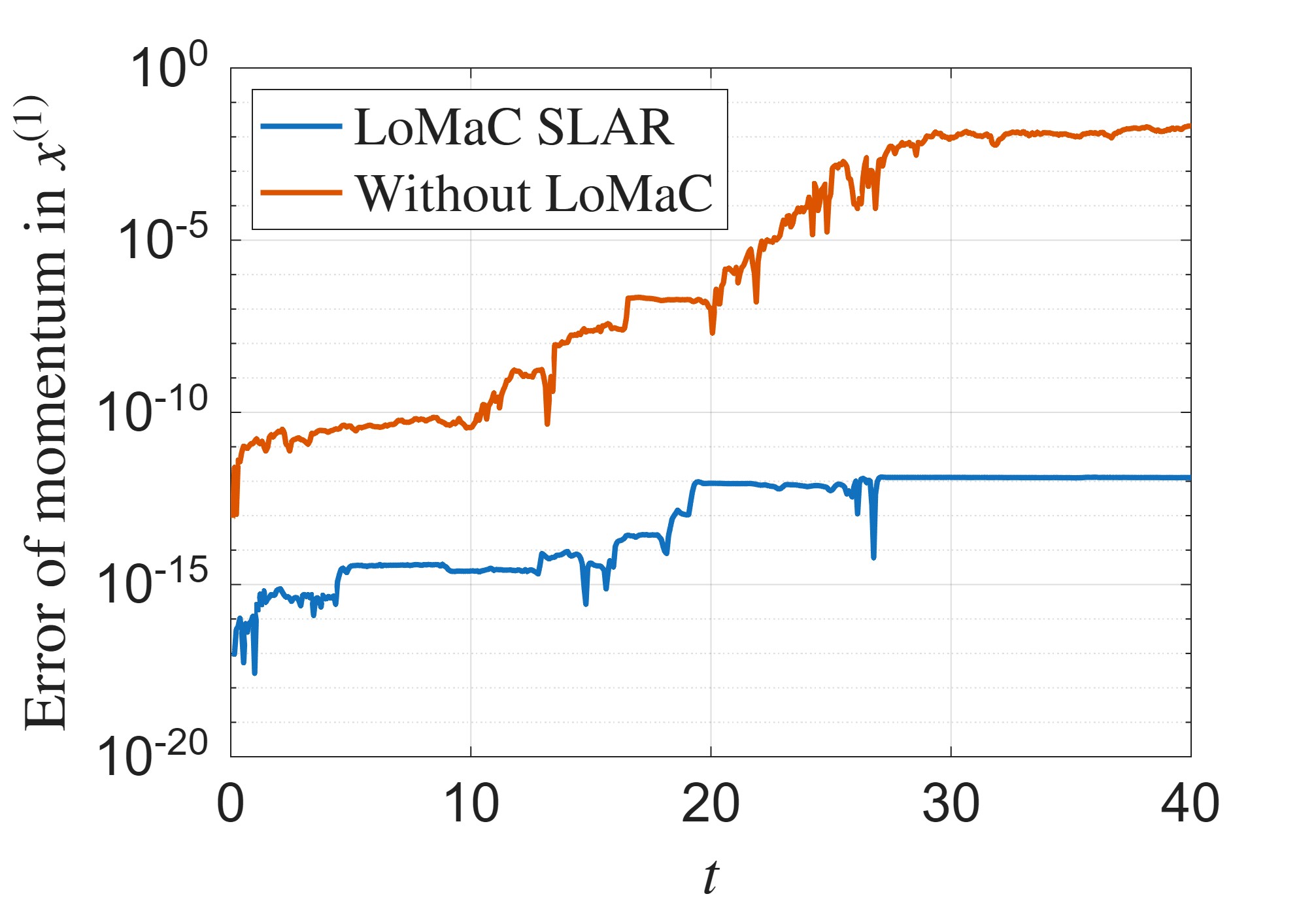}
    }
    \hspace{-0.55cm}
    \subfigure{
    \includegraphics[width=0.32\textwidth]{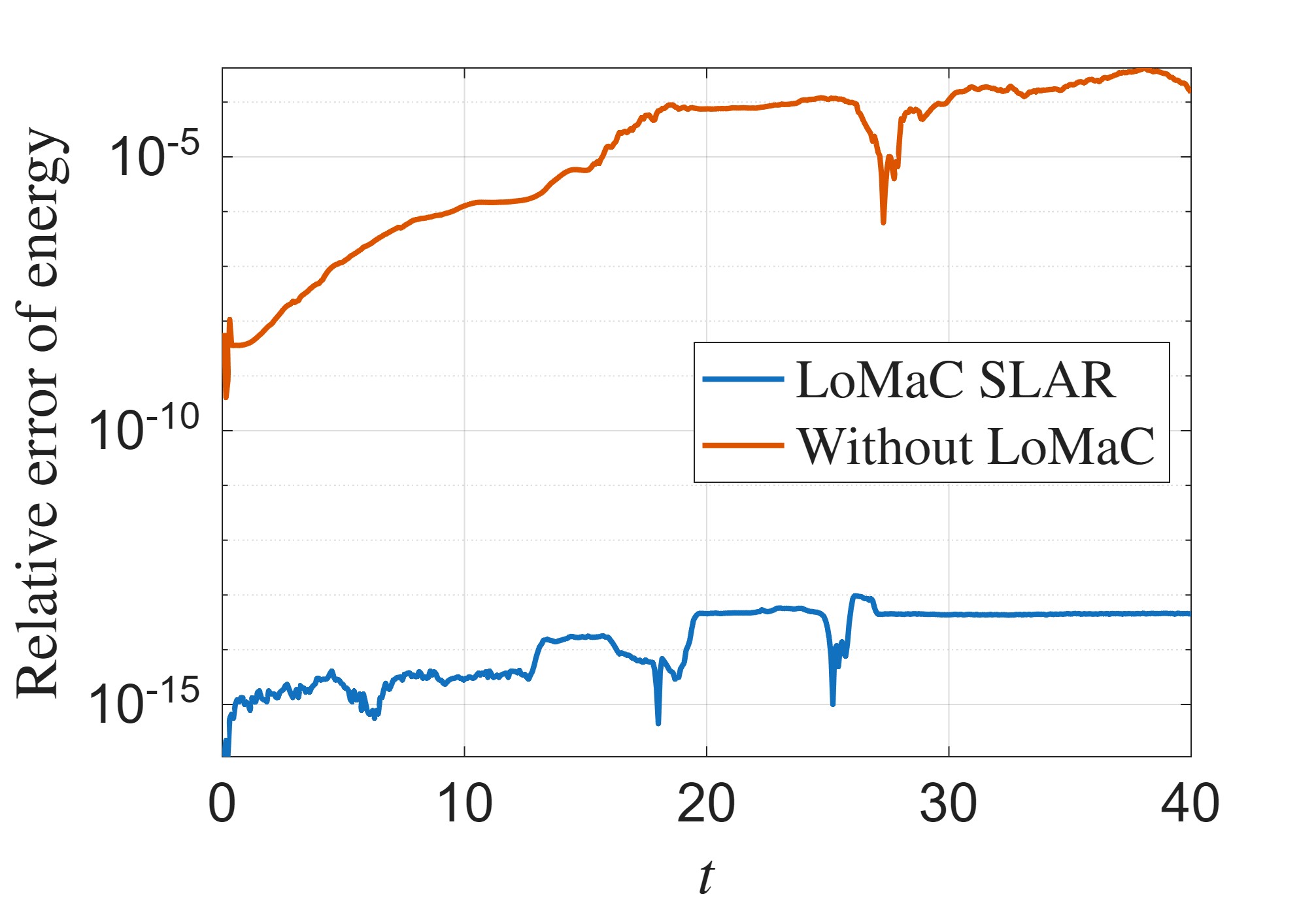}
    }

    \vspace{-0.5cm}

    \hspace{-0.4cm}
    \subfigure{
    \includegraphics[width=0.45\textwidth]{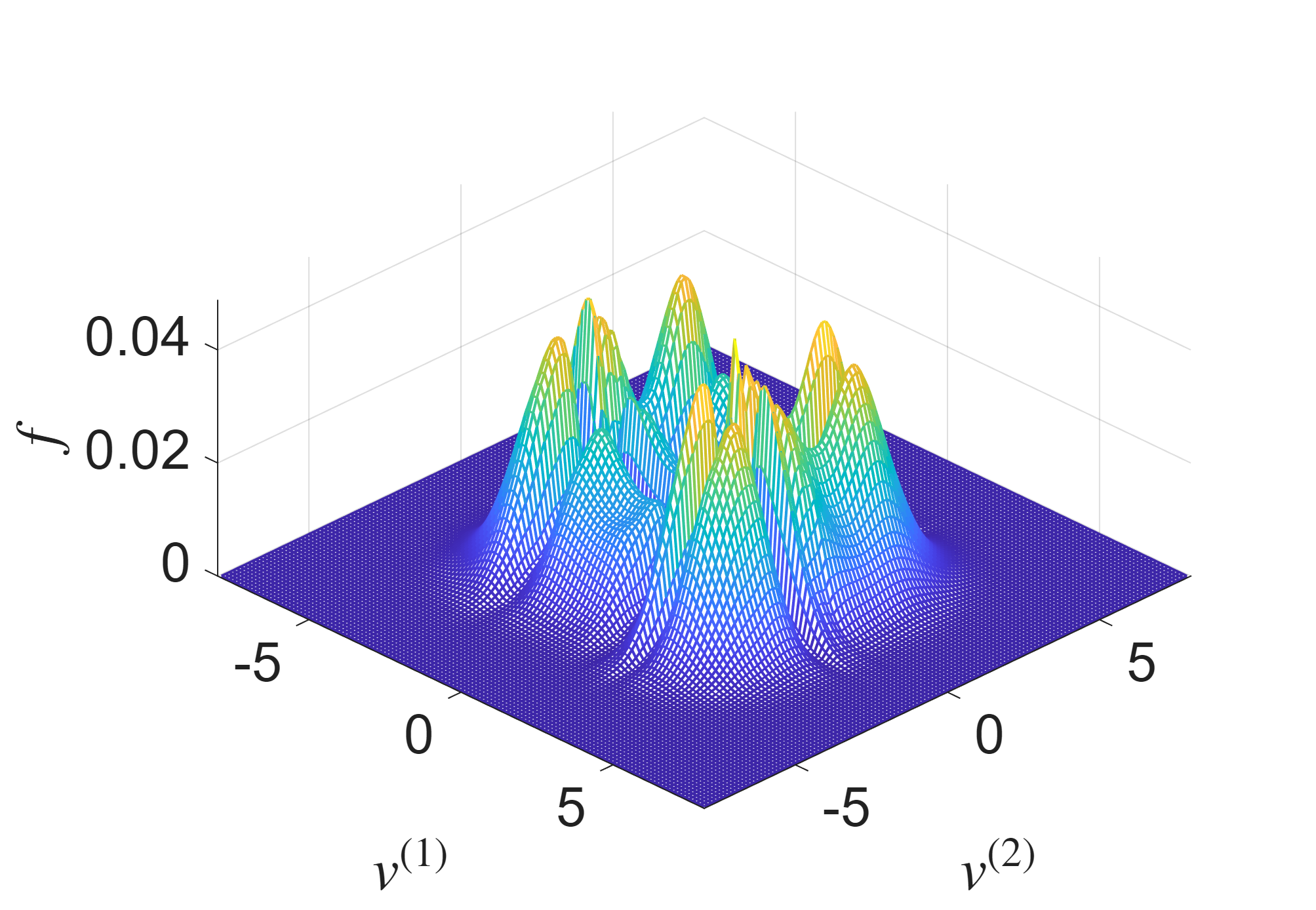}
    }
    \hspace{-0.55cm}
    \subfigure{
    \includegraphics[width=0.45\textwidth]{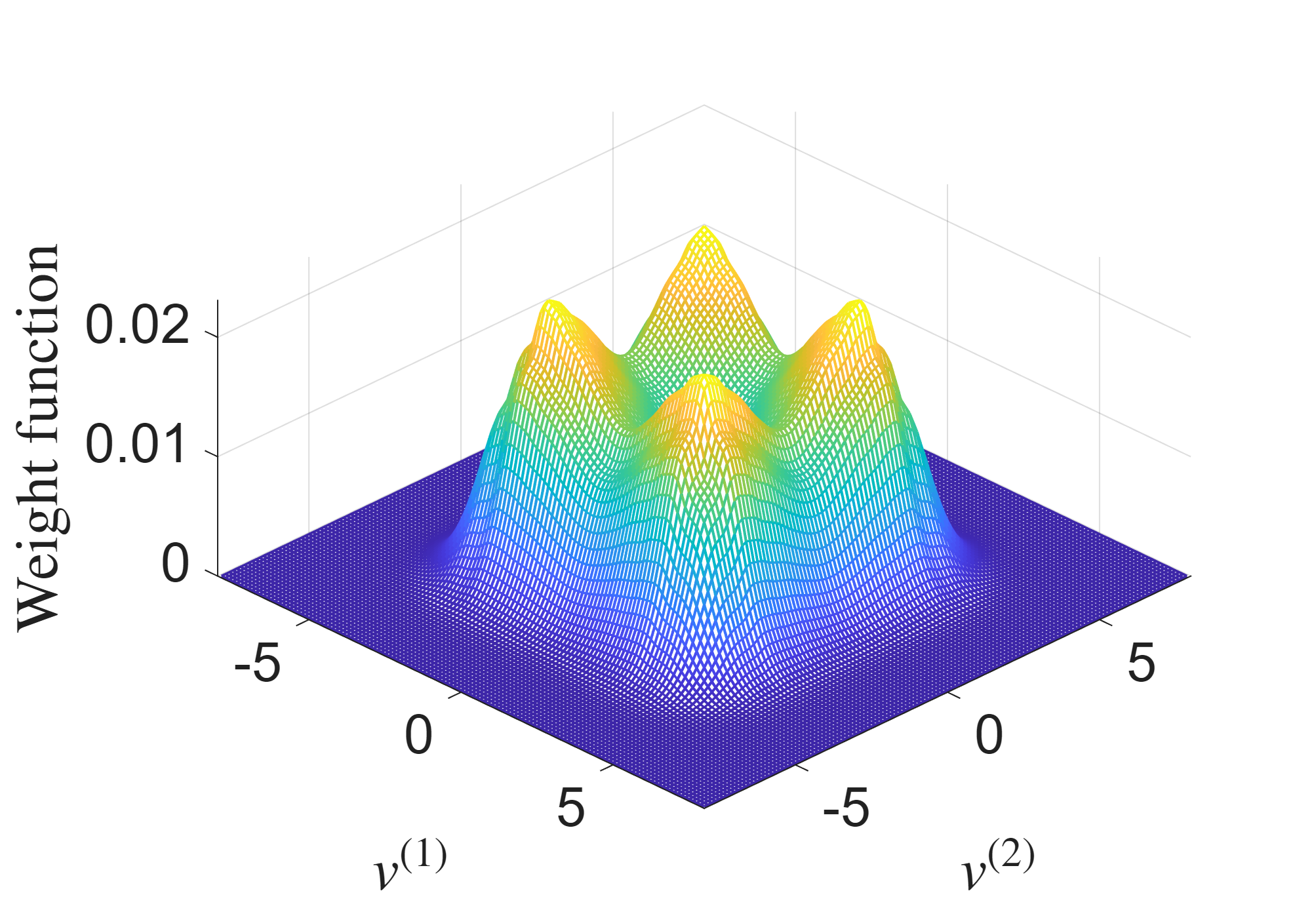}
    }

 \caption{(2D--2V two-stream instability). The first row shows the time evolution of the electric energy, the HT rank history, and a contour plot of the distribution slice at \(x^{(2)} = x^{(2)}_1 = \Delta x^{(2)}/2\) and \(v^{(2)} = v^{(2)}_{70} = 11/16\) and $t=40$. The second row shows the relative errors in mass, the absolute error in the first component of momentum and the relative error in the total energy. The last row shows the mesh plots of a selected $(v^{(1)},v^{(2)})$-slice of the distribution function and the corresponding adaptive weight function selected by the LoMaC projection at $t=40$.}
\label{fig:2D2V_TSI_results}
\end{figure}

\begin{figure}[!htbp]
\centering

    \hspace{-0.4cm}
    \subfigure{
    \includegraphics[width=0.33\textwidth]{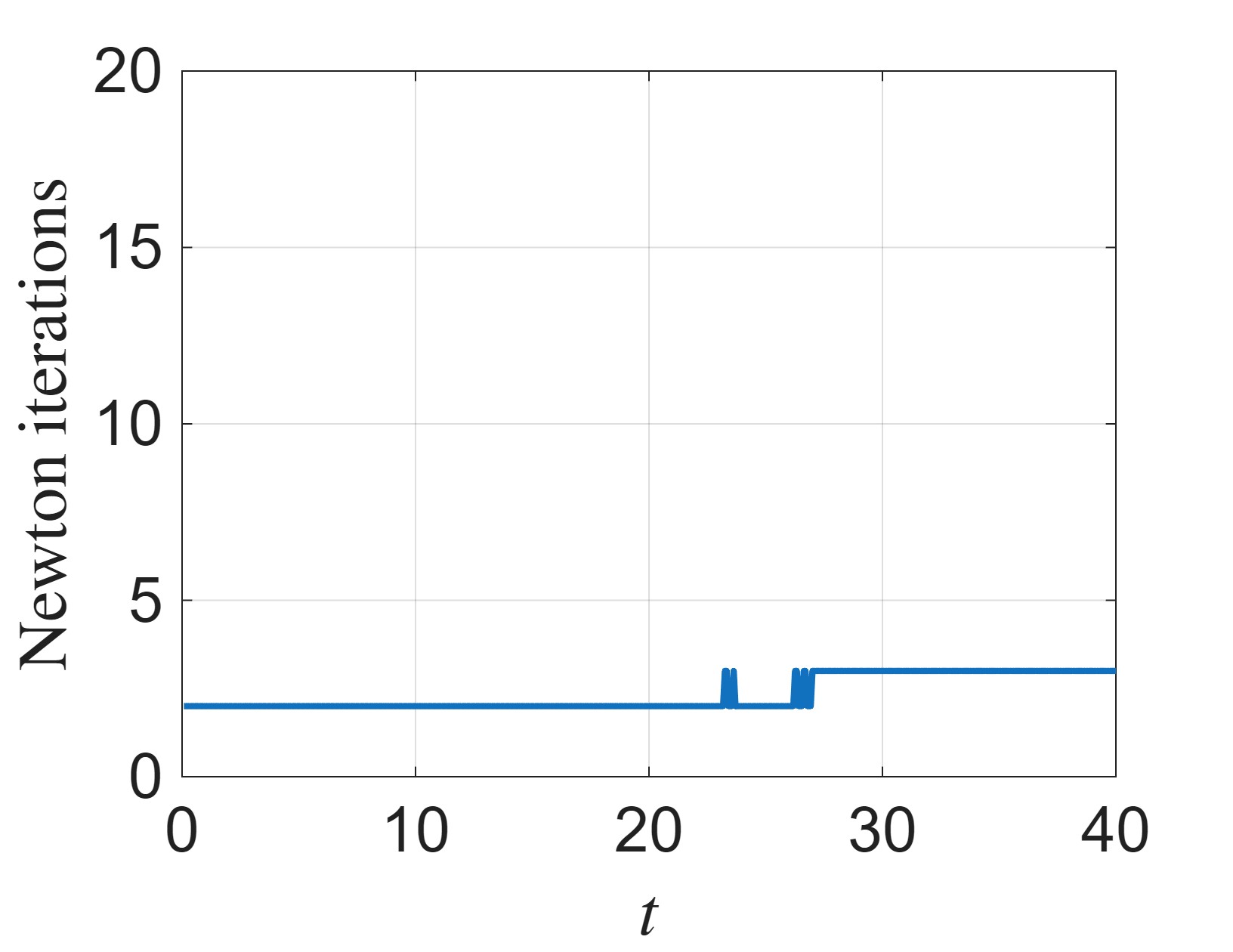}
    }
    \hspace{-0.55cm}
    \subfigure{
    \includegraphics[width=0.33\textwidth]{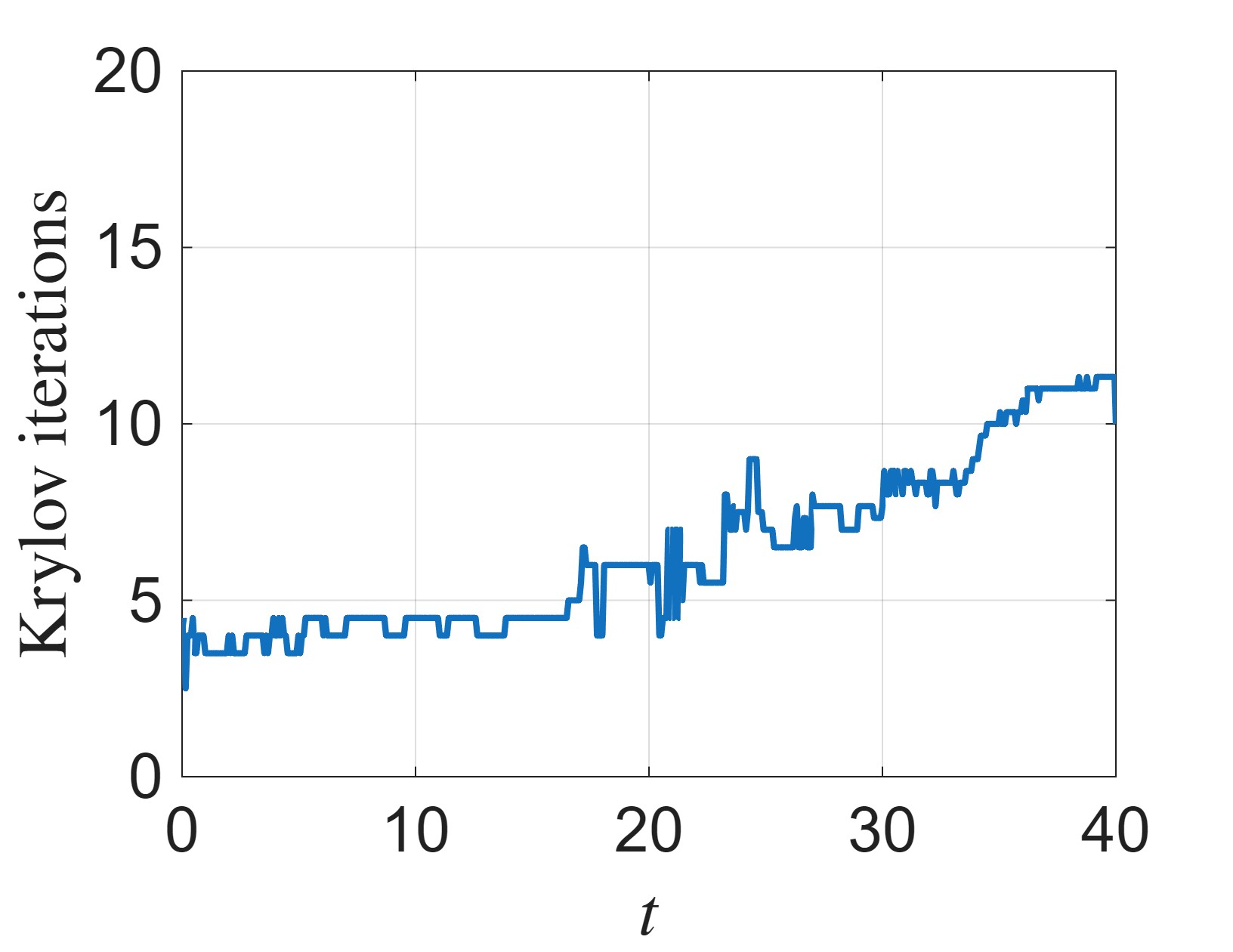}
    }
    \hspace{-0.55cm}
    \subfigure{
    \includegraphics[width=0.35\textwidth]{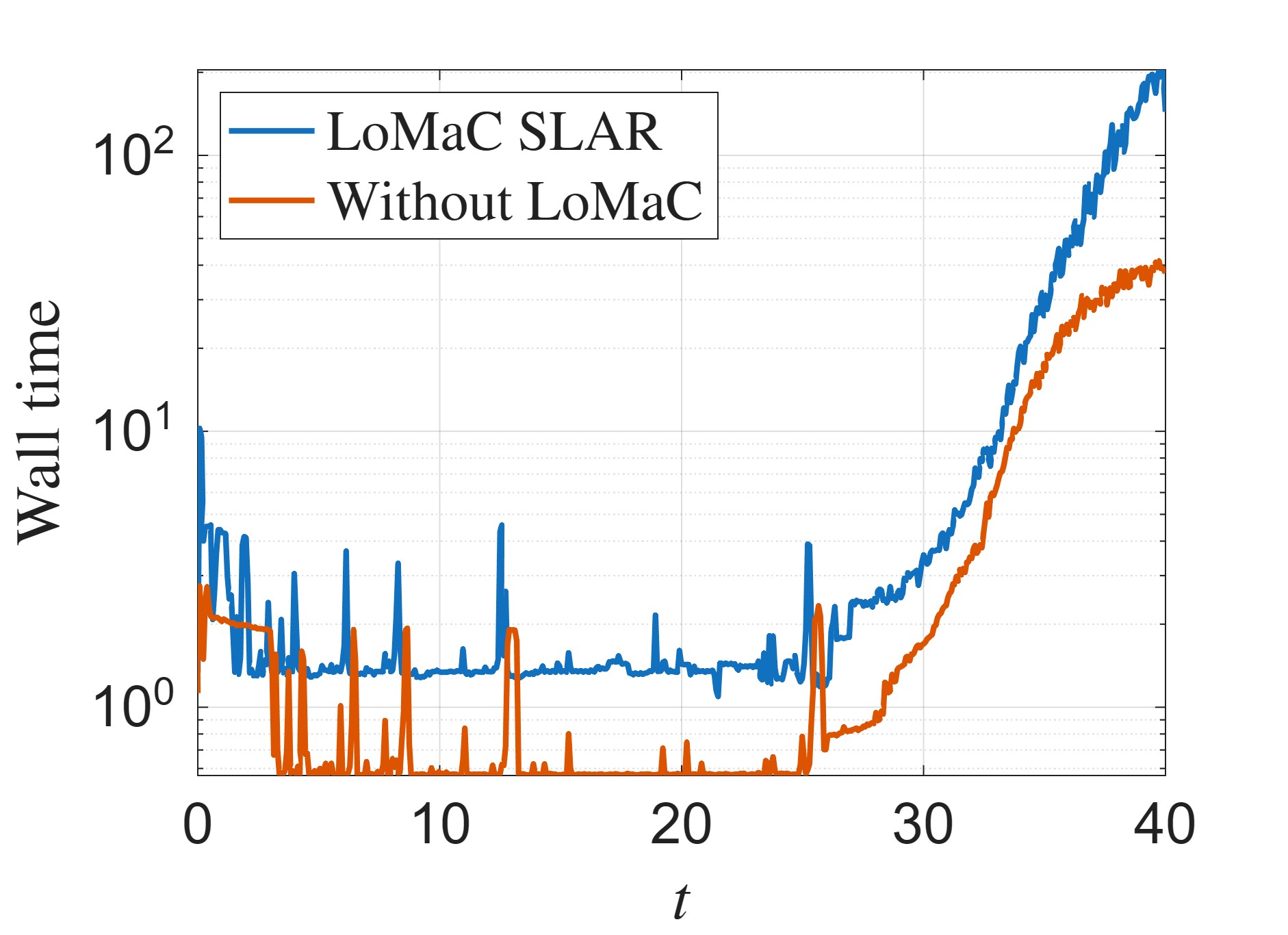}
    }

 \caption{(2D--2V two-stream instability). JFNK solver and performance diagnostics. The panels show the number of Newton iterations, number of Krylov iterations, and wall-clock time in seconds.}
  \label{fig:2D2V_TSI_solver}
\end{figure}

\Cref{fig:2D2V_TSI_solver} shows the performance of the nonlinear macroscopic solver and the wall-clock time history of the numerical scheme with and without LoMaC. The average Newton iteration count remains nearly constant, while the Krylov iteration count increases only moderately as finer structures develop, indicating robust JFNK performance in this high-dimensional setting. The right panel compares the wall-clock time with and without the LoMaC correction. At later times, the LoMaC-corrected SLAR solver becomes increasingly expensive, whereas the macroscopic solver iteration counts increase only mildly. This suggests that the additional cost is mainly caused by reduced compression efficiency in the adaptive-rank representation, rather than by a substantial increase in the difficulty of the nonlinear macroscopic solve. This behavior is expected, since the original HTACA algorithm proposed in \cite{zheng2025semihighD} is sensitive to the rank of the solution, and the current MATLAB implementation does not include rank-level parallelization. Additionally, while the LoMaC correction successfully enforces the local conservation laws, it also leads to an increase in the ranks of the kinetic solution. This rank growth makes rank-level parallelism an important consideration for low-rank kinetic solvers and motivates our ongoing work on an enhanced HTACA algorithm.

\end{example}

\begin{example}(2D--2V bump-on-tail instability). Consider the initial condition
\begin{equation*}
\begin{aligned}
    f(\bm{x},\bm{v},0)
    &=
    \frac{1}{(2\sqrt{2\pi})^2}
    \left(
        1+\alpha\prod_{\mu=1}^{2}\cos\!\big(kx^{(\mu)}\big)
    \right)
    \\
    &\quad \times
    \left[
        n_p
        \exp\!\left(
            -\frac{(v^{(1)})^2+(v^{(2)})^2}{2}
        \right)
        +
        n_b
        \exp\!\left(
            -\frac{(v^{(1)}-u)^2+(v^{(2)})^2}{2v_t}
        \right)
    \right],
\end{aligned}
\end{equation*}
with $\Omega_{\bm{x}} = [0,20\pi/3]^2$ and $\Omega_{\bm{v}} = [-13,13]^2$, where \(k = 0.3\), \(\alpha = 0.04\), \(n_p = 17/(40\pi)\), \(n_b = 3/(40\pi v_t)\), \(v_t = 2.4\), and \(u = 4.5\). Unlike the classical 1D--1V bump-on-tail test, the present 2D--2V configuration is constructed as a high-dimensional extension for which no direct reference solution or established benchmark
is available in the literature. For this test, the HTACA tolerance is set to $\varepsilon_{\mathrm Base} = 10^{-4}$.

\Cref{fig:2D2V_BOT_results} shows representative results for the 2D--2V bump-on-tail instability. The first row displays the time evolution of the electric energy, the HT rank history, and a representative \((x^{(1)},v^{(1)})\)-slice of the distribution at \(t=30\). The electric energy exhibits oscillations at early times, followed by a gradual growth as the nonlinear instability develops. The HT ranks increase throughout the simulation, with a particularly rapid growth of the non-leaf ranks \(r_{12}\) and \(r_{34}\) at later times, indicating the increasing complexity of the velocity space structures. The right panel in the first row shows an asymmetric deformation of the distribution caused by the drifting bump, together with the resulting nonlinear trapping structures. The second row shows the relative error in total mass, the absolute error in the first component of momentum, and the relative error in total energy. The errors remain at the level of the nonlinear solver tolerance throughout the simulation. Without LoMaC, the corresponding errors are \(\mathcal{O}(10^{-4})\) for mass and energy and \(\mathcal{O}(10^{-1})\) for momentum. The momentum drift without LoMaC is severe for this test. The last row shows a representative \((v^{(1)},v^{(2)})\)-slice of the distribution function and the corresponding adaptive weight selected by the LoMaC projection at \(t=30\). The adaptive weight closely follows the dominant velocity space profile of the selected distribution slice, capturing both the primary core and the secondary bump structure while preserving the decay in the velocity domain.

\begin{figure}[!htbp]
\centering

    \hspace{-0.4cm}
    \subfigure{
    \includegraphics[width=0.32\textwidth]{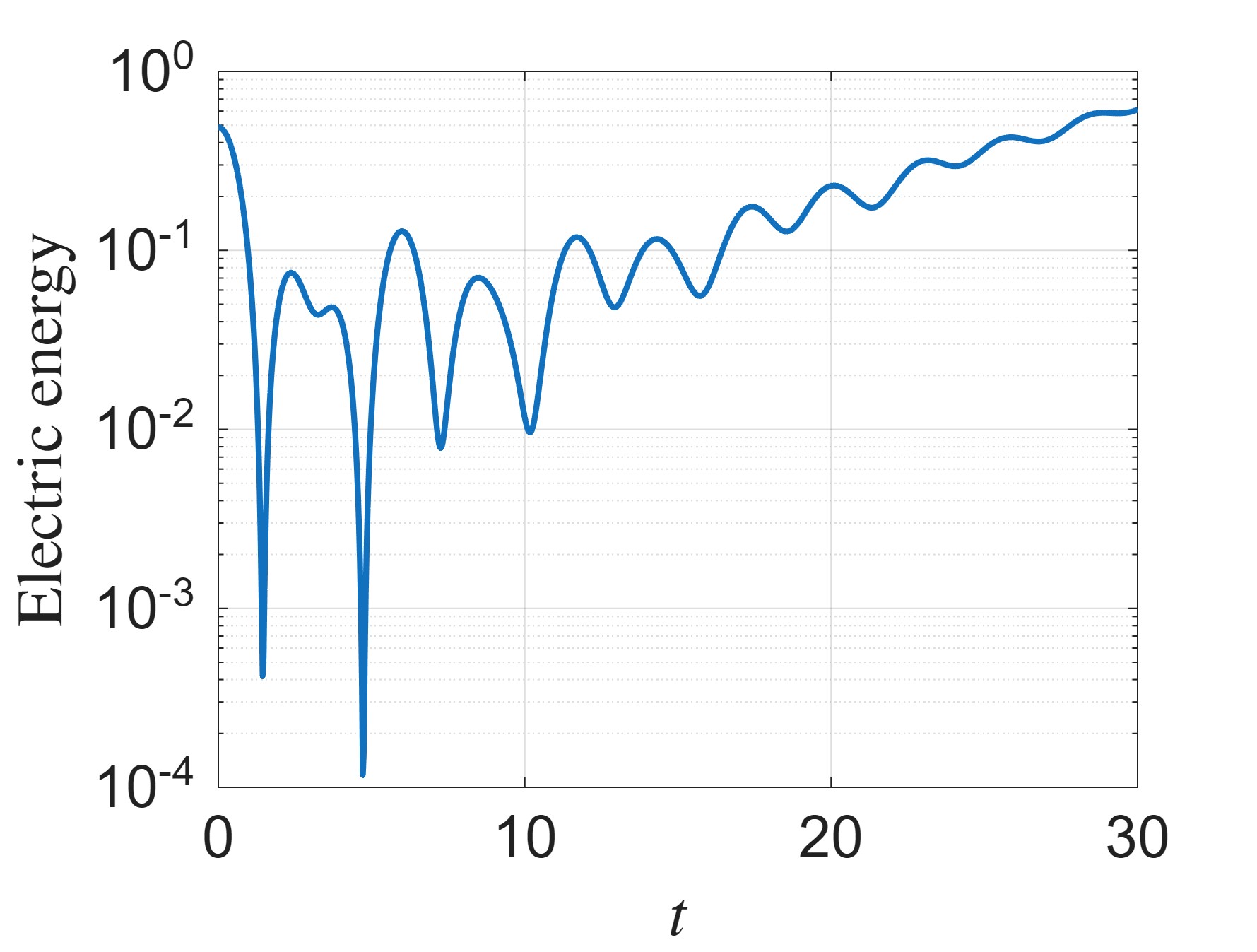}
    }
    \hspace{-0.55cm}
    \subfigure{
    \includegraphics[width=0.32\textwidth]{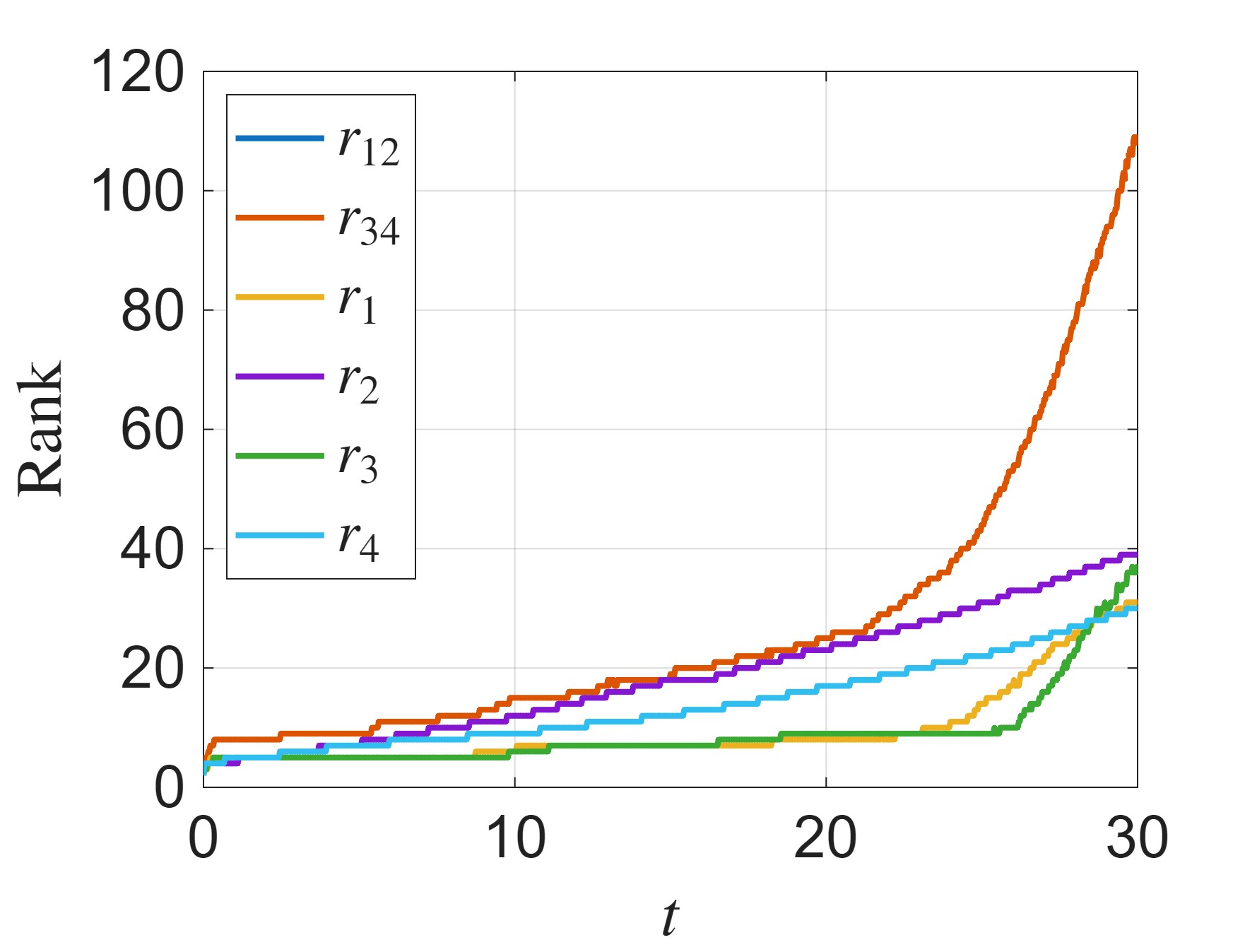}
    }
    \hspace{-0.55cm}
    \subfigure{
    \includegraphics[width=0.32\textwidth]{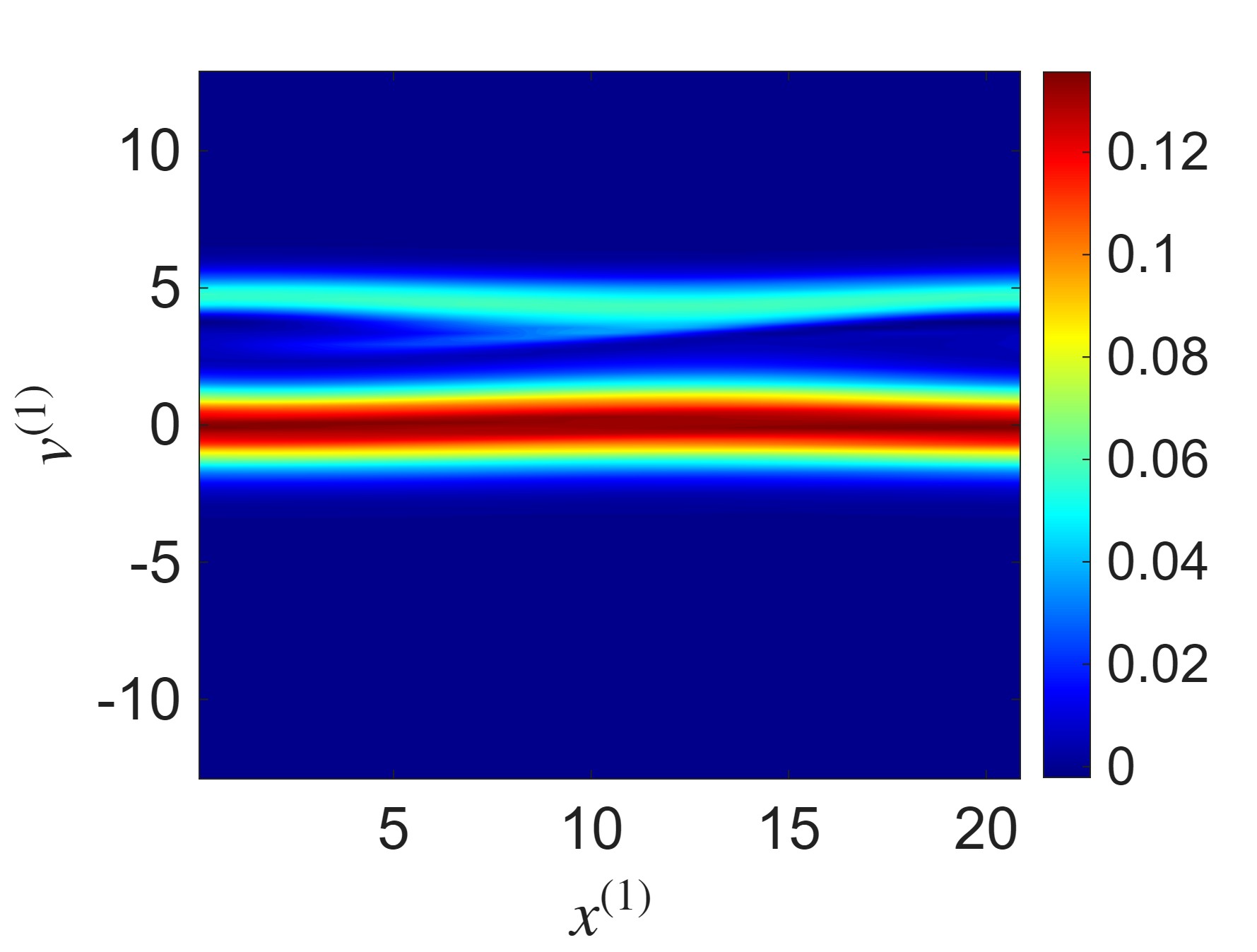}
    }

    \vspace{-0.25cm}

    \hspace{-0.4cm}
    \subfigure{
    \includegraphics[width=0.32\textwidth]{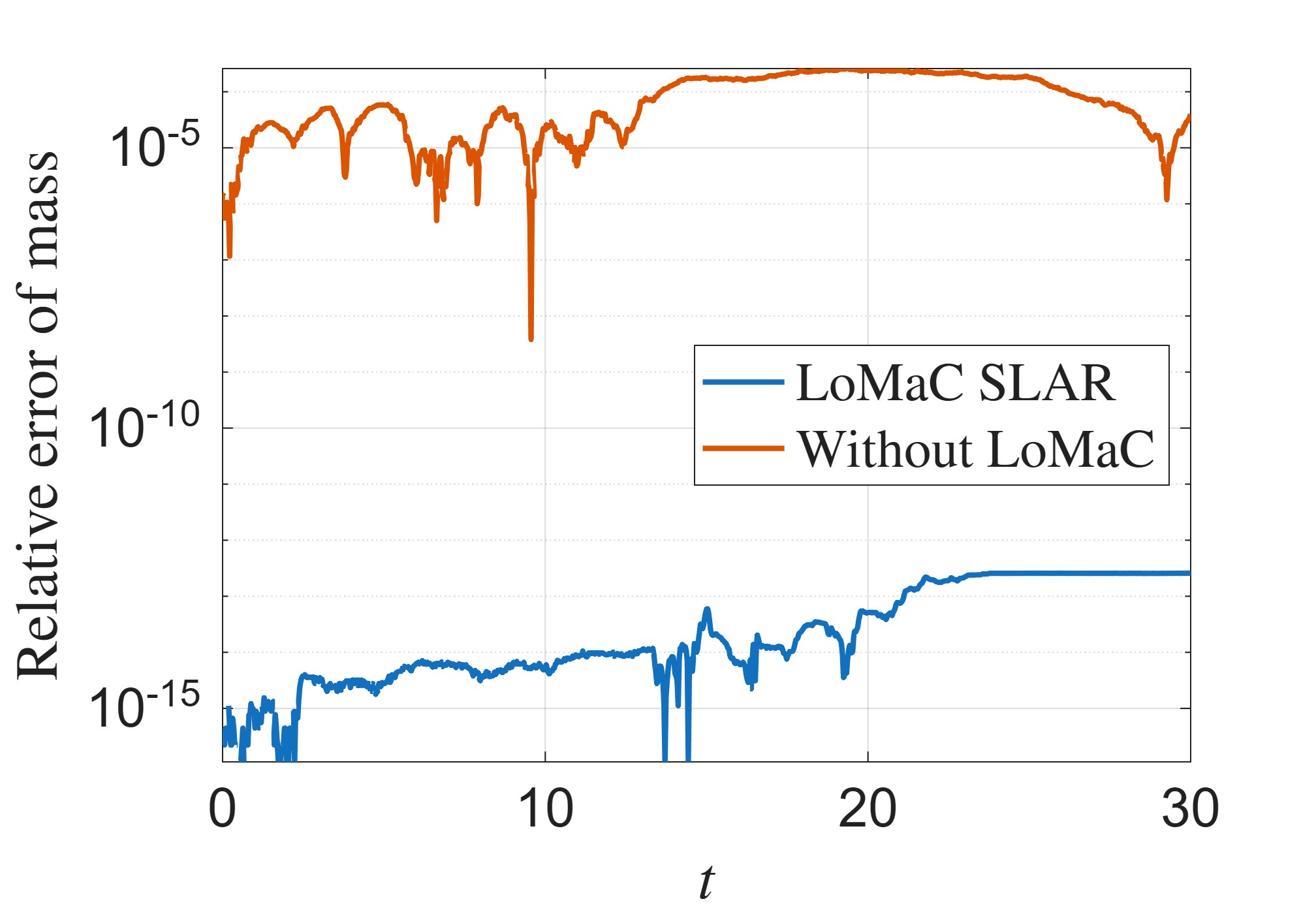}
    }
    \hspace{-0.55cm}
    \subfigure{
    \includegraphics[width=0.32\textwidth]{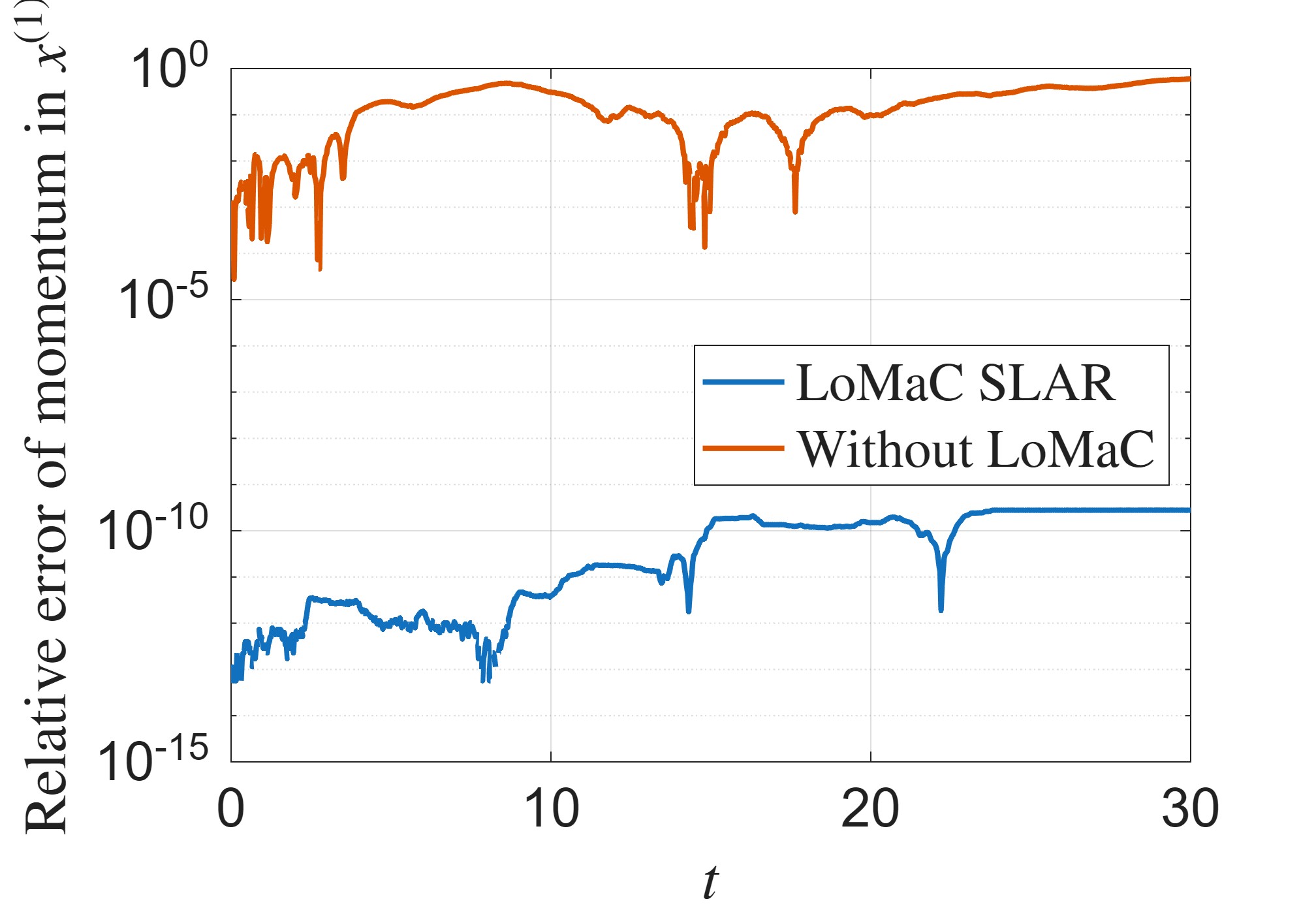}
    }
    \hspace{-0.55cm}
    \subfigure{
    \includegraphics[width=0.32\textwidth]{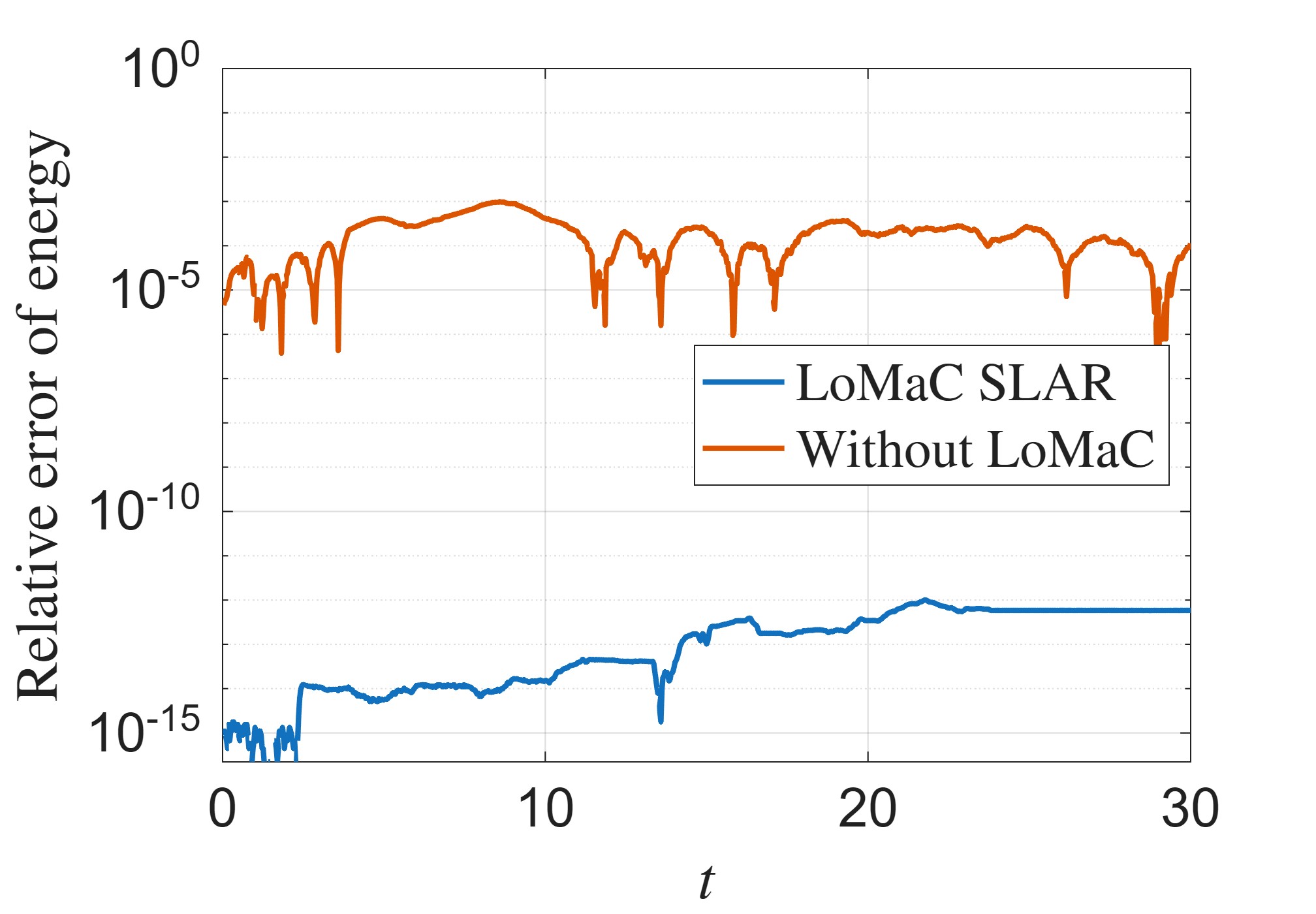}
    }

    \vspace{-0.5cm}

    \hspace{-0.4cm}
    \subfigure{
    \includegraphics[width=0.45\textwidth]{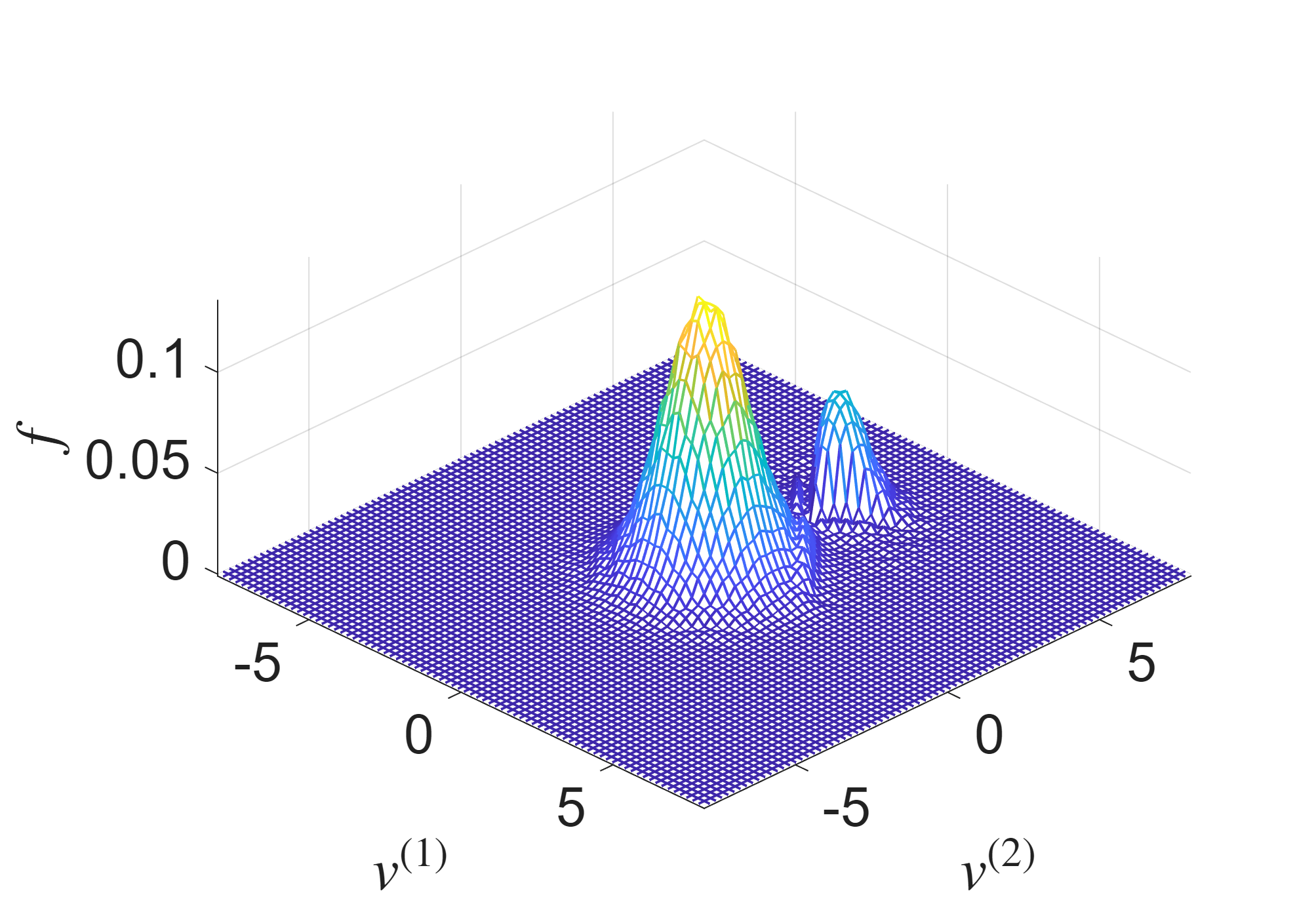}
    }
    \hspace{-0.55cm}
    \subfigure{
    \includegraphics[width=0.45\textwidth]{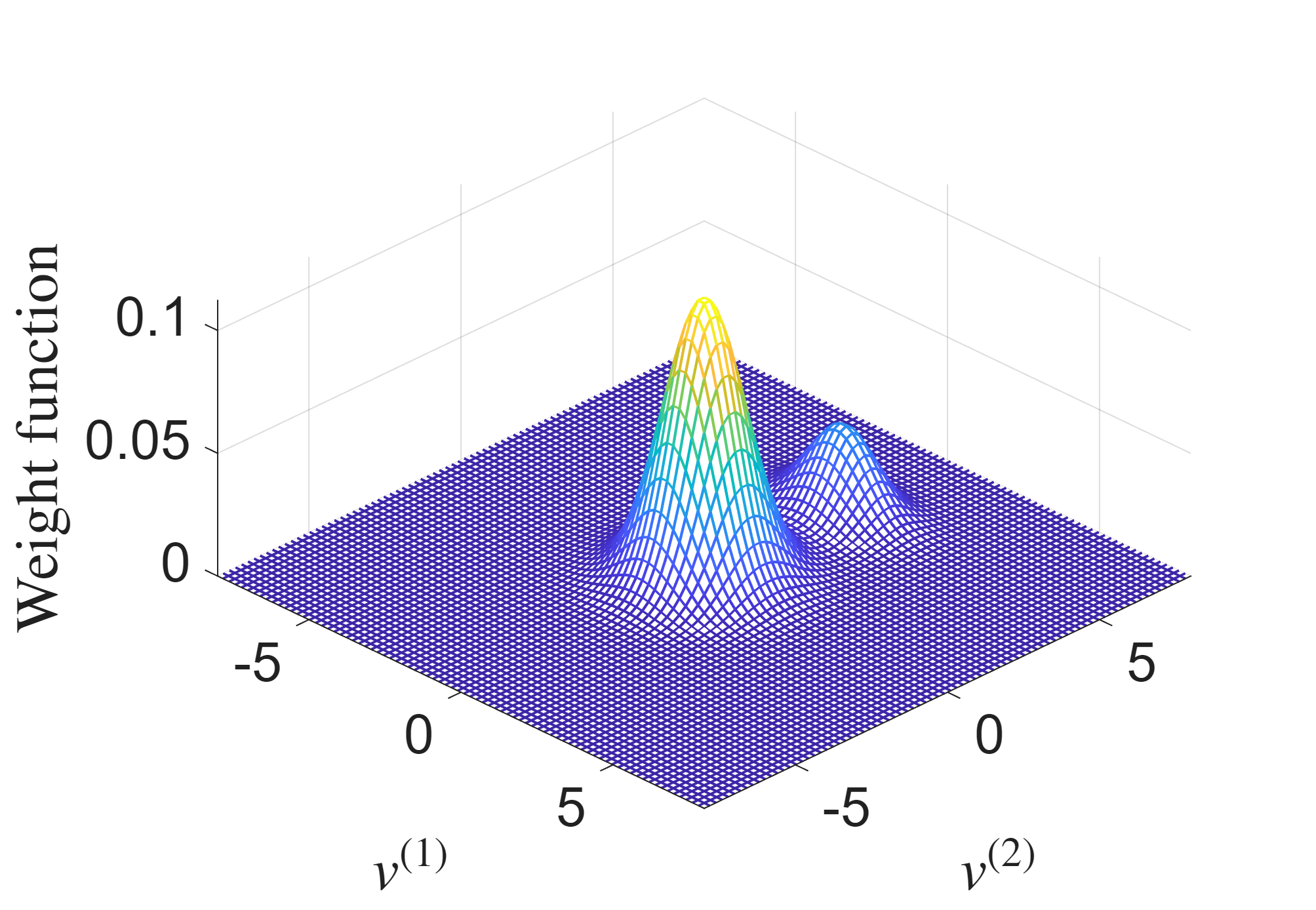}
    }

 \caption{(2D--2V bump-on-tail instability). The first row shows the time evolution of the electric energy, the HT rank history, and a contour plot of the distribution slice at \(x^{(2)} = x^{(2)}_1 = \Delta x^{(2)}/2\) and \(v^{(2)} = v^{(2)}_{64} =  -\Delta v^{(2)}/2\) and $t=30$. The second row shows the relative error in mass, the absolute error in the first component of momentum, and the relative error in total energy. The last row shows the mesh plots of a selected $(v^{(1)},v^{(2)})$-slice of the distribution function and the corresponding adaptive weight function selected by the LoMaC projection at $t=30$.}
\label{fig:2D2V_BOT_results}
\end{figure}

\Cref{fig:2D2V_BOT_solver} shows the performance of the nonlinear macroscopic solver and the wall-clock time histories of the numerical scheme with and without LoMaC. The JFNK iteration counts and wall-clock times are similar to those observed in the 2D--2V two-stream instability test. These results further motivate the development of an enhanced HTACA algorithm to address the sensitivity of the solver to higher solution ranks.

\begin{figure}[!htbp]
\centering

    \hspace{-0.4cm}
    \subfigure{
    \includegraphics[width=0.33\textwidth]{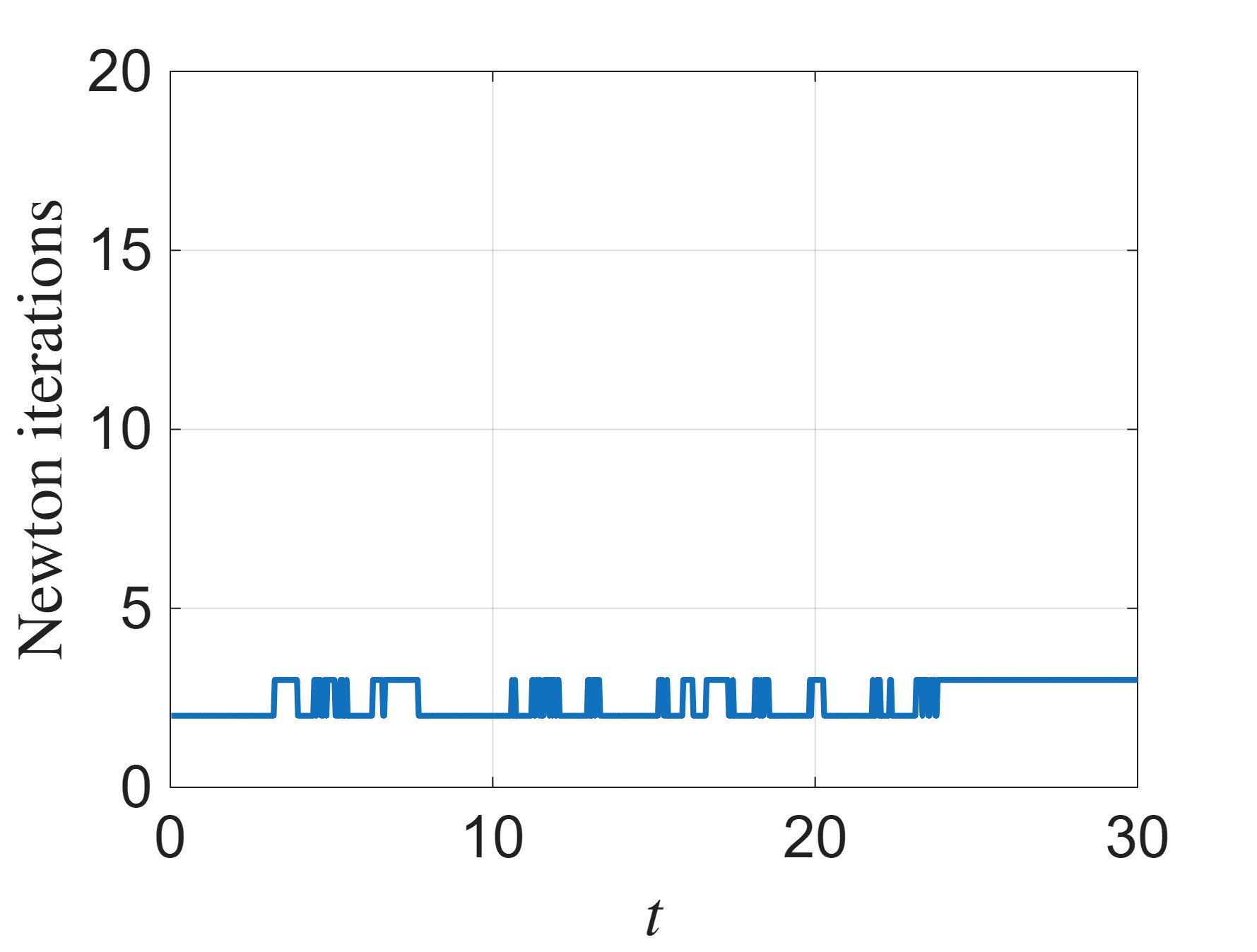}
    }
    \hspace{-0.55cm}
    \subfigure{
    \includegraphics[width=0.33\textwidth]{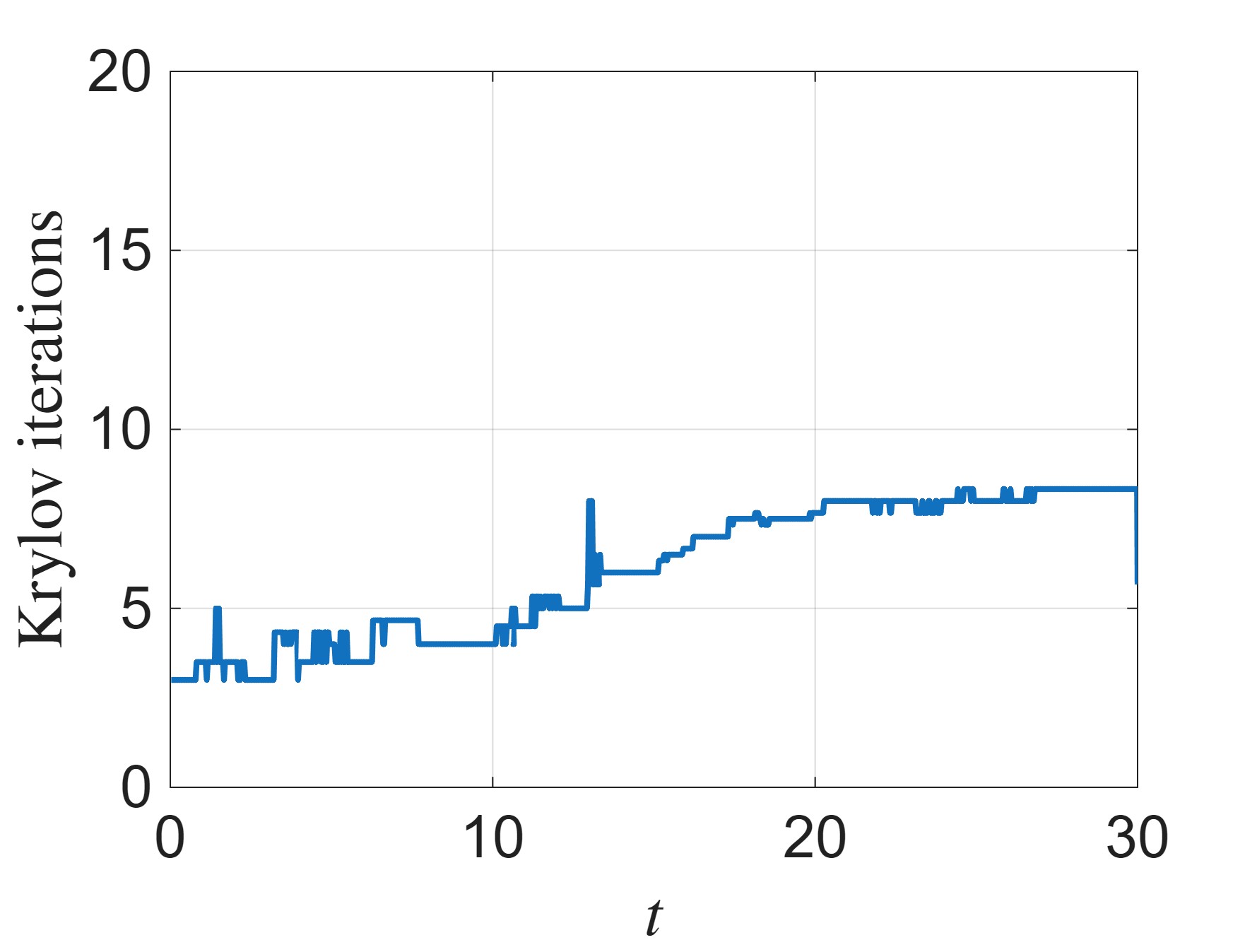}
    }
    \hspace{-0.55cm}
    \subfigure{
    \includegraphics[width=0.33\textwidth]{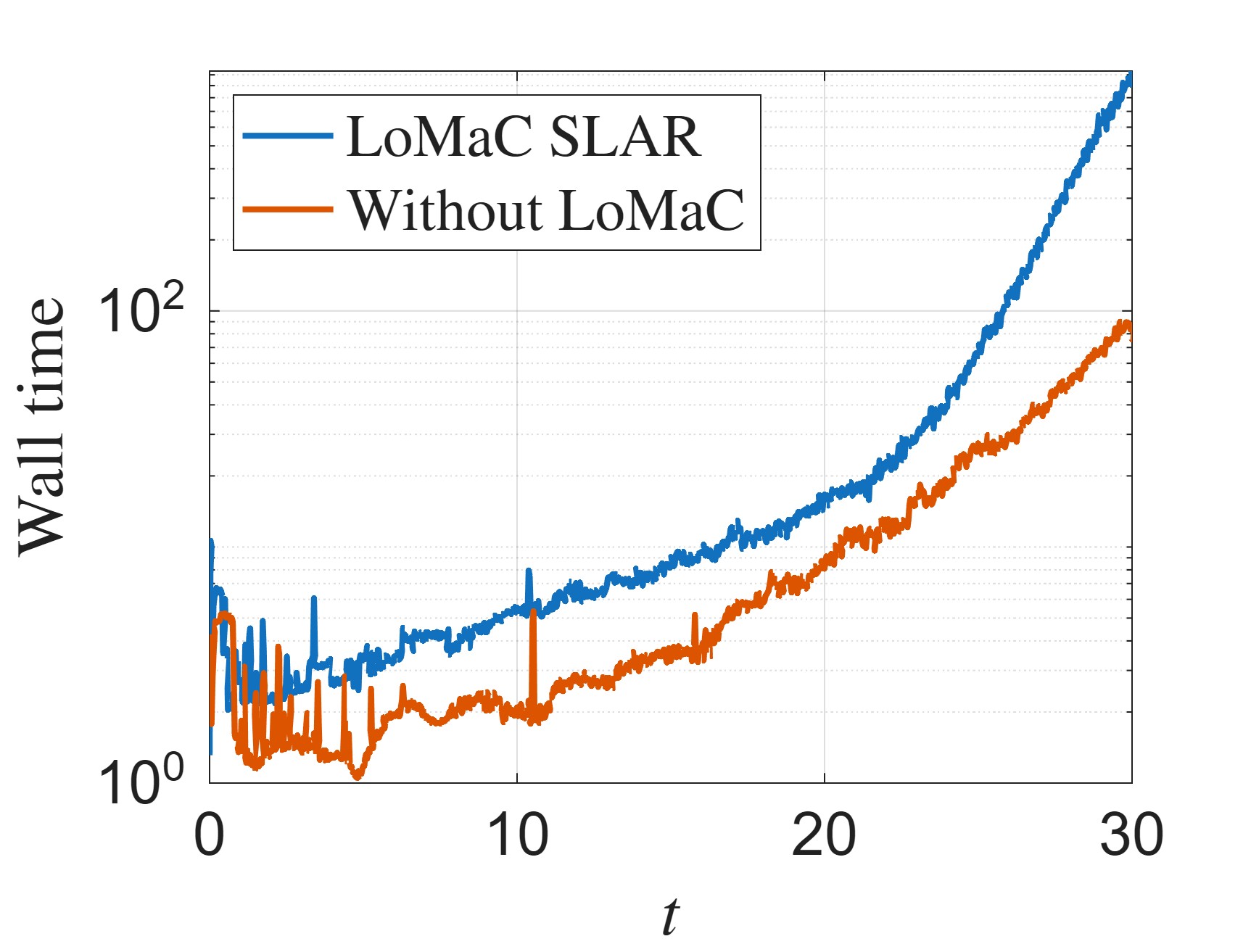}
    }

 \caption{(2D--2V bump-on-tail instability). JFNK solver and performance diagnostics. The panels show the number of Newton iterations, number of Krylov iterations, and wall-clock time in seconds.}
  \label{fig:2D2V_BOT_solver}
\end{figure}

\end{example}

\section{Conclusion}\label{sec:conclusion}

We proposed a new SLAR scheme for the VP system that preserves local conservation laws corresponding to mass, momentum, and total energy. The method achieves third-order accuracy in both space and time and does not use dimensional splitting techniques. Preservation of the local conservation laws is achieved by an implicit coupling of the macroscopic and microscopic descriptions of the VP system, which ensures consistency between the descriptions and permits the use of large time steps, consistent with SL methods. The LoMaC projection used to correct the moments of the kinetic solution is formulated using an adaptive weight function that can be efficiently constructed and naturally captures the dominant velocity space structure of the kinetic solution. When combined with the HTACA algorithm introduced in our previous work \cite{zheng2025semihighD}, the resulting algorithm demonstrates significant improvements in preserving fundamental conservation laws across important numerical benchmarks. Although it is demonstrated in the context of low-rank simulations, the proposed projection technique is quite general and offers a way to construct a family of conservative SL methods without relying on dimensional splitting techniques. Such capabilities are particularly useful in the context of fusion applications where particle transport may not be aligned with the mesh. 

There are several possible directions for future research, some of which are currently being explored. The solver diagnostics presented in this work show that the cost of these corrections is reasonable for modest CFL numbers (e.g., CFL $\leq 10$). However, in order to make these methods competitive at larger CFL numbers, we plan to develop preconditioning strategies for the inner Krylov solver. The development of preconditioning strategies for implicit discretizations of hyperbolic systems is challenging. However, an advantage of the current approach is that this component is of reduced dimensionality when compared to the original kinetic problem. This may allow one to appeal to strategies that cannot be applied to the kinetic system due to its high dimensionality.

Other possible directions concern improvements to the adaptive-rank components of the proposed algorithms. One such example is the HTACA method which is used to construct low-rank approximations in the SLAR method. Although this method was an improvement to existing methods in the literature, the original HTACA algorithm proposed in \cite{zheng2025semihighD} becomes less efficient as the hierarchical ranks increase. We note that there are ongoing efforts to develop more efficient cross-type algorithms for high-dimensional kinetic simulations. Recent work \cite{sands2026multilevel} has also shown multilevel preconditioning to be an effective strategy for linear and nonlinear systems in the adaptive-rank setting.

\appendix

\section{Stability Analysis}
\label{app:stability_proof}
We prove the stability result in \cref{thm:uncond_stab} for the SL--FD scheme with tensor-product quadratic interpolation for the constant-coefficient advection equation.

\subsection*{One-Dimensional Analysis} Consider the one-dimensional advection equation
\(
    f_t + a f_x = 0,
\)
on a uniform grid \(x_j = j\Delta x\). Let \(\Delta t > 0\) and define the CFL
number
\(
    \lambda = a\Delta t/\Delta x.
\)
The SL update reads
\[
    f(x_j,t^{n+1}) = f(x_j - a\Delta t,t^n).
\]
We first assume that
\(
    |\lambda| \le 1/2,
\)
so that the characteristic foot
\(
    x^\star = x_j-a\Delta t
\)
lies in the cell \(I_j=[x_{j-\frac12},x_{j+\frac12}]\). Under this assumption,
the local reconstruction uses the centered three-point stencil
\(
    \{j-1,j,j+1\}.
\)
Since $\lambda = - (x^\star - x_j)/\Delta x$, the discrete SL--FD update can be written as
\begin{equation*}
    f_i^{n+1}
    = \langle\mathbf{w}(-\lambda),(f^n_{j-1},f^n_j,f^n_{j+1})\rangle,
\end{equation*}
where the interpolation weight $\mathbf{w}(\cdot)$ is provided in \eqref{eq:1d_quadratic_interp_weight}. Assume that at time level $n$ the grid function
$\{f_j^n\}_{j\in\mathbb Z}$ consists of a single discrete Fourier mode,
that is,
\[
    f_j^n
    =
    \hat f^{\,n}
    e^{i j\zeta},
    \qquad j\in\mathbb Z,
\]
where $i=\sqrt{-1}$ and
$\zeta\in[0,2\pi]$ is the dimensionless discrete wave number. Because the scheme is linear and translation-invariant, a single Fourier mode remains a Fourier mode after one time step. Hence we write
\[
    f_j^{n+1}
    =
    \hat f^{\,n+1} e^{i j\zeta},\qquad j\in\mathbb{Z}.
\]
Substituting the ansatz into the update formula and
factoring out $e^{i j\zeta}$ yields
\[
    \hat f^{\,n+1}
    =
    g_{1D}(\lambda,\zeta)\,\hat f^{\,n},
\]
with amplification factor 
\begin{equation*}
g_{1D}(\lambda,\zeta) = 1 + \lambda^2(\cos\zeta - 1) - i\lambda \sin\zeta. 
\end{equation*}
A direct calculation gives
\begin{align*}
|g_{1D}(\lambda,\zeta)|^2
&= \bigl(1-\lambda^2+\lambda^2\cos\zeta\bigr)^2
+ \lambda^2 \sin^2\zeta \\
&=
\bigl(1-\lambda^2(1-\cos\zeta)\bigr)^2
+ \lambda^2(1-\cos\zeta)(1+\cos\zeta) \\
&=
1 - \lambda^2(1-\lambda^2)(1-\cos\zeta)^2 .
\end{align*}
Hence, if $|\lambda|\le 1$, then for any $\zeta$, we have $|g_{1D}(\lambda,\zeta)| \le 1$.

\subsection*{Extension to Arbitrary CFL} This can be extended to an arbitrary CFL number as follows. For general $\lambda\in\mathbb{R}$, 
decompose
\[
    \lambda = m + \theta,
    \qquad 
    m = \operatorname{round}(\lambda),
    \quad 
    \theta \in \left(-\frac{1}{2},\frac{1}{2}\right),
\]
where $\operatorname{round}(\cdot)$ denotes the nearest integer. On a uniform grid, the SL--FD interpolation depends only on the relative 
position of the foot point within a local stencil. 
Thus, the update with CFL $\lambda$ can be interpreted as an integer shift by $m$ grid points and a centered quadratic update with reduced CFL $\theta$. 
In Fourier space, the integer shift contributes only a phase factor.
Therefore,
\begin{equation*}
    g_{1D}(\lambda,\zeta)
    =
    e^{-i m \zeta}\,
    g_{1D}(\theta,\zeta),
\end{equation*}
and hence
\[
    |g_{1D}(\lambda,\zeta)|
    =
    |g_{1D}(\theta,\zeta)|
    \le 1,
    \qquad 
    \forall \lambda\in\mathbb{R}.
\]

\subsection*{Extension to Arbitrary Dimensions} Lastly, we extend the one-dimensional stability property to $d$ dimensions. Similar to the one-dimensional case, we define the dimension-wise CFL number \[\bm{\lambda} = (\lambda_1,\cdots,\lambda_d) = \left(\frac{a_1\Delta t}{\Delta x^{(1)}},\ldots,\frac{a_d\Delta t}{\Delta x^{(d)}} \right).\] For a Fourier mode
\[
    F_{j_1,\ldots,j_d}^n
    =
    \hat F^{\,n}
    e^{i\sum_{\mu=1}^d j_\mu \zeta_\mu},\qquad (j_1,\ldots,j_d)\in\mathbb{Z}^d,
\]
where $\zeta_\mu\in[0,2\pi]$ for $\mu = 1,\ldots, d$. Using the tensor-product interpolation formula \eqref{eq:tp_interp}, the amplification factor factorizes as
\[
    g_d(\bm{\lambda},\bm{\zeta})
    =
    \prod_{\mu=1}^d
    g_{1D}(\lambda_\mu,\zeta_\mu),
\]
Taking moduli yields
\[
    |g_d(\bm{\lambda},\bm{\zeta})|
    =
    \prod_{\mu=1}^d
    |g_{1D}(\lambda_\mu,\zeta_\mu)|
    \le 1,
    \qquad
    \forall \bm{\lambda}\in\mathbb{R}^d.
\]
Finally, Parseval's identity implies unconditional
$\ell^2$ stability in any spatial dimension.
\qed

\section*{Acknowledgements}
The authors wish to acknowledge support from AFOSR through grant FA9550-24-1-0254 and the DOE Office of Applied Scientific Computing Research (ASCR) Mathematical Multifaceted Integrated Capabilities Center (MMICC) Program under grant DE-SC0023164. Los Alamos National Laboratory report number LA-UR-26-25294.

\bibliographystyle{unsrtnat}
\bibliography{Reference}

\begin{thebibliography}{10}

\bibitem{alexander1977diagonally}
{\sc R.~Alexander}, {\em Diagonally implicit {Runge--Kutta} methods for stiff
  {ODE’s}}, SIAM Journal on Numerical Analysis, 14 (1977), pp.~1006--1021.

\bibitem{cai2021high}
{\sc X.~Cai, S.~Boscarino, and J.-M. Qiu}, {\em High order semi-{L}agrangian
  discontinuous {G}alerkin method coupled with {Runge-Kutta} exponential
  integrators for nonlinear {V}lasov dynamics}, Journal of Computational
  Physics, 427 (2021), p.~110036.

\bibitem{cai2018high}
{\sc X.~Cai, W.~Guo, and J.-M. Qiu}, {\em A high order semi-{L}agrangian
  discontinuous {G}alerkin method for {V}lasov--{P}oisson simulations without
  operator splitting}, Journal of Computational Physics, 354 (2018),
  pp.~529--551.

\bibitem{celledoni2003commutator}
{\sc E.~Celledoni, A.~Marthinsen, and B.~Owren}, {\em Commutator-free {Lie}
  group methods}, Future Generation Computer Systems, 19 (2003), pp.~341--352.

\bibitem{chacon2017multiscale}
{\sc L.~Chacon, G.~Chen, D.~A. Knoll, C.~Newman, H.~Park, W.~Taitano, J.~A.
  Willert, and G.~Womeldorff}, {\em Multiscale high-order/low-order ({HOLO})
  algorithms and applications}, Journal of Computational Physics, 330 (2017),
  pp.~21--45.

\bibitem{cho2024conservative}
{\sc S.-Y. Cho, M.~Groppi, J.-M. Qiu, G.~Russo, and S.-B. Yun}, {\em
  Conservative semi-{L}agrangian methods for kinetic equations}, in Active
  Particles, Volume 4: Theory, Models, Applications, Springer, 2024,
  pp.~283--420.

\bibitem{crouseilles2010conservative}
{\sc N.~Crouseilles, M.~Mehrenberger, and E.~Sonnendr{\"u}cker}, {\em
  Conservative semi-{L}agrangian schemes for {V}lasov equations}, Journal of
  Computational Physics, 229 (2010), pp.~1927--1953.

\bibitem{ding2020semi}
{\sc M.~Ding, X.~Cai, W.~Guo, and J.-M. Qiu}, {\em A semi-{L}agrangian
  discontinuous {G}alerkin ({DG})--local {DG} method for solving
  convection-diffusion equations}, Journal of Computational Physics, 409
  (2020), p.~109295.

\bibitem{einkemmer2021mass}
{\sc L.~Einkemmer and I.~Joseph}, {\em A mass, momentum, and energy
  conservative dynamical low-rank scheme for the {V}lasov equation}, Journal of
  Computational Physics, 443 (2021), p.~110495.

\bibitem{einkemmer2025review}
{\sc L.~Einkemmer, K.~Kormann, J.~Kusch, R.~G. McClarren, and J.-M. Qiu}, {\em
  A review of low-rank methods for time-dependent kinetic simulations}, Journal
  of Computational Physics,  (2025), p.~114191.

\bibitem{einkemmer2018low}
{\sc L.~Einkemmer and C.~Lubich}, {\em A low-rank projector-splitting
  integrator for the {V}lasov--{P}oisson equation}, SIAM Journal on Scientific
  Computing, 40 (2018), pp.~B1330--B1360.

\bibitem{falcone2013semi}
{\sc M.~Falcone and R.~Ferretti}, {\em Semi-{L}agrangian approximation schemes
  for linear and {H}amilton—{J}acobi equations}, SIAM, 2013.

\bibitem{filbet2001conservative}
{\sc F.~Filbet, E.~Sonnendr{\"u}cker, and P.~Bertrand}, {\em Conservative
  numerical schemes for the {V}lasov equation}, Journal of Computational
  Physics, 172 (2001), pp.~166--187.

\bibitem{gear1971numerical}
{\sc C.~W. Gear}, {\em Numerical initial value problems in ordinary
  differential equations}, Prentice Hall PTR, 1971.

\bibitem{guo2024localDG}
{\sc W.~Guo, J.~F. Ema, and J.-M. Qiu}, {\em A local macroscopic conservative
  ({LoMaC}) low rank tensor method with the discontinuous {G}alerkin method for
  the {V}lasov dynamics}, Communications on Applied Mathematics and
  Computation, 6 (2024), pp.~550--575.

\bibitem{guo2022low}
{\sc W.~Guo and J.-M. Qiu}, {\em A low rank tensor representation of linear
  transport and nonlinear {V}lasov solutions and their associated flow maps},
  Journal of Computational Physics, 458 (2022), p.~111089.

\bibitem{guo2024conservative}
\leavevmode\vrule height 2pt depth -1.6pt width 23pt, {\em A conservative low
  rank tensor method for the {V}lasov dynamics}, SIAM Journal on Scientific
  Computing, 46 (2024), pp.~A232--A263.

\bibitem{guo2024local}
\leavevmode\vrule height 2pt depth -1.6pt width 23pt, {\em A local macroscopic
  conservative ({LoMaC}) low rank tensor method for the {V}lasov dynamics},
  Journal of Scientific Computing, 101 (2024), p.~61.

\bibitem{knoll2004jacobian}
{\sc D.~A. Knoll and D.~E. Keyes}, {\em Jacobian-free {N}ewton--{K}rylov
  methods: a survey of approaches and applications}, Journal of Computational
  Physics, 193 (2004), pp.~357--397.

\bibitem{knoll2005jacobian}
{\sc D.~A. Knoll, V.~Mousseau, L.~Chac{\'o}n, and J.~Reisner}, {\em
  Jacobian-free {N}ewton-{K}rylov methods for the accurate time integration of
  stiff wave systems}, Journal of Scientific Computing, 25 (2005),
  pp.~213--230.

\bibitem{kormann2015semi}
{\sc K.~Kormann}, {\em A semi-{L}agrangian {V}lasov solver in tensor train
  format}, SIAM Journal on Scientific Computing, 37 (2015), pp.~B613--B632.

\bibitem{li2023high}
{\sc L.~Li, J.~Qiu, and G.~Russo}, {\em A high-order semi-{L}agrangian finite
  difference method for nonlinear {V}lasov and {BGK} models}, Communications on
  Applied Mathematics and Computation, 5 (2023), pp.~170--198.

\bibitem{li2022stability}
{\sc Z.~Li and H.-L. Liao}, {\em Stability of variable-step {BDF2} and {BDF3}
  methods}, SIAM Journal on Numerical Analysis, 60 (2022), pp.~2253--2272.

\bibitem{qiu2010conservative}
{\sc J.-M. Qiu and A.~Christlieb}, {\em A conservative high order
  semi-{L}agrangian {WENO} method for the {V}lasov equation}, Journal of
  Computational Physics, 229 (2010), pp.~1130--1149.

\bibitem{qiu2017high}
{\sc J.-M. Qiu and G.~Russo}, {\em A high order multi-dimensional
  characteristic tracing strategy for the {V}lasov--{P}oisson system}, Journal
  of Scientific Computing, 71 (2017), pp.~414--434.

\bibitem{rossmanith2011positivity}
{\sc J.~A. Rossmanith and D.~C. Seal}, {\em A positivity-preserving high-order
  {semi-Lagrangian} discontinuous {Galerkin} scheme for the {Vlasov--Poisson}
  equations}, Journal of Computational Physics, 230 (2011), pp.~6203--6232.

\bibitem{sands2026multilevel}
{\sc W.~A. Sands, P.~T. Guthrey, and N.~V. Roberts}, {\em Multilevel
  adaptive-rank methods for linear and nonlinear systems in the hierarchical
  {T}ucker format}, 2026.

\bibitem{sands2025adaptive}
{\sc W.~A. Sands, J.-M. Qiu, D.~Hayes, and N.~Zheng}, {\em An adaptive-rank
  approach with greedy sampling for multi-scale {BGK} equations}, Journal of
  Computational Physics,  (2025), p.~114523.

\bibitem{sonnendrucker1999semi}
{\sc E.~Sonnendr{\"u}cker, J.~Roche, P.~Bertrand, and A.~Ghizzo}, {\em The
  semi-{L}agrangian method for the numerical resolution of the {V}lasov
  equation}, Journal of computational physics, 149 (1999), pp.~201--220.

\bibitem{ye2024quantized}
{\sc E.~Ye and N.~F. Loureiro}, {\em Quantized tensor networks for solving the
  {V}lasov--{M}axwell equations}, Journal of Plasma Physics, 90 (2024),
  p.~805900301.

\bibitem{zheng2022fourth}
{\sc N.~Zheng, X.~Cai, J.-M. Qiu, and J.~Qiu}, {\em A fourth-order conservative
  semi-{L}agrangian finite volume {WENO} scheme without operator splitting for
  kinetic and fluid simulations}, Computer Methods in Applied Mechanics and
  Engineering, 395 (2022), p.~114973.

\bibitem{zheng2025semi}
{\sc N.~Zheng, D.~Hayes, A.~Christlieb, and J.-M. Qiu}, {\em A
  semi-{L}agrangian adaptive-rank ({SLAR}) method for linear advection and
  nonlinear {V}lasov-{P}oisson system}, Journal of Computational Physics,
  (2025), p.~113970.

\bibitem{zheng2025semihighD}
{\sc N.~Zheng, W.~A. Sands, D.~Hayes, A.~J. Christlieb, and J.-M. Qiu}, {\em A
  semi-{L}agrangian adaptive rank ({SLAR}) method for high-dimensional {V}lasov
  dynamics}, 2025.

\end{thebibliography}


\end{document}


\maketitle

\section{Supplement A. Derivation of the Macroscopic System}

To derive the macroscopic system of conservation laws, we first obtain a weak
form of the kinetic equation by testing against a function
\(\Phi=\Phi(\bm{v})\). Multiplying the kinetic equation by \(\Phi\) and
applying \(\langle \cdot \rangle_{\bm{v}}\), we obtain
\begin{equation*}
    \left\langle (\partial_t f_{\mathrm e})\Phi \right\rangle_{\bm{v}}
    + \left\langle (\bm{v}\cdot\nabla_{\bm{x}} f_{\mathrm e})\Phi \right\rangle_{\bm{v}}
    + \frac{q_{\mathrm e}}{m_{\mathrm e}}
    \left\langle (\bm{E}\cdot\nabla_{\bm{v}} f_{\mathrm e})\Phi \right\rangle_{\bm{v}}
    = 0.
\end{equation*}
Next, we make use of the identities
\begin{align*}
    \left\langle (\partial_t f_{\mathrm e})\Phi \right\rangle_{\bm{v}}
    &=
    \partial_t \left\langle f_{\mathrm e}\Phi \right\rangle_{\bm{v}},
    \\
    \left\langle (\bm{v}\cdot\nabla_{\bm{x}} f_{\mathrm e})\Phi \right\rangle_{\bm{v}}
    &=
    \nabla_{\bm{x}}\cdot
    \left\langle \bm{v}\Phi f_{\mathrm e} \right\rangle_{\bm{v}},
    \\
    \frac{q_{\mathrm e}}{m_{\mathrm e}}
    \left\langle (\bm{E}\cdot\nabla_{\bm{v}} f_{\mathrm e})\Phi \right\rangle_{\bm{v}}
    &=
    -\frac{q_{\mathrm e}}{m_{\mathrm e}}
    \left\langle f_{\mathrm e}\nabla_{\bm{v}}\Phi \right\rangle_{\bm{v}}
    \cdot \bm{E}.
\end{align*}
The first two identities follow from the fact that \(\Phi\) depends only on
\(\bm{v}\). The third identity follows from integration by parts in
\(\bm{v}\), together with the assumption that the boundary term in velocity
space vanishes.

The macroscopic equations are obtained by choosing
\[
    \Phi(\bm{v})
    \in
    \left\{
        1,\,
        m_{\mathrm e} \bm{v},\,
        \frac{1}{2}m_{\mathrm e}|\bm{v}|^2
    \right\},
\]
which leads to the moment system
\renewcommand{\arraystretch}{1.25}
\begin{equation}
\label{eq:raw moment system}
    \partial_{t}\begin{pmatrix}
        n_e \\
        m_{\mathrm e} n_e \bm{u}_e \\
        e_{k,\mathrm e}
    \end{pmatrix}
    +
    \nabla_{\bm{x}}\cdot
    \begin{pmatrix}
        \left\langle \bm{v} f_{\mathrm e} \right\rangle_{\bm{v}} \\
        \left\langle m_{\mathrm e}(\bm{v}\otimes\bm{v}) f_{\mathrm e} \right\rangle_{\bm{v}} \\
        \frac{1}{2}
        \left\langle
            m_{\mathrm e}|\bm{v}|^2\bm{v} f_{\mathrm e}
        \right\rangle_{\bm{v}}
    \end{pmatrix}
    =
    \begin{pmatrix}
        0 \\
        \rho_{\mathrm e}\bm{E} \\
        \bm{J}_{\mathrm e}\cdot\bm{E}
    \end{pmatrix}.
\end{equation}
These equations represent conservation of particle number, a momentum balance,
and a kinetic energy balance. The kinetic energy is not conserved by itself,
because the particles exchange energy with the electric field.

To identify the corresponding field energy, we integrate the kinetic energy density
equation (last equation in \eqref{eq:raw moment system}) over \(\Omega_{\bm{x}}\):
\begin{equation*}
    \frac{d}{dt}
    \left(
        \int_{\Omega_{\bm{x}}}
        e_{k,\mathrm e}\,d\bm{x}
    \right)
    +
    \int_{\Omega_{\bm{x}}}
    \nabla_{\bm{x}}\cdot
    \left(
        \frac{1}{2}
        \left\langle
            m_{\mathrm e}|\bm{v}|^2\bm{v} f_{\mathrm e}
        \right\rangle_{\bm{v}}
    \right)
    d\bm{x}
    =
    \int_{\Omega_{\bm{x}}}
    \bm{J}_{\mathrm e}\cdot\bm{E}\,d\bm{x}.
\end{equation*}
Applying the divergence theorem to the second term on the left-hand side gives
\begin{equation*}
    \frac{d}{dt}
    \left(
        \int_{\Omega_{\bm{x}}}
        e_{k,\mathrm e}\,d\bm{x}
    \right)
    +
    \int_{\partial\Omega_{\bm{x}}}
    \left(
        \frac{1}{2}
        \left\langle
            m_{\mathrm e}|\bm{v}|^2\bm{v} f_{\mathrm e}
        \right\rangle_{\bm{v}}
    \right)
    \cdot \bm{n}\,dS
    =
    \int_{\Omega_{\bm{x}}}
    \bm{J}_{\mathrm e}\cdot\bm{E}\,d\bm{x}.
\end{equation*}
For periodic boundary conditions, or boundary conditions for which the
corresponding surface term vanishes, this further reduces to
\begin{equation*}
    \frac{d}{dt}
    \left(
        \int_{\Omega_{\bm{x}}}
        e_{k,\mathrm e}\,d\bm{x}
    \right)
    =
    \int_{\Omega_{\bm{x}}}
    \bm{J}_{\mathrm e}\cdot\bm{E}\,d\bm{x}.
\end{equation*}

We now simplify the right-hand side. Using
\(\bm{E}=-\nabla_{\bm{x}}\phi\), we have
\begin{align*}
    \int_{\Omega_{\bm{x}}}
    \bm{J}_{\mathrm e}\cdot\bm{E}\,d\bm{x}
    &=
    -\int_{\Omega_{\bm{x}}}
    \bm{J}_{\mathrm e}\cdot\nabla_{\bm{x}}\phi\,d\bm{x}
    \\
    &=
    -\int_{\Omega_{\bm{x}}}
    \nabla_{\bm{x}}\cdot(\bm{J}_{\mathrm e}\phi)\,d\bm{x}
    +
    \int_{\Omega_{\bm{x}}}
    (\nabla_{\bm{x}}\cdot\bm{J}_{\mathrm e})\phi\,d\bm{x}
    \\
    &=
    -\int_{\partial\Omega_{\bm{x}}}
    (\bm{J}_{\mathrm e}\phi)\cdot\bm{n}\,dS
    +
    \int_{\Omega_{\bm{x}}}
    (\nabla_{\bm{x}}\cdot\bm{J}_{\mathrm e})\phi\,d\bm{x}.
\end{align*}
Under the same boundary assumptions, the surface term vanishes. Moreover, since
the ions are stationary, the total current is the electron current,
\(\bm{J}=\bm{J}_{\mathrm e}\). Therefore,
\begin{equation*}
    \int_{\Omega_{\bm{x}}}
    \bm{J}_{\mathrm e}\cdot\bm{E}\,d\bm{x}
    =
    \int_{\Omega_{\bm{x}}}
    (\nabla_{\bm{x}}\cdot\bm{J})\phi\,d\bm{x}.
\end{equation*}

We next use the charge continuity equation
\begin{equation*}
    \partial_t\rho+\nabla_{\bm{x}}\cdot\bm{J}=0,
\end{equation*}
together with the time derivative of Poisson's equation. Since
\[
    -\epsilon_0\Delta_{\bm{x}}\phi = \rho,
\]
we obtain
\begin{equation*}
    \nabla_{\bm{x}}\cdot\bm{J}
    =
    -\partial_t\rho
    =
    \epsilon_0\Delta_{\bm{x}}(\partial_t\phi).
\end{equation*}
Again, by the divergence theorem
\begin{align*}
    \int_{\Omega_{\bm{x}}}
    (\nabla_{\bm{x}}\cdot\bm{J})\phi\,d\bm{x}
    &=
    \epsilon_0
    \int_{\Omega_{\bm{x}}}
    \Delta_{\bm{x}}(\partial_t\phi)\,\phi\,d\bm{x}
    \\
    &=
    \epsilon_0
    \int_{\Omega_{\bm{x}}}
    \nabla_{\bm{x}}\cdot
    \left(
        \nabla_{\bm{x}}(\partial_t\phi)\,\phi
    \right)
    d\bm{x}
    -
    \epsilon_0
    \int_{\Omega_{\bm{x}}}
    \nabla_{\bm{x}}(\partial_t\phi)
    \cdot
    \nabla_{\bm{x}}\phi\,d\bm{x}
    \\
    &=
    \epsilon_0
    \int_{\partial\Omega_{\bm{x}}}
    \left(
        \nabla_{\bm{x}}(\partial_t\phi)\,\phi
    \right)
    \cdot\bm{n}\,dS
    -
    \frac{d}{dt}
    \left(
        \frac{\epsilon_0}{2}
        \int_{\Omega_{\bm{x}}}
        |\nabla_{\bm{x}}\phi|^2\,d\bm{x}
    \right)
    \\
    &=
    -
    \frac{d}{dt}
    \left(
        \frac{\epsilon_0}{2}
        \int_{\Omega_{\bm{x}}}
        |\nabla_{\bm{x}}\phi|^2\,d\bm{x}
    \right)
    \\
    &=
    -
    \frac{d}{dt}
    \left(
        \frac{\epsilon_0}{2}
        \int_{\Omega_{\bm{x}}}
        |\bm{E}|^2\,d\bm{x}
    \right),
\end{align*}
where the surface term again vanishes under the assumed boundary conditions.
Combining these identities, we obtain conservation of the total energy,
\begin{equation*}
    \frac{d}{dt}
    \int_{\Omega_{\bm{x}}}
    \left(
        e_{k,\mathrm e}
        +
        \frac{\epsilon_0}{2}|\bm{E}|^2
    \right)
    d\bm{x}
    =
    0.
\end{equation*}

Since the kinetic energy density is not conserved locally by itself, we define
the total energy density
\begin{equation*}
    e := e_{k,\mathrm e} + e_p,
\end{equation*}
where the electric field energy density is
\begin{equation*}
    e_p = \frac{\epsilon_0}{2}|\bm{E}|^2.
\end{equation*}
We can derive a local balance for \(e_p\) by differentiating in time:
\begin{align*}
    \partial_t e_p
    &=
    \partial_t
    \left(
        \frac{\epsilon_0}{2}|\bm{E}|^2
    \right)
    \\
    &=
    \epsilon_0\bm{E}\cdot\partial_t\bm{E}
    \\
    &=
    \epsilon_0
    \nabla_{\bm{x}}\phi
    \cdot
    \nabla_{\bm{x}}(\partial_t\phi)
    \\
    &=
    \epsilon_0
    \nabla_{\bm{x}}\cdot
    \left(
        \phi\nabla_{\bm{x}}(\partial_t\phi)
    \right)
    -
    \epsilon_0
    \Delta_{\bm{x}}(\partial_t\phi)\,\phi
    \\
    &=
    \epsilon_0
    \nabla_{\bm{x}}\cdot
    \left(
        \phi\nabla_{\bm{x}}(\partial_t\phi)
    \right)
    -
    (\nabla_{\bm{x}}\cdot\bm{J})\phi
    \\
    &=
    \epsilon_0
    \nabla_{\bm{x}}\cdot
    \left(
        \phi\nabla_{\bm{x}}(\partial_t\phi)
    \right)
    -
    \nabla_{\bm{x}}\cdot(\bm{J}\phi)
    +
    \bm{J}\cdot\nabla_{\bm{x}}\phi
    \\
    &=
    \epsilon_0
    \nabla_{\bm{x}}\cdot
    \left(
        \phi\nabla_{\bm{x}}(\partial_t\phi)
    \right)
    -
    \nabla_{\bm{x}}\cdot(\bm{J}\phi)
    -
    \bm{J}\cdot\bm{E}.
\end{align*}
Adding this field energy balance to the kinetic energy balance gives the
local conservation law
\begin{equation*}
    \partial_t e
    +
    \nabla_{\bm{x}}\cdot
    \left(
        \frac{1}{2}
        \left\langle
            m_{\mathrm e}|\bm{v}|^2\bm{v} f_{\mathrm e}
        \right\rangle_{\bm{v}}
        -
        \epsilon_0\phi\nabla_{\bm{x}}(\partial_t\phi)
        +
        \bm{J}\phi
    \right)
    =
    0.
\end{equation*}
The quantity \(\partial_t\phi\) is obtained by solving the Poisson equation
\begin{equation*}
    -\epsilon_0\Delta_{\bm{x}}(\partial_t\phi)
    =
    -\nabla_{\bm{x}}\cdot\bm{J}.
\end{equation*}

\section{Supplement B. Construction of the Discrete Projection}
\label{app:discrete_LoMaC}

This section provides the details regarding the construction of the discrete LoMaC projection with the adaptive weight function. First, we introduce weighted inner products and use them with Gram-Schmidt orthogonalization to construct an explicit projection.

\subsection{Discrete velocity weight and weighted inner products}
\label{app:discrete_LoMaC_A1}

Let $\mathbf F^\star$ denote a discrete phase-space tensor.
In the LoMaC update, $\mathbf F^\star$ corresponds to the intermediate SLAR solution $\mathcal F^{n+1,\star}$.
We denote by $\mathbf F^\star_{(\mu)}\in
\mathbb R^{N_{v^{(\mu)}}\times \bar{M}}$
the matricization of $\mathbf F^\star$ with respect to the
dimension corresponding to $v^{(\mu)}$,
where $\bar{M}$ is the product of all remaining spatial and velocity degrees of freedom.
Under standard quadrature discretization,
$\mathcal T_\mu$ is represented by $\mathbf F^\star_{(\mu)}$ and
$\mathcal T_\mu \mathcal T_\mu^*$ corresponds to
\[
\mathbf{F}_{(\mu)}^\star (\mathbf{F}_{(\mu)}^\star)^\top.
\]
Therefore, the discrete dominant eigenfunction coincides with the first left
singular vector of $\mathbf F^\star_{(\mu)}$:
\[
\mathbf u_1^{(\mu)}
=
\text{the first left singular vector of }\mathbf F_{(\mu)}^\star.
\]

For structured tensor formats, e.g., the HTD, there is an analogous concept. As an example, the final truncation step of the HTACA algorithm applies the high-order SVD to the nodes of the dimension tree using a post-order traversal. Upon completion, the first column of each leaf basis matrix of $\mathcal{F}^{n+1,\star}$ coincides with the leading
left singular vector of the corresponding mode-$\mu$ matricization.
Consequently, the discrete factors $\mathbf u_1^{(\mu)}$ required for the adaptive
velocity weight construction are obtained without additional computational cost.


To remove the sign ambiguity, we define the one-dimensional velocity factors
componentwise by
\[
\bm{\omega}^{(\mu),\star}
=
\big| \mathbf u_1^{(\mu)} \big|,
\]
where $|\cdot|$ denotes elementwise absolute value.
The discrete velocity weight is then constructed in separable form,
\begin{equation}
\label{eq:omega_star_tensor}
\mathbf{\Omega}^\star
=
\bm{\omega}^{(1),\star}
\otimes
\cdots
\otimes
\bm{\omega}^{(d_v),\star}.
\end{equation}
To ensure well-posedness of the weighted inner product defined below,
we enforce strict positivity at the level of the one-dimensional velocity
factors. Let $\delta>0$ be a small floor parameter independent of $\mu$.
For each velocity component $\mu=1,\dots,d_v$, we modify
\begin{equation*}
\omega^{(\mu),\star}_{k_{\mu}}
\leftarrow
\max\!\big(\omega^{(\mu),\star}_{k_{\mu}},\delta\big),
\qquad
k_{\mu}=1,\dots,N_{v^{(\mu)}}.
\end{equation*}
The discrete velocity weight is then defined in separable form as in
\eqref{eq:omega_star_tensor}. Consequently, for any multi-index
$k=(k_1,\dots,k_{d_v})$,
\[
\Omega^\star_k
=
\prod_{\mu=1}^{d_v}
\omega^{(\mu),\star}_{k_\mu}
\;\ge\;
\delta^{d_v}
\;>\;0,
\]
which guarantees strict positivity on the entire discrete velocity grid
while preserving the rank-one tensor structure.


Let
\(
\mathbf{F}
\in
\mathbb R^{N_{x^{(1)}}\times\cdots\times N_{x^{(d_x)}}
\times
N_{v^{(1)}}\times\cdots\times N_{v^{(d_v)}}}
\)
denote a discrete phase-space tensor.
We write its entries as $F_{i,k}$,
where $i=(i_1,\dots,i_{d_x})$ and
$k=(k_1,\dots,k_{d_v})$
denote spatial and velocity multi-indices, respectively. We introduce the following discrete inner products.

For phase-space tensors \(\mathbf A\) and \(\mathbf B\), we define
\begin{equation}
\label{eq:inner_global}
\langle \mathbf A,\mathbf B\rangle_{x,v;\omega^{-1}}
:=
\sum_i\sum_k
A_{i,k}B_{i,k}
\frac{w_i^x w_k^v}{\Omega^\star_k},
\end{equation}
where \(w_i^x\) and \(w_k^v\) are the spatial and velocity quadrature weights,
respectively. On uniform grids,
\(w_i^x=\prod_{\mu=1}^{d_x}\Delta x^{(\mu)}\) and
\(w_k^v=\prod_{\mu=1}^{d_v}\Delta v^{(\mu)}\). Similarly, for velocity tensors
\(\mathbf G,\mathbf H\in
\mathbb R^{N_{v^{(1)}}\times\cdots\times N_{v^{(d_v)}}}\), we define
\begin{equation}
\label{eq:inner_velocity}
\langle \mathbf G,\mathbf H\rangle_{v;\omega^{-1}}
:=
\sum_k G_kH_k\frac{w_k^v}{\Omega^\star_k}.
\end{equation}
The discrete velocity projection
\(\mathcal P_{\mathcal N(\mathbf\Omega^\star)}\) is the orthogonal projection
onto \(\mathcal N(\mathbf\Omega^\star)\) with respect to
\eqref{eq:inner_velocity}. The corresponding phase-space projection is
\[
\mathcal P^\star
=
\mathcal I_{\mathbb R^{N_{x^{(1)}}\times\cdots\times N_{x^{(d_x)}}}}
\otimes
\mathcal P_{\mathcal N(\mathbf\Omega^\star)},
\]
which acts independently on each spatial multi-index. By the tensor-product
structure of \eqref{eq:inner_global}, \(\mathcal P^\star\) is the orthogonal
projection onto
\[
\mathbb R^{N_{x^{(1)}}\times\cdots\times N_{x^{(d_x)}}}
\otimes
\mathcal N(\mathbf\Omega^\star)
\]
with respect to \eqref{eq:inner_global}.

Since \(\mathbf\Omega^\star\) is separable, see \eqref{eq:omega_star_tensor},
we have
\[
\Omega^\star_k
=
\prod_{\mu=1}^{d_v}\omega^{(\mu),\star}_{k_\mu},
\qquad
w_k^v
=
\prod_{\mu=1}^{d_v}w^{(\mu)}_{k_\mu}.
\]
For one-dimensional velocity components, define
\begin{equation*}
\langle \mathbf p,\mathbf q\rangle_{v^{(\mu)};\omega^{(\mu),-1}}
:=
\sum_{k_\mu}
p_{k_\mu}q_{k_\mu}
\frac{w^{(\mu)}_{k_\mu}}{\omega^{(\mu),\star}_{k_\mu}}.
\end{equation*}
Consequently, if
\[
\mathbf G=\bigotimes_{\mu=1}^{d_v}\mathbf g^{(\mu)},
\qquad
\mathbf H=\bigotimes_{\mu=1}^{d_v}\mathbf h^{(\mu)},
\]
then
\begin{equation}
\label{eq:inner_prod_factorization}
\langle \mathbf G,\mathbf H\rangle_{v;\omega^{-1}}
=
\prod_{\mu=1}^{d_v}
\left\langle
\mathbf g^{(\mu)},\mathbf h^{(\mu)}
\right\rangle_{v^{(\mu)};\omega^{(\mu),-1}}.
\end{equation}
This factorization is used in the weighted Gram--Schmidt construction below.

\subsection{Weighted Gram--Schmidt Construction of $\mathcal P^\star$}
\label{app:discrete_LoMaC_A2}

For each velocity component $\mu=1,\dots,d_v$,
we perform a one-dimensional weighted Gram--Schmidt orthogonalization
on the set
\[
\big\{
\bm{\omega}^{(\mu),\star},\;
\mathbf{v}^{(\mu)} * \bm{\omega}^{(\mu),\star},\;
(\mathbf{v}^{(\mu)})^{* 2} * \bm{\omega}^{(\mu),\star}
\big\}
\subset \mathbb R^{N_{v^{(\mu)}}},
\]
with respect to
$\langle\cdot,\cdot\rangle_{v^{(\mu)};\omega^{(\mu),-1}}$. Define the one-dimensional quantities
\[
\gamma_\mu
=
\frac{
\big\langle
\mathbf{v}^{(\mu)} * \bm{\omega}^{(\mu),\star},
\;
\bm{\omega}^{(\mu),\star}
\big\rangle_{v^{(\mu)};\bm{\omega}^{(\mu),-1}}
}{
\big\langle
\bm{\omega}^{(\mu),\star},
\;
\bm{\omega}^{(\mu),\star}
\big\rangle_{v^{(\mu)};\bm{\omega}^{(\mu),-1}}
},
\]
\[
\mathbf{q}_1^{(\mu)}
=
\mathbf{v}^{(\mu)}*\bm{\omega}^{(\mu),\star} - \gamma_\mu \bm{\omega}^{(\mu),\star},
\]
and
\[
\eta_\mu
=
\frac{
\big\langle
(\mathbf{v}^{(\mu)})^{* 2}*\bm{\omega}^{(\mu),\star},
\;
\bm{\omega}^{(\mu),\star}
\big\rangle_{v^{(\mu)};\bm{\omega}^{(\mu),-1}}
}{
\big\langle
\bm{\omega}^{(\mu),\star},
\;
\bm{\omega}^{(\mu),\star}
\big\rangle_{v^{(\mu)};\bm{\omega}^{(\mu),-1}}
},
\]
\[
\beta_\mu
=
\frac{
\big\langle
(\mathbf{v}^{(\mu)})^{* 2}*\bm{\omega}^{(\mu),\star},
\;
\mathbf{q}_1^{(\mu)}
\big\rangle_{v^{(\mu)};\bm{\omega}^{(\mu),-1}}
}{
\big\langle
\mathbf{q}_1^{(\mu)},
\;
\mathbf{q}_1^{(\mu)}\big\rangle_{v^{(\mu)};\bm{\omega}^{(\mu),-1}}}
.
\]
Then the one-dimensional orthogonalized quadratic factor is
\[
\mathbf{q}_2^{(\mu)}
=
(\mathbf{v}^{(\mu)})^{* 2}*\bm{\omega}^{(\mu),\star}
-
\eta_\mu \bm{\omega}^{(\mu),\star}
-
\beta_\mu \mathbf{q}_1^{(\mu)}.
\]
By construction,
\(
\bm{\omega}^{(\mu),\star}\),
\(\mathbf{q}_1^{(\mu)}\), and
\(\mathbf{q}_2^{(\mu)}
\)
are mutually orthogonal under
$\langle\cdot,\cdot\rangle_{v^{(\mu)};\omega^{(\mu),-1}}$.


We now lift the one-dimensional weighted Gram--Schmidt construction
to $d_v$ dimensions.
Owing to the tensor-product structure of the multidimensional
velocity inner product
\eqref{eq:inner_velocity}
and its factorization property
\eqref{eq:inner_prod_factorization},
the orthogonality in each coordinate direction
extends to the full velocity space.
Accordingly, we define the following (non-normalized)
tensor-product tensors in
\(
\mathbb R^{N_{v^{(1)}}\times\cdots\times N_{v^{(d_v)}}}
\):
\[
\mathbf Q_0
=
\mathbf\Omega^\star,
\]
\[
\mathbf Q_{1,\mu}
=
\bm{\omega}^{(1),\star}
\otimes \cdots \otimes
\mathbf q_1^{(\mu)}
\otimes \cdots \otimes
\bm{\omega}^{(d_v),\star},
\]
\[
\mathbf Q_2
=
\sum_{\mu=1}^{d_v}
\bm{\omega}^{(1),\star}
\otimes \cdots \otimes
\mathbf q_2^{(\mu)}
\otimes \cdots \otimes
\bm{\omega}^{(d_v),\star}.
\]
Using the factorization property
\eqref{eq:inner_prod_factorization},
one verifies that
\(
\mathbf Q_0,
\{\mathbf Q_{1,\mu}\}_{\mu=1}^{d_v}\), and
\(\mathbf Q_2
\)
are mutually orthogonal
with respect to
\(
\langle\cdot,\cdot\rangle_{v;\omega^{-1}}
\),
and they span
\(
\mathcal N(\mathbf\Omega^\star)
\). Moreover, their squared norms admit explicit tensor-product
representations.
In particular,
\[
\|\mathbf Q_0\|_{v;\omega^{-1}}^2
=
\prod_{\mu=1}^{d_v}
\|\bm{\omega}^{(\mu),\star}\|_{v^{(\mu)};\omega^{(\mu),-1}}^2,
\]
\[
\|\mathbf Q_{1,\mu}\|_{v;\omega^{-1}}^2
=
\|\mathbf q_1^{(\mu)}\|_{v^{(\mu)};\omega^{(\mu),-1}}^2
\prod_{\nu\ne\mu}
\|\bm{\omega}^{(\nu),\star}\|_{v^{(\nu)};\omega^{(\nu),-1}}^2,
\]
while
\[
\|\mathbf Q_2\|_{v;\omega^{-1}}^2
=
\sum_{\mu=1}^{d_v}
\left(\|\mathbf q_2^{(\mu)}\|_{v^{(\mu)};\omega^{(\mu),-1}}^2
\prod_{\nu\ne\mu}
\|\bm{\omega}^{(\nu),\star}\|_{v^{(\nu)};\omega^{(\nu),-1}}^2\right),
\]
where the cross terms vanish due to orthogonality.
These explicit norm formulas will be used
in the construction of the projection operator below.


We next derive a closed-form expression of the weighted orthogonal
projection
\(
\mathcal{P}^\star
\)
with respect to the global phase-space inner product
\eqref{eq:inner_global}. For any discrete phase-space tensor
\(
\mathbf F,
\)
we define the corresponding (raw) velocity moments
\[
\mathbf M_0,\, \mathbf M_{1,\mu},\, \mathbf M_2
\in
\mathbb R^{N_{x^{(1)}}\times\cdots\times N_{x^{(d_x)}}}, \qquad \mu=1,\ldots,d_v,
\]
componentwise by
\begin{align}
\big(\mathbf M_0(\mathbf F)\big)_i
&=
\sum_k F_{i,k}\, w_k^{v}
=
\langle
\mathbf F_{i,\cdot},
\mathbf\Omega^\star
\rangle_{v;\omega^{-1}},
\label{eq:disc_m0}
\\
\big(\mathbf M_{1,\mu}(\mathbf F)\big)_i
&=
\sum_k (V_\mu)_k\, F_{i,k}\, w_k^{v}
=
\langle
\mathbf F_{i,\cdot},
\mathbf V_\mu*\mathbf\Omega^\star
\rangle_{v;\omega^{-1}},
\qquad \mu=1,\ldots,d_v,
\label{eq:disc_m1}
\\
\big(\mathbf M_2(\mathbf F)\big)_i
&=
\sum_k (V^{* 2})_k\, F_{i,k}\, w_k^{v}
=
\langle
\mathbf F_{i,\cdot},
\mathbf V^{* 2}*\mathbf\Omega^\star
\rangle_{v;\omega^{-1}}.
\label{eq:disc_m2}
\end{align}

Since
\(
\mathcal P^\star
=
\mathcal I_{R^{N_{x^{(1)}}\times\cdots\times N_{x^{(d_x)}}}}
\otimes
\mathcal P_{\mathcal N(\mathbf\Omega^\star)},
\)
the projection acts independently on each spatial multi-index.
Using the orthogonal basis
\(
\mathbf Q_0,
\{\mathbf Q_{1,\mu}\}_{\mu=1}^{d_v},
\mathbf Q_2
\)
of
\(
\mathcal N(\mathbf\Omega^\star),
\)
we obtain the explicit representation
\begin{equation}
\label{eq:phase_proj_closed}
\mathcal P^\star \mathbf F
=
\mathbf C_0(\mathbf F)\otimes \mathbf Q_0
+
\sum_{\mu=1}^{d_v}
\mathbf C_{1,\mu}(\mathbf F)\otimes \mathbf Q_{1,\mu}
+
\mathbf C_2(\mathbf F)\otimes \mathbf Q_2,
\end{equation}
where
\[
\mathbf C_0, \,
\mathbf C_{1,\mu}, \,
\mathbf C_2
\in
\mathbb R^{N_{x^{(1)}}\times\cdots\times N_{x^{(d_x)}}}
\]
are spatial coefficients which depend on the tensor $\mathbf{F}$.
Here the tensor product
\(
\otimes
\)
denotes the tensor product between a spatial tensor
and a velocity tensor, yielding a phase-space tensor. The operators $\mathbf C_0,\mathbf C_{1,\mu},\mathbf C_2$ are explicit and are evaluated solely from the discrete velocity moments defined above and one-dimensional inner products, as detailed below. 

Since
\(
\mathbf Q_0\),
\(\{\mathbf Q_{1,\mu}\}_{\mu=1}^{d_v}\), and
\(\mathbf Q_2\)
form an orthogonal basis of
\(
\mathcal N(\mathbf\Omega^\star)
\)
under
\(
\langle\cdot,\cdot\rangle_{v;\omega^{-1}},
\)
the velocity projection admits the standard orthogonal expansion formula.
In particular, for any velocity tensor
\(
\mathbf G\in\mathbb R^{N_{v^{(1)}}\times\cdots\times N_{v^{(d_v)}}},
\)
\[
\mathcal P_{\mathcal N(\mathbf\Omega^\star)}\mathbf G
=
\frac{
\langle \mathbf G,\mathbf Q_0\rangle_{v;\omega^{-1}}
}{
\|\mathbf Q_0\|_{v;\omega^{-1}}^2
}
\mathbf Q_0 + 
\sum_{\mu=1}^{d_v}\frac{
\langle \mathbf G,\mathbf Q_{1,\mu}\rangle_{v;\omega^{-1}}
}{
\|\mathbf Q_{1,\mu}\|_{v;\omega^{-1}}^2
}
\mathbf Q_{1,\mu}+
\frac{
\langle \mathbf G,\mathbf Q_2\rangle_{v;\omega^{-1}}
}{
\|\mathbf Q_2\|_{v;\omega^{-1}}^2
}
\mathbf Q_2.
\]
Applying this expansion to each fixed spatial multi-index $i$,
that is, to the velocity slice
\(
\mathbf F_{i,\cdot},
\)
and collecting the resulting coefficients over all spatial indices,
we obtain the spatial coefficients
\begin{align}
\mathbf{C}_{0}(\mathbf F)
&=
\frac{\mathbf M_0(\mathbf F)}
{\|\mathbf Q_0\|_{v;\omega^{-1}}^2},
\label{eq:C0_def}
\\
\mathbf{C}_{1,\mu}(\mathbf F)
&=
\frac{
\mathbf M_{1,\mu}(\mathbf F)
-
\gamma_\mu\,\mathbf M_0(\mathbf F)
}{
\|\mathbf Q_{1,\mu}\|_{v;\omega^{-1}}^2},
\qquad
\mu=1,\ldots,d_v,
\label{eq:C1_def}
\\
\mathbf{C}_{2}(\mathbf F)
&=
\frac{
\mathbf M_2(\mathbf F)
-
\Big(\sum_{\mu=1}^{d_v}\eta_\mu\Big)\mathbf M_0(\mathbf F)
-
\sum_{\mu=1}^{d_v}
\beta_\mu
\big(
\mathbf M_{1,\mu}(\mathbf F)
-
\gamma_\mu \mathbf M_0(\mathbf F)
\big)
}{
\|\mathbf Q_2\|_{v;\omega^{-1}}^2}.
\label{eq:C2_def}
\end{align}

\subsection{Reconstruction and Energy Consistency}

We now define the reconstruction operator that realizes the phase-space projection
$\mathcal{P}^\star$ using the conserved moment vector $\mathbf U$. The projection coefficients
in \eqref{eq:C0_def}--\eqref{eq:C2_def} are written in terms of the raw velocity moments
\eqref{eq:disc_m0}--\eqref{eq:disc_m2}. Therefore, the first step is to convert the conserved quantities into raw moments. This is accomplished using the discrete form of the identities
\begin{equation}
\label{eq:raw_from_conserved_supp}
    M_0 = U_0, 
    \qquad 
    M_{1,\mu} = \frac{1}{m_e} U_{1,\mu},
    \qquad
    M_2 = \frac{2}{m_e}\left(U_2 - e_p\right).
\end{equation}
The last relation assumes that the electric field energy density $e_p$ is computed
from the solution of Poisson's equation using the charge density determined by $U_0$.
Equivalently, for a fixed discrete Poisson solver, the discrete map
\[
    \mathbf U_0 \longmapsto \boldsymbol \rho \longmapsto \mathbf E \longmapsto \mathbf e_p
\]
is well-defined. Hence the conserved variables uniquely determine the corresponding
raw moments.

Given $\mathbf U$, we first compute its raw moments from a discrete analogue of the identity \eqref{eq:raw_from_conserved_supp}
and then substitute these quantities into the coefficient formulas
\eqref{eq:C0_def}--\eqref{eq:C2_def}. We denote the resulting spatial coefficients by $\widetilde{\mathbf{C}}_{0}(\mathbf{U})$, $\widetilde{\mathbf{C}}_{1,\mu}(\mathbf{U})$, and
$\widetilde{\mathbf{C}}_{2}(\mathbf{U})$. The reconstruction operator is then defined by
\begin{equation}
\label{eq:reconstruction_supp}
    \mathcal{R}(\mathbf{U})
    =
    \widetilde{\mathbf{C}}_{0}(\mathbf{U}) \otimes \mathbf{Q}_{0}
    +
    \sum_{\mu=1}^{d_v}
    \widetilde{\mathbf{C}}_{1,\mu}(\mathbf{U}) \otimes \mathbf{Q}_{1,\mu}
    +
    \widetilde{\mathbf{C}}_{2}(\mathbf{U}) \otimes \mathbf{Q}_{2}.
\end{equation}
Here each tensor product is taken between a spatial tensor and a velocity tensor,
which results in a phase-space tensor.

This construction is consistent with the projection derived in the previous section.
Indeed, if $\mathbf{U}(\mathbf{F})$ denotes the conserved moment vector associated with a phase-space tensor $\mathbf{F}$, then converting $\mathbf{\mathbf{U}}$ to raw moments using
\eqref{eq:raw_from_conserved_supp} recovers precisely the raw moments
$\mathbf{M}_{0}(\mathbf{F})$, $\mathbf{M}_{1,\mu}(\mathbf{F})$, and $\mathbf{M}_{2}(\mathbf{F})$. Therefore the coefficients in
\eqref{eq:reconstruction_supp} agree with those in the projection formula
\eqref{eq:phase_proj_closed}, and hence
\begin{equation*}
    \mathcal{R}(\mathbf{U}(\mathbf{F})) = \mathcal{P}^\star \mathbf{F}.
\end{equation*}

Moreover, $\mathcal{P}^\star \mathbf{F}$ preserves the raw moments of $\mathbf{F}$. This follows directly from
the construction of $\mathcal{P}^\star$ as the weighted orthogonal projection onto
$\mathcal N(\mathbf\Omega^\star)$, whose basis was chosen to represent the zeroth, first,
and second raw velocity moments. Therefore,
\[
    \mathbf{M}_{0}(\mathcal{P}^\star \mathbf{F}) = \mathbf{M}_{0}(\mathbf{F}), 
    \qquad
    \mathbf{M}_{1,\mu}(\mathcal{P}^\star \mathbf{F}) = \mathbf{M}_{1,\mu}(\mathbf{F}),
    \qquad
    \mathbf{M}_{2}(\mathcal{P}^\star \mathbf{F}) = \mathbf{M}_{2}(\mathbf{F}).
\]
Applying the identity \eqref{eq:raw_from_conserved_supp} in reverse shows that
\begin{equation*}
    \mathbf{U} (\mathcal{P}^\star \mathbf{F}) = \mathbf{U}(\mathbf{F}).
\end{equation*}
Thus the reconstruction preserves the mass, momentum, and total energy, provided that the electric field energy is evaluated with the same discrete Poisson
solver used in \eqref{eq:raw_from_conserved_supp}.